\documentclass[twoside,11pt]{article}
\usepackage{geometry}
\usepackage{graphicx}
\usepackage{indentfirst}                                % make the first line of all sections be indented
\usepackage[colorlinks]{hyperref}
\hypersetup{linkcolor=blue,filecolor=black,urlcolor=blue, citecolor=black}   % book-mark black
\usepackage[numbers,sort&compress]{natbib}         % references

\newcommand\orcidicon[1]{\href{https://orcid.org/#1}{\includegraphics[scale=0.02]{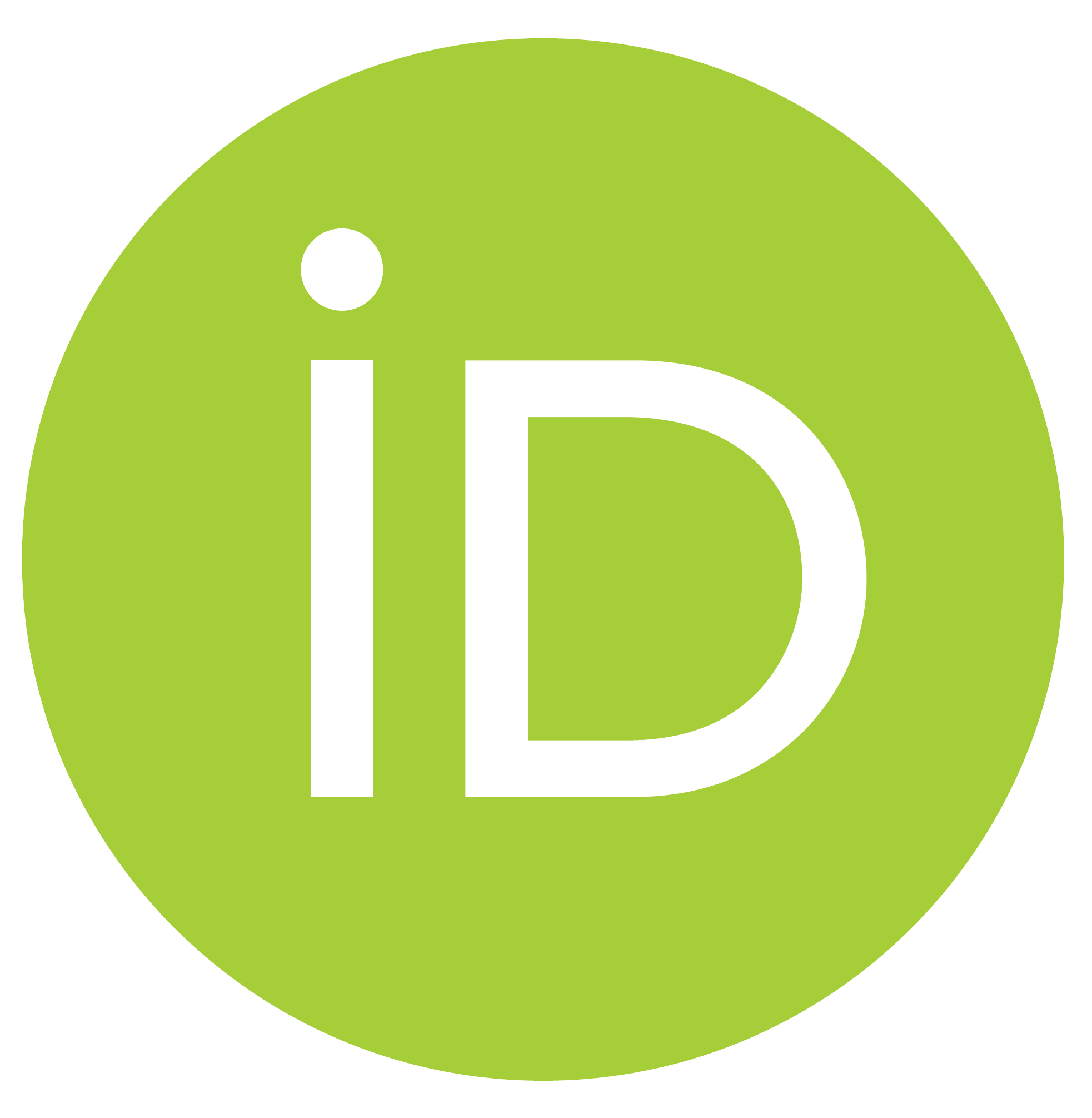}}}

\usepackage[titletoc,title]{appendix}
\usepackage{titlesec}
\titleformat{\section}{\large\bfseries\centering}{\thesection}{1em}{}
\titleformat{\subsection}{\centering\bfseries}{\thesubsection}{0.5em}{}
\titleformat{\subsubsection}[runin]{\bfseries}{\thesubsubsection.}{0.4em}{}[.]

\usepackage{amsmath,amsfonts,amssymb,bm,mathrsfs,amsthm,amsxtra}
\numberwithin{equation}{section}
\newtheorem{lemma}{Lemma}[section]
\newtheorem{proposition}[lemma]{Proposition}
\newtheorem{theorem}{Theorem}[section]
\newtheorem{corollary}[lemma]{Corollary}
\newtheorem{definition}{Definition}[section]
\newtheorem{remark}{Remark}[section]

\usepackage{ulem}%%%% bar the text
\usepackage{amsmath,amsfonts,amssymb,bm,mathrsfs,amsthm,amsxtra}
\newcommand{\VERT}{\vert\kern-0.3ex\vert\kern-0.3ex\vert}
\def\d{\,\mathrm{d}}
\def\p{\partial}

\usepackage{accents}
\makeatletter
\def\widebar{\accentset{{\cc@style\underline{\mskip10mu}}}}
\makeatother

\newcommand\w[1]{\makebox[1em]{$#1$}}

\begin{document}
%%%%%----------
%%%%%---------- Full Title
%%%%%----------
\title{\bf Nonlinear Stability of MHD Contact Discontinuities with Surface Tension
\let\thefootnote\relax\footnotetext{
The research of {\sc Yuri Trakhinin} was supported by the Russian Science Foundation under Grant No. 20-11-20036.
The research of {\sc Tao Wang} was supported by the Fundamental Research Funds for the Central Universities and the National Natural Science Foundation of China under Grants 11971359 and 11731008.
}
}

%%%%%----------
%%%%%---------- Full Author
%%%%%----------
\author{
{\sc Yuri Trakhinin}\orcidicon{0000-0001-8827-2630}\thanks{e-mail: trakhin@math.nsc.ru}\\
{\footnotesize Sobolev Institute of Mathematics, Koptyug av.~4, 630090 Novosibirsk, Russia}\\
{\footnotesize Novosibirsk State University, Pirogova str.~1, 630090 Novosibirsk, Russia}\\[2mm]
{\sc Tao Wang}\orcidicon{0000-0003-4977-8465}\thanks{e-mail: tao.wang@whu.edu.cn}\\
{\footnotesize School of Mathematics and Statistics, Wuhan University, Wuhan 430072, China}\\ 
{\footnotesize Hubei Key Laboratory of Computational Science, Wuhan University, Wuhan 430072, China} %$
}

\date{\empty}
\maketitle

%%%%%----------
%%%%%---------- Abstract
%%%%%----------
\begin{abstract}

We consider the motion of two inviscid, compressible, and electrically conducting fluids separated by an interface across which there is no fluid flow in the presence of surface tension. The magnetic field is supposed to be nowhere tangential to the interface. This leads to the characteristic free boundary problem for contact discontinuities with surface tension in three-dimensional ideal compressible magnetohydrodynamics (MHD). We prove the nonlinear structural stability of MHD contact discontinuities with surface tension in Sobolev spaces by a modified Nash--Moser iteration scheme. The main ingredient of our proof is deriving the resolution and tame estimate of the linearized problem in usual Sobolev spaces of sufficiently large regularity. In particular, for solving the linearized problem, we introduce a suitable regularization that preserves the transport-type structure for the linearized entropy and divergence of the magnetic field.

\end{abstract}

 \vspace*{2mm}
{\small
%%%%%----------
%%%%%---------- Keywords
%%%%%----------

\vspace*{2mm}

%%%%%----------
%%%%%---------- MSC2020
%%%%%----------
%%(2000-now) Interfacial stability and instability in hydrodynamic stability

\tableofcontents
}

\bigskip

%%%%%----------
%%%%%---------- Introduction
%%%%%----------
\section{Introduction}\label{sec:intro}

We are concerned with the evolution of a smooth interface $\Sigma(t)$ between two inviscid, compressible, and electrically conducting fluids that occupy the domains $\Omega^+(t)$ and $\Omega^-(t)$ in $\mathbb{R}^3$ at time $t\geq 0$.
The fluid motion is described by the following equations of ideal compressible magnetohydrodynamics (MHD) (see {\sc Landau--Lifshitz} \cite[\S 65]{LL84MR766230}):
\begin{subequations} \label{MHD1}
\begin{alignat}{3}
&    \p_t \rho +\nabla\cdot (\rho v )=0
 && \textrm{in } \Omega^{\pm}(t),
\label{MHD1a} \\
&\p_t (\rho v )+\nabla\cdot (\rho v \otimes v -H \otimes H )+\nabla q =0
 && \textrm{in } \Omega^{\pm}(t),
  \label{MHD1b} \\
&\p_t H-\nabla\times(v\times H)=0
 && \textrm{in } \Omega^{\pm}(t),
 \label{MHD1c} \\
&\p_t(\rho E+\tfrac{1}{2}|H|^2) +\nabla\cdot ( v  (\rho E+p)+H\times(v\times H) )=0
 &\quad & \textrm{in } \Omega^{\pm}(t), \label{MHD1d}
\end{alignat}
\end{subequations}
together with the divergence constraint
\begin{align} \label{divH1}
\nabla\cdot H=0
\quad \textrm{in } \Omega^{\pm}(t).
\end{align}
Here the density $\rho$, fluid velocity $v\in\mathbb{R}^3$, magnetic field $H\in\mathbb{R}^3$, and pressure $p$ are unknown functions of the time $t$ and spatial variable $x=(x_1,x_2,x_3)$.
We denote by $q=p+\frac{1}{2}|H|^2$ the total pressure and by $E=e+\frac{1}{2}|v|^2$ the specific total energy, where $e$ is the specific internal energy.
The thermodynamic variables $\rho$, $p$, and $e$ are related to the specific entropy $S$ and the absolute temperature $\vartheta>0$ by the Gibbs relation
\begin{align} \nonumber %\label{Gibbs}
\vartheta \d S=\d e+p\d \left(\frac1\rho\right).
\end{align}
The constitutive relations $\rho=\rho(p, S)$ and $e=e(p, S)$ render the system of conservation laws \eqref{MHD1} closed.
Here and below, we denote by $\p_t$ the time derivative $\frac{\p}{\p t}$, by $\nabla$ the spatial gradient $(\p_1,\p_2,\p_3)^{\top}$ with $\p_i:=\frac{\p}{\p x_i}$,
and by $u\otimes w$ the tensor product of vectors $u,w\in\mathbb{R}^3$ with $(i,j)$-entry $u_i w_j$.

In the absence of surface tension, the assumption that there is no fluid flow across the moving interface
allows one to consider two distinct types of characteristic discontinuities in compressible MHD \cite[\S 71]{LL84MR766230}:
%{\it compressible current-vortex sheets} \textcolor{red}{(called tangential discontinuities in \cite{LL84MR766230})} 
{\it tangential discontinuities} (or called current-vortex sheets)
for which the magnetic field is parallel to the interface,
and {\it contact discontinuities} for which the magnetic field intersects the interface.

Without magnetic fields, compressible current-vortex sheets are reduced to compressible vortex sheets for the Euler equations in gas dynamics. 
{\sc Syrovatski\u{\i}} \cite{Syr54MR0065343} and {\sc Fejer--Miles} \cite{FM63MR0154509} showed by normal modes analysis that every compressible vortex sheet in three dimensions is linearly unstable.
This linear instability is the analogue of the Kelvin--Helmholtz instability for incompressible fluids; see, {\it e.g.}, {\sc Chandrasekhar} \cite[Chapter 11]{C61MR0128226}.
The linear and nonlinear stability of compressible current-vortex sheets in three-dimensional MHD was established independently by {\sc Trakhinin} \cite{T05MR2187618,T09ARMAMR2481071} and {\sc Chen--Wang} \cite{CW08MR2372810,CW12MR3289359} under some stability condition.
The results of \cite{T05MR2187618,T09ARMAMR2481071,CW08MR2372810,CW12MR3289359}
indicate that {\it non-paralleled} magnetic fields can stabilize the motion of three-dimensional compressible vortex sheets.

Regarding MHD contact discontinuities, {\sc  Morando et al.~}\cite{MTT15MR3306348,MTT18MR3766987} recently obtained the local-in-time existence of solutions for two-dimensional polytropic fluids 
provided the Rayleigh--Taylor sign condition on the jump of the normal derivative of the pressure holds at each point of the discontinuity front.
%We would expect that the Rayleigh--Taylor sign condition provides the existence of MHD contact discontinuities also for the general three-dimensional case. 
We would expect that the Rayleigh--Taylor sign condition implies the existence of MHD contact discontinuities also for the general three-dimensional case. 
However, it remains open to confirm this expectation rigorously.
Remark here that the approach in \cite{MTT15MR3306348} for deriving the basic energy estimate for the linearized problem cannot be directly applied to the three-dimensional case due to the appearance of additional boundary terms in energy integrals (see \cite[\S 6]{MTT15MR3306348} for more details).

Surface tension has been proved to suppress the instability of vortex sheets in three dimensions 
by {\sc Ambrose--Masmoudi} \cite{AM07MR2334849} for incompressible irrotational flows, 
by {\sc Cheng et al.}~\cite{CCS08MR2456184} and {\sc Shatah--Zeng} \cite{SZ11MR2763036,SZ08MR2400608} for incompressible rotational flows, 
and by {\sc Stevens} \cite{Stevens16MR3544315} for compressible flows.
Numerical and experimental studies of free-interface MHD flows with surface tension have been provided in {\sc Samulyak et al.}~\cite{SDGX07MR2356385} and the references therein.
However, to the best of our knowledge,
there is no result currently available for the nonlinear fluid--fluid interface problem with surface tension in ideal compressible MHD.
%We consider here three-dimensional {\it MHD contact discontinuities with surface tension}, in the hope that surface tension will provide a stabilizing effect on the dynamics of the moving interface.
The purpose of this paper is to examine the stabilizing effect of surface tension on the dynamics of free interfaces for ideal compressible conducting fluids, 
or more precisely, to establish the nonlinear structural stability of {\it MHD contact discontinuities with surface tension} in three dimensions without assuming the fulfillment of the Rayleigh--Taylor sign condition.

For MHD contact discontinuities with surface tension,
there is no flow across the interface $\Sigma(t)$
and the magnetic field is nowhere tangent to $\Sigma(t)$.
Let $\bm{n}$ and $\mathcal{V}$ denote the unit normal vector pointing into $\Omega^+(t)$ and the normal speed of $\Sigma(t)$, respectively.
Taking into account the surface tension force on $\Sigma(t)$ gives rise to the following boundary conditions:
\begin{align} \label{BC1}
	H^{\pm}\cdot \bm{n}\neq 0,\quad
	[p]=\mathfrak{s}\mathcal{H},\quad
	[H]=0,\quad
	[v]=0, \quad
	\mathcal{V}=v^{+}\cdot \bm{n}\quad
	\textrm{on }\Sigma(t),
\end{align}
where $\mathfrak{s}>0$ is the constant coefficient of surface tension,
$\mathcal{H}$ is twice the mean curvature of $\Sigma(t)$,
given any function $g$ we denote
\[
g^{\pm}(t,x):=\lim_{\epsilon\to 0^+}g(t,x\pm \epsilon \bm{n}(t,x))
\quad \textrm{for } x\in\Sigma(t),
\]
and the bracket $[\,\cdot\,]$ stands for the jump of the enclosed quantity across the interface, 
that is,
$[g](t,x):=g^+(t,x)-g^-(t,x)$ at any point $x\in\Sigma(t)$.
A derivation of the boundary conditions \eqref{BC1} can be found in Appendix \ref{App:A}.
We remark that the second condition in \eqref{BC1} is the same as the Young--Laplace law for the pressure discontinuity across static interfaces due to the presence of surface tension 
(see {\sc Lautrup} \cite[\S 5.3]{L11zbMATH05278418}).

The problem \eqref{MHD1}--\eqref{BC1} is a nonlinear hyperbolic problem with the free boundary $\Sigma(t)$ being characteristic thanks to the last two conditions in \eqref{BC1}.
We consider here nonisentropic fluids under the physical assumption that the sound speed is positive,
so that the equations \eqref{MHD1} can become symmetric hyperbolic for smooth solutions.
It is worth mentioning that our constitutive relations are very general and include the polytropic case studied in \cite{MTT15MR3306348,MTT18MR3766987} as a special example.
Moreover, we assume that the interface $\Sigma(t)$ has the form of a graph, allowing us to reformulate the nonlinear problem \eqref{MHD1}--\eqref{BC1} to that in a fixed domain by a simple lift of the graph.

For the linearized problem around a certain basic state,
we construct the unique solution in the usual Sobolev space $H^1$ via the duality argument.
To this end, we show the $H^1$ {\it a priori} estimate for the linearized problem and introduce a suitable $\varepsilon$--regularization that admits a unique solution satisfying a uniform-in-$\varepsilon$ energy estimate in $H^1$.
More precisely,
we first deduce the $L^2$ estimates of solutions and their tangential derivatives
by making full use of the improved spatial regularity for the interface due to surface tension.
For general hyperbolic problems with characteristic boundary, energy estimates exhibit a loss of control of normal derivatives and it is natural to work in the anisotropic weighted Sobolev spaces (see {\sc Chen} \cite{C10MR2827870} and {\sc Secchi} \cite{S96MR1405665}).
Nevertheless, as in \cite{MTT15MR3306348,MTT18MR3766987}, we manage to compensate the missing normal derivatives through the transport equations for the linearized entropy and divergence of the magnetic field.
But since our basic {\it a priori} estimate is closed in $H^1$ rather than in $L^2$,
the duality argument cannot be employed directly for solving the linearized problem.
To overcome this difficulty, we introduce a carefully chosen $\varepsilon$--regularization
that preserves the transport-type structure for the linearized entropy and divergence of the magnetic field.
Given any fixed and sufficiently small parameter $\varepsilon>0$, we can close the $\varepsilon$--dependent $L^2$ {\it a priori} estimate for both the $\varepsilon$--regularization and its dual problem.
This enables us to construct solutions of the regularized problem in $L^2$ by the duality argument.
Then we build an energy estimate in $H^1$ {\it uniformly in $\varepsilon$} for the regularization in order to solve the linearized problem by passing to the limit $\varepsilon\to 0$.

For the linearized problem,
we also prove the existence and uniqueness of solutions in the Sobolev spaces $H^m$ with $m\geq3$ based on the resolution in $H^1$ and a high-order {\it a priori} energy estimate.
The high-order energy estimate, which follows by using the Moser-type calculus inequalities,  is a so-called {\it tame estimate}, since the loss of derivatives from the basic state to the solution is {\it fixed}.
Finally we establish the local-in-time existence of solutions to the nonlinear problem through an appropriate iteration scheme of Nash--Moser type developed by {\sc H\"{o}rmander} \cite{H76MR0602181} and {\sc Coulombel--Secchi} \cite{CS08MR2423311}. 
We refer to {\sc Alinhac--G{{\'e}}rard} \cite{AG07MR2304160} and {\sc Secchi} \cite{S16MR3524197} for a general description of the Nash--Moser method.

The rest of this paper is organized as follows.
In \S \ref{sec:non}, we first introduce the free boundary problem and an equivalent reformulation in a fixed domain for MHD contact discontinuities with surface tension.
Then we state the main result of this paper, namely Theorem \ref{thm:main},
and present the notation and Moser-type calculus inequalities.
In \S \ref{sec:lin1}, after linearizing the problem around a certain basic state,
we prove the existence and uniqueness of the effective linear problem in the usual Sobolev space $H^1$.
Section \ref{sec:tame} deals with the tame estimate for the effective linear problem in the usual Sobolev spaces $H^m$ with $m\geq 3$.
In \S \ref{sec:NM}, we combine the linear results in \S\S \ref{sec:lin1}--\ref{sec:tame} with a suitable modified Nash--Moser iteration scheme
to conclude the proof of the nonlinear stability of MHD contact discontinuities with surface tension.
Appendix \ref{App:A} provides the jump conditions for free-interface ideal compressible MHD with or without surface tension.

\section{Nonlinear Problems and Main Result} \label{sec:non}

In this section we first introduce the free boundary problem for MHD contact discontinuities with surface tension and an equivalent reformulation in a fixed domain.
Then we state the main result of this paper, namely Theorem \ref{thm:main}.
We also present the notation and Moser-type calculus inequalities for later use.

\subsection{Free Boundary Problem}

We assume that the interface $\Sigma(t)$ has the form of a graph:
\begin{align*}
	\Sigma(t):=\{x\in\mathbb{R}^3:x_1=\varphi(t,x')\} \quad \textrm{with \ } x'=(x_2,x_3),
\end{align*}
where the interface function $\varphi$ is to be determined.
Our main problem is to construct \textit{MHD contact discontinuities with surface tension}, that is,
smooth solutions $U^{\pm}:=(p^{\pm},v^{\pm},H^{\pm},S^{\pm})^{\top}$ 
of the equations \eqref{MHD1}--\eqref{divH1} in $\Omega^{\pm}(t):=\{x\in\mathbb{R}^3: x_1\gtrless \varphi(t,x') \}$ satisfying the boundary conditions \eqref{BC1}.
Then 
\begin{align} \label{N:def}
	\bm{n}=\frac{N}{|N|}\qquad
	\textrm{for } N:=\begin{pmatrix}
		1\\ -\mathrm{D}_{x'}\varphi
	\end{pmatrix}
	\  \textrm{ with \ }
	\mathrm{D}_{x'}:=\begin{pmatrix}
		\p_2\\ \p_3
	\end{pmatrix},
\end{align}
which implies
\begin{align}
	\label{H.cal:def}
	\mathcal{V}=\frac{\p_t\varphi}{|N|},\quad
	\mathcal{H}=\mathcal{H}(\varphi):= \mathrm{D}_{x'}\cdot\left( \frac{\mathrm{D}_{x'}\varphi }{\sqrt{1+|\mathrm{D}_{x'}\varphi|^2 }}\right).
\end{align}
Hence the boundary conditions \eqref{BC1} become
\begin{gather} \label{BC2a}
	H^{\pm}\cdot N\neq 0 \qquad \textrm{on }\Sigma(t),\\[0.5mm]
 \label{BC2}
	[p]=\mathfrak{s}\mathcal{H}(\varphi),\quad
	[H]=0,\quad[v]= 0, \quad  \p_t \varphi=v^+\cdot N\qquad \textrm{on }\Sigma(t).
\end{gather}
Clearly, there exist trivial contact-discontinuity solutions consisting of two constant states separated by a flat surface as follows:
\begin{align} \label{U.bar:def}
\widebar{U}(x):=
\left\{
\begin{aligned}
	&\widebar{U}^{+}
	:=(\bar{p},\bar{v},\widebar{H},\widebar{S}^{+})^{\top}\quad \mathrm{if}\  x_1>  0, \\
	&\widebar{U}^{-}
	:=(\bar{p},\bar{v},\widebar{H},\widebar{S}^{-})^{\top}\quad \mathrm{if}\  x_1<  0,
\end{aligned}
\right.
\end{align}
where we require that
$\bar{v}_1=0$, $\widebar{H}_1\neq 0$, and $\widebar{S}^+\neq \widebar{S}^-$ on account of
the conditions \eqref{BC2a}--\eqref{BC2}.

We consider very general, smooth constitutive relations $\rho^{\pm}=\rho^{\pm}(p,S)$ and $e^{\pm}=e^{\pm}(p,S)$ for the two fluid phases in $\Omega^{\pm}(t)$, respectively.
%All we need is that the sound speeds $a_{\pm}: =  p_{\rho}^{\pm}  (\rho, S)^{1/2}>0$ for all $\rho \in (\rho_*,\rho^*)$, 
We suppose that the sound speeds $a_{\pm}: =  p_{\rho}^{\pm}  (\rho, S)^{1/2}$ are positive for all $\rho \in (\rho_*,\rho^*)$, 
where $\rho_*$ and $\rho^*$ are some positive constants with $\rho_*<\rho^*$.
Then the equations \eqref{MHD1} are equivalent to the symmetric hyperbolic system
\begin{align}
	\label{MHD.vec}
	A^{\pm}_0 (U^{\pm})\p_t U^{\pm}+\sum_{i=1}^{3}A^{\pm}_i (U^{\pm})\p_i U^{\pm}=0
	\quad \textrm{in }\Omega^{\pm}(t),
\end{align}
for smooth solutions $U^{\pm}$ satisfying the hyperbolicity condition
\begin{align}
	\label{hyperbolicity}
	\rho_*<\rho^{\pm}(p^{\pm},S^{\pm})<\rho^*,
\end{align}
where
\setlength{\arraycolsep}{4pt}
\begin{align}
\label{A0:def}
&A^{\pm}_0 (U):=\mathrm{diag}\Big(\frac{1}{\rho^{\pm} a_{\pm}^2},\, \rho^{\pm},\, \rho^{\pm},\, \rho^{\pm},\, 1,\, 1,\, 1,\, 1\Big),
\\[1mm]
\label{Ai:def}
&A^{\pm}_i (U):=
\begin{pmatrix}
\dfrac{v_i}{\rho^{\pm} a_{\pm}^2} & \bm{e}_i^{\top} & 0 & 0\\[2mm]
\bm{e}_i & \rho^{\pm} v_i {I}_3 & \bm{e}_i\otimes H -H_i {I}_3 & 0\\[1.5mm]
0& H\otimes\bm{e}_i-H_i {I}_3 & v_i{I}_3  & 0\\[1.5mm]
\w{0}  & \w{0}  & \w{0} & \w{v_i}
\end{pmatrix},
\end{align}
for $U:=(p,v,H,S)^{\top}$ and $i=1,2,3.$
Throughout this paper,  we denote the identity matrix of order $m$ by $I_m$ and the standard basis of $\mathbb{R}^3$  by  $\{\bm{e}_1:=(1,0,0)^{\top}$, $\bm{e}_2:=(0,1,0)^{\top}$, $\bm{e}_3:=(0,0,1)^{\top}\}$.

It follows from the last two conditions in \eqref{BC2} that
\begin{align*}
&\bigg(\p_t\varphi A^{\pm}_0 (U^{\pm})-\sum_{i=1}^3 N_{i}A^{\pm}_{i} (U^{\pm})\bigg)\bigg|_{\Sigma(t)}
\\[1mm]
&\quad
=\left.\begin{pmatrix}
0& -N^{\top} & 0 & 0\\
-N & O_3 & H^{\pm}\cdot N I_3-  N \otimes  H^{\pm} &0\\
0 & H^{\pm}\cdot N I_3- H^{\pm}\otimes N &O_3 &0\\
\w{0} &\w{0} &\w{0}  &\w{0}
\end{pmatrix}
\right|_{\Sigma(t)},
\end{align*}
where $N_i$ is the $i$--th component of the normal vector $N$ ({\it cf.}~\eqref{N:def}) and $O_m$ denotes the zero matrix of order $m$.
Taking into account the constraint \eqref{BC2a}, we calculate that the boundary matrix for our problem,
\begin{align*}
\begin{pmatrix}
\p_t\varphi A^+_0(U^{+})-\sum_{i=1}^3 N_{i}A^+_{i}(U^{+}) & 0\\
0&
-\p_t\varphi A^-_0(U^{-})+\sum_{i=1}^3 N_{i}A^-_{i}(U^{-})
\end{pmatrix},
\end{align*}
has six positive, six negative, and four zero eigenvalues on $\Sigma(t)$.
As a result, the free boundary $\Sigma(t)$ is characteristic, {\it i.e.}, the boundary matrix is singular.
Noting that one boundary condition is necessary for determining the interface function $\varphi$,
we know from the well-posedness theory for hyperbolic problems
that the correct number of the boundary conditions is seven.
Therefore, we have to take one of the boundary conditions \eqref{BC2} as an initial constraint rather than as a real boundary condition. 
It will turn out that the identity
\begin{align} \label{BC2b}
[H]\big|_{\Sigma(t)}\cdot N=0
\end{align}
can be regarded as a constraint on the initial data.
Then the boundary conditions for our problem should consist of \eqref{BC2a} and
\begin{align} \label{BC3}
[p]=\mathfrak{s}\mathcal{H}(\varphi), \quad [v]=0,\quad  [H]\cdot \tau_i=0,\quad
\p_t \varphi=v^+\cdot N\quad
\textrm{on }\Sigma(t),
\end{align}
for $i=1,2$,
where the vectors $\tau_1$, $\tau_2$ are defined by
\begin{align}\label{tau:def}
\tau_1:=(\p_2\varphi,1,0)^{\top},\quad \tau_2:=(\p_3\varphi,0,1)^{\top}.
\end{align}

\subsection{Reformulated Problem in a Fixed Domain}
Let us reformulate the free boundary problem for MHD contact discontinuities with surface tension
into an equivalent problem in a fixed domain.
For this purpose, we replace the unknowns $U^{\pm}$ by
\begin{align} \label{transform}
U^{\pm}_{\sharp}(t,x):=U^{\pm}(t,\varPhi^{\pm}(t,x), x'),
\end{align}
respectively, where
\begin{align} \label{varPhi:def}
\varPhi^{\pm}(t,x):=\pm x_1+ \chi(\pm x_1)\varphi(t,x')
\end{align}
with $\chi\in C^{\infty}_0(\mathbb{R})$ satisfying
$\chi\equiv 1$  on  $[-1,1]$  and $\|\chi'\|_{L^{\infty}(\mathbb{R})} <1.$
We will assume without loss of generality that
$\|\varphi_0\|_{L^{\infty}(\mathbb{R}^{2}) }\leq  \frac{1}{4}$, so that
the change of variables is admissible on sufficiently short time interval $[0,T]$.
Here we introduce the cut-off function $\chi$ as in \cite{M01MR1842775,T09CPAMMR2560044,T09ARMAMR2481071, TW21MR4201624,TW21b}
to avoid the assumption that the initial perturbations have compact support.

The nonlinear stability of MHD contact discontinuities with surface tension amounts to constructing smooth solutions $U^{\pm}_{\sharp}$ in the half-space $\Omega:=\{x\in\mathbb{R}^3:x_1>0\}$ of the following initial--boundary value problem:
\begin{subequations}   
	\label{MCD0}
\begin{alignat}{2}
\label{MCD0a}
&\mathbb{L}_{\pm}(U^{\pm},\varPhi^{\pm}) :=L_{\pm}(U^{\pm},\varPhi^{\pm})U^{\pm} =0
&  \qquad &\textrm{in }    \Omega,\\[0.5mm]
\label{MCD0b}
&\mathbb{B}(U^{+},U^{-},\varphi)
:=\begin{pmatrix}
[p]-\mathfrak{s}\mathcal{H}(\varphi)\\[0.5mm]
[v]\\[0.5mm]
[H]\cdot \tau_1\\[0.5mm]
[H]\cdot \tau_2\\[0.5mm]
\p_t \varphi-v^+\cdot N
\end{pmatrix}
=0
& \qquad  &\textrm{on }   \Sigma,\\[0.5mm]
\label{MCD0c}
&(U^{+},U^{-},\varphi)=(U^{+}_0,U^{-}_0,\varphi_0)
& \qquad &\textrm{if } t=0,
\end{alignat}
\end{subequations}
where we drop the subscript $``\sharp"$ for notational simplicity,
$\Sigma:=\{x\in \mathbb{R}^3:\, x_1=0\}$ denotes the boundary,
and
\begin{align}
\label{L:def}
&L_{\pm}(U,\varPhi):=A^{\pm}_0 (U)\partial_t+\widetilde{A}^{\pm}_1 (U,\varPhi)\partial_1+ A^{\pm}_2(U)\partial_2
+A^{\pm}_3(U)\partial_3
\end{align}
with
\begin{align}
\label{A1t:def}
&\widetilde{A}^{\pm}_1(U,\varPhi):=
\frac{1}{\partial_1\varPhi}\big(A^{\pm}_1(U)-\partial_t\varPhi A^{\pm}_0(U)-\partial_2\varPhi A^{\pm}_2(U)-\partial_3\varPhi A^{\pm}_3(U)\big).
\end{align}
Recall that the vectors $\tau_1,\tau_2$ and the matrices $A^{\pm}_0,\ldots,A^{\pm}_3$ are given in \eqref{tau:def} and \eqref{A0:def}--\eqref{Ai:def}, respectively.
According to \eqref{BC2a}, we assume that
\begin{align}
\label{BC3a}
&|H^{\pm}\cdot N|\geq  \kappa >0\qquad  \textrm{on }  \Sigma
\end{align}
for some positive constant $\kappa$.
In the new variables, the equation \eqref{divH1} and the jump condition \eqref{BC2b}
are reduced to
\begin{alignat}{3}
\label{inv1b}
& \nabla^{\varPhi^{\pm}}\cdot H^{\pm}=0\qquad && \textrm{in }  \Omega,\\
\label{inv2b}
&[H]\cdot N=0\qquad &&\textrm{on }  \Sigma,
\end{alignat}
where
\begin{align}
\nabla^{\varPhi}:=(\p_1^{\varPhi},\, \p_2^{\varPhi},\,\p_3^{\varPhi})^{\top}
\label{nabla.varPhi}
\end{align}
with
\begin{align} \label{differential}
\partial_t^{\varPhi}:=\partial_t-\frac{\partial_t\varPhi}{\partial_1\varPhi}\p_1,\ \
\partial_1^{\varPhi}:=\frac{1}{\partial_1\varPhi}\partial_1,\ \
\partial_i^{\varPhi}:=\partial_i-\frac{\partial_i\varPhi}{\partial_1\varPhi}\partial_1\ \
\textrm{for $i=2,3$. }
\end{align}
As in \cite[Appendix A]{T09ARMAMR2481071}, we can show that
the identities \eqref{inv1b}--\eqref{inv2b} hold for any $t>0$
provided they are satisfied at the initial time.

\subsection{Main Result}

We are now ready to state the main result of this paper.

\begin{theorem}
\label{thm:main}
Let $m\geq 12$ be an integer.
Suppose that the initial data \eqref{MCD0c}
satisfy the requirements \eqref{BC3a}--\eqref{inv2b} and the hyperbolicity condition
\begin{align} \label{thm:H1}
\rho_*<\inf_{\Omega}\rho^{\pm}({U}_0^{\pm})\leq\sup_{\Omega}\rho^{\pm}({U}_0^{\pm})<\rho^*.
\end{align}
Suppose further that
$(U_0^{\pm}-\widebar{U}^{\pm},\varphi_0)
$ belong to $ H^{m+3/2}(\Omega)\times H^{m+2}(\mathbb{R}^{2})$
for the constant states $\widebar{U}^{\pm}$ defined in \eqref{U.bar:def}
and the initial data are compatible up to order $m$ (see Definition \ref{def:1}).
Then there is a sufficiently small $T>0$, such that the problem \eqref{MCD0} has a unique solution $(U^{+}, U^{-},\varphi)$ on the time interval $[0,T]$ satisfying
\begin{align*}
U^{\pm}-\widebar{U}^{\pm}\in H^{m-6}([0,T]\times\Omega),\quad
(\varphi,\mathrm{D}_{x'}\varphi)\in H^{m-6}([0,T]\times\mathbb{R}^{2}).
\end{align*}
\end{theorem}
\begin{remark}
Since the relations $ \p_1\varPhi^{+}\geq \frac{1}{4}$ and $ \p_1\varPhi^{-}\leq - \frac{1}{4}$ hold in $[0,T]\times \Omega$ for $T>0$ sufficiently small,
we can obtain from Theorem \ref{thm:main}
a corresponding result for MHD contact discontinuities with surface tension in the original variables.	
\end{remark}
\begin{remark}
The proof of Theorem \ref{thm:main} is based on the tame energy estimate \eqref{tame:es} 
%for the linearized problem, which exhibits 
that exhibits a loss of two derivatives from the basic state to the solution.
It will be interesting to see whether the loss of regularity in Theorem \ref{thm:main} can be reduced through a direct nonlinear energy method, which has been employed by {\sc Stevens}  {\rm\cite{Stevens16MR3544315}} on compressible vortex sheets with surface tension.
\end{remark}

\subsection{Notation and Moser-type Calculus Inequalities}

Throughout this paper we adopt the following notation:
\begin{itemize}
 \item[(i)] We write the letter $C$ for some universal positive constant,
 and $C(\cdot)$ for some generic positive constant depending on the quantities listed in the parenthesis.
 The symbol $A\lesssim B$ means that $A \leq C B$.
 Given some parameters $a_1,\ldots,a_m$,
 we use $A\lesssim_{a_1,\ldots,a_m} B$ to denote the statement that $A \leq C(a_1,\ldots,a_m)B$.
 The notation $A\sim B$ means that $A\lesssim B\lesssim A$.

 \item[(ii)] The symbol $\Omega$ stands for the half-space $\{x\in\mathbb{R}^3: x_1>0 \}$.
 The boundary $\Sigma:=\{x\in\mathbb{R}^3: x_1=0 \}$ can be identified to $\mathbb{R}^{2}$.
We introduce $\Omega_t:=(-\infty,t)\times \Omega$ and $\Sigma_t:=(-\infty,t)\times \Sigma.$
Let us denote by $\p_t$ (or $\p_0$) the time derivative $\frac{\p}{\p t}$
and by $\p_i$ the space derivative $\frac{\p}{\p x_i}$.
We define $\nabla:=(\p_1,\p_2,\p_3)^{ \top}$ and  $x':=(x_2,x_3)$.

\item[(iii)]
For any $\alpha=(\alpha_1,\ldots,\alpha_n)\in\mathbb{N}^n$ and $u=(u_1,\ldots,u_n)\in\mathbb{R}^n$,
we introduce
\begin{align*}
&\alpha!:=\alpha_1!\cdots \alpha_n!,
\quad |\alpha|:=\alpha_1+\cdots+\alpha_n,
\quad u^{\alpha}:=u_1^{\alpha_1}\cdots u_n^{\alpha_n},
\\
&\mathrm{D}_{u}:=\left(\frac{\p}{\p u_1},\ldots, \frac{\p}{\p u_n} \right)^{\top},
\quad
\mathrm{D}_{u}^{\alpha}:=\left(\frac{\p}{\p u_1}\right)^{\alpha_1}\cdots
\left(\frac{\p}{\p u_n} \right)^{\alpha_n}.
\end{align*}
In particular,
$\mathrm{D}_{x'}:=(\p_2,\p_3)^{ \top}$
and $\mathrm{D}_{x'}^{\alpha}:=\p_2^{\alpha_2} \p_3^{\alpha_3}$
for $\alpha:=(\alpha_{2},\alpha_{3})\in\mathbb{N}^{2}$.
If $m\geq 2$ is an integer,  then we denote by
\begin{align}
\nonumber
 \mathrm{D}_{x'}^{m}:=(\p_2^m,\p_2^{m-1}\p_3,\ldots,\p_2 \p_3^{m-1}, \p_3^m)^{\top}
\end{align}
the vector of all partial derivatives in $x'$ of order $m$.

 \item[(iv)]
To simplify the notation, we write
\begin{alignat}{3}
 \nonumber
 &\mathrm{D}_{\rm tan}:=\mathrm{D}_{(t,x')}
 =(\p_t,\p_2,\p_3)^{\top},
 \quad &&
 \mathrm{D}_{\rm tan}^{\beta}:=\mathrm{D}_{(t,x')}^{\beta}
 =\p_t^{{\beta}_0}\p_2^{{\beta}_2}\p_3^{{\beta}_3},\\
 \nonumber
&\mathrm{D}:=\mathrm{D}_{(t,x)}=(\p_t,\p_1,\p_2,\p_3)^{ \top},
\quad &&\mathrm{D}^{\alpha}:=\mathrm{D}_{(t,x)}^{\alpha}
=\p_t^{\alpha_0} \p_1^{\alpha_1}\p_2^{\alpha_2} \p_3^{\alpha_3},
\end{alignat}
where  ${\beta}=({\beta}_0,{\beta}_{2},{\beta}_{3})\in\mathbb{N}^{3}$
and $\alpha=(\alpha_0,\alpha_1,\alpha_2,\alpha_{3})\in\mathbb{N}^{4}$.
Given any  integer $m\geq 0$, we define
 \begin{align}
  \label{VERT.tan}
  \VERT u \VERT_{{\rm tan},\,m}^2
  := \sum_{|{\beta}|\leq m}\| \mathrm{D}_{\rm tan}^{\beta} u  \|_{L^2(\Omega)}^2,
 \quad
\VERT u  \VERT_{m}^2 :=\sum_{|\alpha|\leq m} \|\mathrm{D}^{\alpha} u  \|_{L^{2} (\Omega)}^2.
 \end{align}

 \item[(v)]
 For any integer $m\geq 0$,
 a generic and smooth matrix-valued function of
 $\{(\mathrm{D}^{\alpha} \mathring{V},\mathrm{D}^{\alpha}\mathring{\varPsi},
 \mathrm{D}^{\alpha}\mathrm{D}_{x'}\mathring{\varPsi}): |\alpha|\leq m\}$,
 is denoted by $\mathring{\rm c}_m$,
 and by $\underline{\mathring{\rm c}}_m$ if it vanishes at the origin.
 The exact forms of $\mathring{\rm c}_m$ and $\underline{\mathring{\rm c}}_m$
 may change at each occurrence.
%We introduce
% \begin{align}
%  \label{C.ring}
%  &\mathring{\rm C}_m
%  :={1+\| (\mathring{V},  \mathring{\varPsi}, \mathrm{D}_{x'}\mathring{\varPsi}
%   )\|_{H^{m}(\Omega_T)}^2},
% \end{align}
%in order to shorten our formulas in the computations.

\end{itemize}

The following Moser-type calculus inequalities will be frequently employed in our calculations.
We refer the reader to \cite[Lemma 4.3]{CSW20MR4110436} and the references therein for the detailed proof.

\begin{lemma}
Let $n,d,m\in \mathbb{N}_+$.
Suppose that
$u= u(y)\in \mathbb{R}^n$ and $ w=w(y)\in \mathbb{R} $ are defined on $\mathcal{O}$,
where $\mathcal{O}\subset \mathbb{R}^d$ is any open set with Lipschitz boundary.
Let $h\in C^{\infty}(\mathbb{R}^n)$ and
$\alpha$, $\beta$, $\gamma\in\mathbb{N}^d$ with $|\alpha+\beta+\gamma|\leq m$.
\begin{itemize}
\item If $h(0)=0$ and $u \in L^{\infty}(\mathcal{O})\cap H^m(\mathcal{O})$, then
  \begin{alignat}{3}
	\label{Moser1}
	&\|h(u)\|_{H^{m}(\mathcal{O})}\lesssim_{C_0}\|u\|_{H^{m}(\mathcal{O})}.
\end{alignat}
\item If $ u,w\in L^{\infty}(\mathcal{O})\cap H^m(\mathcal{O})$, then
  \begin{align}
	\nonumber
	&\|\mathrm{D}_y^{\alpha} u\mathrm{D}_y^{\beta} w \|_{L^{2}(\mathcal{O})}+
	\| uw \|_{H^m(\mathcal{O})}
	\\
	&\qquad\qquad\qquad   \lesssim	
	\|u \|_{H^{m}(\mathcal{O})} \|w \|_{L^{\infty}(\mathcal{O})}
	+  \|u \|_{L^{\infty}(\mathcal{O})}\|w \|_{H^{m}(\mathcal{O})},
	\label{Moser2}
	\\[1mm]
	\nonumber
	&   \|\mathrm{D}_y^{\alpha} h(u)\mathrm{D}_y^{\beta} w \|_{L^{2}(\mathcal{O})}
	+     \| h(u)w \|_{H^m(\mathcal{O})}
	+  \|\mathrm{D}_y^{\alpha}[\mathrm{D}_y^{\beta},h(u)]\mathrm{D}_y^{\gamma}w \|_{L^{2}(\mathcal{O})}\\
	&\qquad\qquad\qquad\qquad \quad\, \lesssim_{C_0}
	\|u \|_{H^{m}(\mathcal{O})} \|w \|_{L^{\infty}(\mathcal{O})}
	+ \|w \|_{H^{m}(\mathcal{O})}.
	\label{Moser3}
\end{align}
\item  If $ w \in L^{\infty}(\mathcal{O})\cap H^{m-1}(\mathcal{O})$ and $u\in W^{1,\infty}(\mathcal{O})\cap H^{m}(\mathcal{O})$, then
  \begin{align}
	\|\mathrm{D}_y^{\alpha}[\mathrm{D}_y^{\beta},h(u)]\mathrm{D}_y^{\gamma}w \|_{L^{2}(\mathcal{O})}
	&\lesssim_{C_1}
	\|u  \|_{H^{m}(\mathcal{O})}  \|w  \|_{L^{\infty}(\mathcal{O})}
	+ \|w \|_{H^{m-1}(\mathcal{O})}.
	\label{Moser4}
\end{align}
\end{itemize}
Here $C_0 \geq \|u\|_{L^{\infty}(\mathcal{O})}$ and $C_1 \geq \|u\|_{W^{1,\infty}(\mathcal{O})}$ are some constants,
and $[a,b]c:=a(bc)-b(ac)$ denotes the commutator.
\end{lemma}

\section{Unique Solvability of the Linearized Problem} \label{sec:lin1}

In this section, we perform the linearization of \eqref{MCD0} and prove the well-posedness in the Sobolev space $H^1$ for the linearized problem.

\subsection{Linearization}

Let the basic state $(\mathring{U}(t,x),\mathring{\varphi}(t,x'))$
be a given and sufficiently smooth vector-valued function with $\mathring{U}:=(\mathring{U}^+,\mathring{U}^-)^{\top}$ and $\mathring{U}^{\pm}:=(\mathring{p}^{\pm},\mathring{v}^{\pm},\mathring{H}^{\pm},\mathring{S}^{\pm})^{\top}$.
Suppose that the basic state satisfies the hyperbolicity condition
\begin{align}
\label{bas1a}
\rho_*<\inf_{\Omega_T}\rho^{\pm}(\mathring{U}^{\pm})\leq \sup_{\Omega_T}\rho^{\pm}(\mathring{U}^{\pm})<\rho^*
\quad \textrm{for }\Omega_T:=(-\infty,T)\times \Omega,
\end{align}
the ``relaxed'' requirement of \eqref{BC3a},
\begin{gather}
\label{bas1b}
|\mathring{H}^{\pm}\cdot\mathring{N}|\geq \frac{\kappa}{2}>0
\quad \textrm{on } \Sigma_T:=(-\infty,T)\times\Sigma,
\end{gather}
and the last six conditions in \eqref{MCD0b} together with the constraint \eqref{inv2b},
\begin{align}
\label{bas1c}
%[\mathring{p}]=\mathfrak{s}\mathcal{H}(\mathring{\varphi}),\quad
[\mathring{v}]=0,\quad
[\mathring{H}]=0,\quad
\p_t\mathring{\varphi}=\mathring{v} ^+ \cdot\mathring{N}
\qquad \textrm{on } \Sigma_T,
\end{align}
where $\mathring{N} :=(1,-\p_2\mathring{\varphi} ,-\p_3\mathring{\varphi})^{\top}.$
Moreover, we suppose that
\begin{align}
\label{bas1d}
\|\mathring{V}\|_{H^5(\Omega_T)}
+\|(\mathring{\varphi},\mathrm{D}_{x'}\mathring{\varphi})\|_{H^5(\Sigma_T)}\leq K
\quad
\textrm{for } \mathring{V}:=(\mathring{V}^+,\mathring{V}^-)^{\top},
\end{align}
where  $K>0$ is some constant and
$\mathring{V}^{\pm}:=\mathring{U}^{\pm}-\widebar{U}^{\pm}$ denote the perturbations from the constant states $\widebar{U}^{\pm}$  ({\it cf.}~\eqref{U.bar:def}).
It follows from the embedding theorem and the assumption \eqref{bas1d} that
%\begin{align}
%\nonumber % \label{bas1d2}
$ \|\mathring{V}\|_{W^{2,\infty}(\Omega_T)} +\|(\mathring{\varphi},\mathrm{D}_{x'}\mathring{\varphi})\|_{W^{3,\infty}(\Sigma_T)}\lesssim K.
$
%\end{align}
Let us define
\begin{align*}
 \mathring{\varPhi}^{\pm}(t,x):=\pm x_1+\mathring{\varPsi}^{\pm}(t,x),\quad
 \mathring{\varPsi}^{\pm}(t,x):=\chi(\pm x_1)\mathring{\varphi}(t,x').
\end{align*}
Without loss of generality we assume that
$\|\mathring{\varphi}\|_{L^{\infty}(\Sigma_T)}\leq {\frac{1}{2}},$
leading to {$\p_1\mathring{\varPhi}^{+} \geq  \frac{1}{2}$ and $\p_1\mathring{\varPhi}^{-} \leq  -\frac{1}{2}$}  in $ \Omega_T.$
%Since $\chi\equiv 1$ on $[-1,1]$ and $\chi\in C_0^{\infty}(\mathbb{R})$, we deduce
Use the properties of the cut-off function $\chi$ to find
\begin{gather*}
\|(\mathring{\varPsi},\mathrm{D}_{x'}\mathring{\varPsi})\|_{H^m(\Omega_T)}\sim \|(\mathring{\varphi},\mathrm{D}_{x'}\mathring{\varphi})\|_{H^m(\Sigma_T)},\\
\|(\mathring{\varPsi},\mathrm{D}_{x'}\mathring{\varPsi})\|_{W^{m,\infty}(\Omega_T)}\sim \|(\mathring{\varphi},\mathrm{D}_{x'}\mathring{\varphi})\|_{W^{m,\infty}(\Sigma_T)},
\end{gather*}
for $m\in\mathbb{N}$ and $\mathring\varPsi:=(\mathring\varPsi^+,\mathring\varPsi^-)^{\top}$.
As a result, we obtain
\begin{align}
\label{bas1d3}
\|(\mathring{V},\mathring{\varPsi},\mathrm{D}_{x'}\mathring{\varPsi})\|_{H^5(\Omega_T)}
+ \|\mathring{V}\|_{W^{2,\infty}(\Omega_T)}
+\|(\mathring{\varPsi},\mathrm{D}_{x'}\mathring{\varPsi})\|_{W^{3,\infty}(\Omega_T)}\lesssim K.
\end{align}

The linearized operators for \eqref{MCD0a}--\eqref{MCD0b} around the basic state $(\mathring{U},\mathring{\varphi})$ are defined by
\begin{align}
\label{L'bb:def}
\left\{
\begin{aligned}
&\mathbb{L}_{\pm}'\big(\mathring{U}^{\pm},\mathring{\varPhi}^{\pm}\big)(V^{\pm},\varPsi^{\pm})
:=\left.\frac{\mathrm{d}}{\mathrm{d}\theta}
\mathbb{L}_{\pm}\big(\mathring{U}^{\pm}+{\theta}V^{\pm}, \mathring{\varPhi}^{\pm}+{\theta}\varPsi^{\pm}\big)\right|_{\theta=0},\\[0.5mm]
&\mathbb{B}'\big(\mathring{U},\mathring{\varphi}\big)(V,\psi)
:=\left.\frac{\mathrm{d}}{\mathrm{d}\theta}
\mathbb{B}(\mathring{U}^{+}+{\theta}V^+,\mathring{U}^{-}+{\theta}V^-,
\mathring{\varphi}+{\theta}\psi)\right|_{\theta=0},
\end{aligned}
\right.
\end{align}
where $V:=(V^+,V^-)^{\top}$.
Applying the ``good unknown'' of {\sc Alinhac} \cite{A89MR976971},
\begin{align} \label{good}
\dot{V}:=\begin{pmatrix}
\dot{V}^+ \\ \dot{V}^-
\end{pmatrix}
\quad
\textrm{with \ }
\dot{V}^{\pm}:=V^{\pm}-\frac{\varPsi^{\pm}}{\partial_1 \mathring{\varPhi}^{\pm}}\partial_1\mathring{U}^{\pm},
\end{align}
we can simplify the linearized interior operators as
\begin{align}
 \mathbb{L}_{\pm}'\big(\mathring{U}^{\pm},\mathring{\varPhi}^{\pm}\big)(V^{\pm},\varPsi^{\pm})
&  = \mathbb{L}'_{e\pm}\big(\mathring{U}^{\pm},\mathring{\varPhi}^{\pm}\big)V^{\pm}
-L_{\pm}(\mathring{U}^{\pm}, \mathring{\varPhi}^{\pm})\varPsi^{\pm}\frac{\p_1\mathring{U}^{\pm}}{\p_1 \mathring{\varPhi}^{\pm}}
\label{L.prime}\\
&  =\mathbb{L}'_{e\pm}\big(\mathring{U}^{\pm},\mathring{\varPhi}^{\pm}\big)\dot{V}^{\pm}
+\frac{\varPsi^{\pm}}{\partial_1\mathring{\varPhi}^{\pm}}
\partial_1\mathbb{L}_{\pm}(\mathring{U}^{\pm},\mathring{\varPhi}^{\pm} )
\label{Alinhac}
\end{align}
with
\begin{align}
\label{L'e:def}
\mathbb{L}'_{e\pm}\big(U,\varPhi\big)V:=L_{\pm}\big(U,\varPhi\big)V+\mathcal{C}_{\pm}( U,\varPhi)V,
\end{align}
where $L_{\pm}(U,\varPhi)$ are the differential operators given in \eqref{L:def}
and $\mathcal{C}_{\pm}(U,\varPhi)$ are the zero-th order operators defined by
\begin{align} \nonumber %\label{C.cal}
\mathcal{C}_{\pm}({U},{\varPhi})V:=
\sum_{k=1}^{8}V_k\bigg(\frac{\p A^{\pm}_0}{\p {U_k}}({U}) \partial_t {U}
+ \frac{\p \widetilde{A}^{\pm}_1}{\p {U_k}}({U},{\varPhi}) \partial_1 {U}
+\sum_{i=2,3}\frac{\p A^{\pm}_i}{\p {U_k}}({U}) \partial_i {U}
\bigg).
\end{align}
It is worth pointing out that $\mathcal{C}_{\pm}({U},{\varPhi})$ are smooth matrix-valued functions of $(U,\mathrm{D}U,\mathrm{D}\mathrm{\varPhi})$ with $\mathrm{D}:=(\p_t,\p_1,\p_2,\p_3)^{\top}$.
The good unknown \eqref{good} is introduced to overcome the potential difficulty arising from the presence of the first-order terms in $\varPsi^{\pm}$; {\it cf.}~\eqref{L.prime}--\eqref{Alinhac}.

Using the constraint $[\mathring{H}_1]=0$, we compute ({\it cf.}~\cite[Section 2.1]{TW21b})
\begin{align}
\mathbb{B}'\big(\mathring{U},\mathring{\varphi}\big)(V,\psi)
=\begin{pmatrix}
 [p]-\mathfrak{s} \mathrm{D}_{x'}\cdot
 \bigg(\dfrac{\mathrm{D}_{x'}\psi}{|\mathring{N}|}-
 \dfrac{\mathrm{D}_{x'}\mathring{\varphi}\cdot\mathrm{D}_{x'}\psi}{|\mathring{N}|^3}\mathrm{D}_{x'}\mathring{\varphi}\bigg)\\[4mm]
 [v]\\[1mm]
 [H]\cdot\mathring{\tau}_1\\[1mm]
 [H]\cdot\mathring{\tau}_2\\[1mm]
 (\p_t + \mathring{v}_2^+ \p_2 + \mathring{v}_3^+ \p_3) \psi-v^+\cdot\mathring{N}
\end{pmatrix},
\label{B'.bb:def}
\end{align}
where  $\mathring{\tau}_1:=(\p_2\mathring{\varphi},1,0)^{\top}$
and
$\mathring{\tau}_2:=(\p_3\mathring{\varphi},0,1)^{\top}.$
Plug \eqref{good} into \eqref{B'.bb:def} to get
\begin{align}
\label{B'.bb:iden}
\mathbb{B}'(\mathring{U} ,\mathring{\varphi})(V,\psi)
=\mathbb{B}'_e(\mathring{U}, \mathring{\varphi}) (\dot{V},\psi),
\end{align}
where
\begin{align}
\mathbb{B}'_e(\mathring{U}, \mathring{\varphi}) (\dot{V},\psi) :=
\begin{pmatrix}
 [\dot{p}]-\mathring{a}_1\psi-\mathfrak{s} \mathrm{D}_{x'}\cdot
 \bigg(\dfrac{\mathrm{D}_{x'}\psi}{|\mathring{N}|}- \dfrac{\mathrm{D}_{x'}\mathring{\varphi}\cdot\mathrm{D}_{x'}\psi}{|\mathring{N}|^3}\mathrm{D}_{x'}\mathring{\varphi}\bigg)\\[4mm]
[{\dot{v}}] +\psi(\p_1\mathring{v}^+ +\p_1\mathring{v}^-) \\[1mm]
[\dot{H}]\cdot\mathring{\tau}_1-\mathring{a}_5\psi\\[1mm]
[\dot{H}]\cdot\mathring{\tau}_2-\mathring{a}_6\psi\\[1mm]
 (\p_t  +\mathring{v}_2^+ \p_2+\mathring{v}_3^+ \p_3)\psi
-\dot{v}^+\cdot\mathring{N}+\mathring{a}_7\psi
\end{pmatrix}
\label{Be.bb:def}
\end{align}
with
\begin{align}
\label{a.ring:def}
\left\{\begin{aligned}
&\mathring{a}_1:=-\p_1\mathring{p}^+-\p_1\mathring{p}^-,
&&\mathring{a}_5:=-\mathring{\tau}_1\cdot(\p_1\mathring{H}^+ +\p_1\mathring{H}^-),\\
&\mathring{a}_7:=-\p_1 \mathring{v}^+\cdot \mathring{N},\ &&
\mathring{a}_6:=-\mathring{\tau}_2\cdot(\p_1\mathring{H}^+ +\p_1\mathring{H}^-).
\end{aligned}
\right.
\end{align}

In light of the nonlinear analysis in \cite{CS08MR2423311,CSW19MR3925528,T09CPAMMR2560044,T09ARMAMR2481071,TW21MR4201624,TW21b},
we neglect the last terms in \eqref{Alinhac} to consider the {\it effective linear problem}
\begin{subequations} \label{ELP1}
\begin{alignat}{3}
\label{ELP1a}
&\mathbb{L}'_{e\pm}\big(\mathring{U}^{\pm},\mathring{\varPhi}^{\pm}\big)\dot{V}^{\pm}=f^{\pm}
&\qquad &\textrm{in } \Omega,\\
\label{ELP1b}
&\mathbb{B}'_e(\mathring{U}, \mathring{\varphi}) (\dot{V},\psi)=g
&&\textrm{on } \Sigma,\\
\label{ELP1c}
&(\dot{V},\psi)=0   &&\textrm{if } t<0,
\end{alignat}
\end{subequations}
where the operators $\mathbb{L}'_{e\pm}$ are defined by \eqref{L'e:def}.
The well-posedness result in $H^1$ for the effective linear problem \eqref{ELP1} is stated in the following theorem.
\begin{theorem}
\label{thm:lin}
Let the basic state $(\mathring{U},\mathring{\varphi})$ satisfy  \eqref{bas1a}--\eqref{bas1d}.
Then for all $f^{\pm}\in H^{1}(\Omega_T)$ and $g\in H^{3/2}(\Sigma_T)$ that vanish in the past, the problem \eqref{ELP1} admits a unique solution $(\dot{V},\psi)\in H^{1}(\Omega_T)\times H^{1}(\Sigma_T)$, such that
\begin{multline}
 \|\dot{V}  \|_{H^1(\Omega_T)}
 +  \|(\psi, \mathrm{D}_{x'}\psi )\|_{H^1(\Sigma_T)}\\
\leq C(K,\kappa,T)
\left( \|(f^+,f^-)\|_{H^1(\Omega_T)}+\|g\|_{H^{3/2}(\Sigma_T)}\right)
\label{H1:es}
\end{multline}
for some positive constant $C(K,\kappa,T)$ independent of $f^{\pm}$ and $g.$
\end{theorem}

The rest of this section is dedicated to the proof of Theorem \ref{thm:lin}.

\subsection{Reductions}

It is more convenient to reduce the linearized problem \eqref{ELP1} into the case with homogeneous boundary conditions.
More precisely, if the source term $g=(g_1,\ldots,g_7)^{\top}$ vanishes in the past and belongs to $ H^{m+1/2}(\Sigma_T)$ for some $m\in\mathbb{N}$,
then we can define
${V}_{\natural}^{\pm}:=(p_{\natural}^{\pm},v_{\natural}^{\pm},H_{\natural}^{\pm},0)^{\top}\in H^{m+1}(\Omega_T)$ by
\begin{align}
\nonumber %\label{V.natural:def}
\left\{
\begin{aligned}
&p_{\natural}^{+}:=\mathfrak{R}_Tg_1,\
&& v_{\natural}^+:=-\mathfrak{R}_T(g_7,\,0,\,0)^{\top},\
&&H_{\natural}^+:=\mathfrak{R}_T (0,\, g_5,\, g_6)^{\top},\\
&p_{\natural}^-:= 0,\
&&v_{\natural}^-:=-\mathfrak{R}_T(g_2+ g_7,\, g_3,\, g_4)^{\top},
\
&&H_{\natural}^-:= (0,\,0,\,0)^{\top}
\end{aligned} \right.
\end{align}
where $\mathfrak{R}_T$ denotes the extension operator that is continuous from $H^{k+1/2}(\Sigma_T) $ to $H^{k+1}(\Omega_T)$ and satisfies
\begin{align} \label{R.frak}
(\mathfrak{R}_T w)|_{\Sigma_T}=w,\quad
\|\mathfrak{R}_T w\|_{H^{k+1}(\Omega_T)}\lesssim \|w\|_{H^{k+1/2}(\Sigma_T)},
\end{align}
for all $k=0,\ldots,m$.
Then the $H^{m+1}(\Omega_T)$--function
${V}_{\natural}:=({V}_{\natural}^+,{V}_{\natural}^-)^{\top}$ vanishes in the past and satisfies
\begin{align}
 \label{V.natural:es}
\left\{\begin{aligned}
&\mathbb{B}'_e(\mathring{U}, \mathring{\varphi}) (V_{\natural},0)
=g
&\quad& \textrm{on } \Sigma_T,\\[0.5mm]
&\|{V}_{\natural} \|_{H^{k+1}(\Omega_T)}\lesssim \|g\|_{H^{k+1/2}(\Sigma_T)}
&& \textrm{for } k=0,\ldots, m.
\end{aligned}
\right.
\end{align}
Consequently,
the new unknowns $V_{\flat}^{\pm}:=\dot{V}^{\pm}-{V}_{\natural}^{\pm}$
solve the following problem with zero boundary source term:
\begin{subequations} \label{ELP2}
\begin{alignat}{3}
\label{ELP2a}
&\mathbb{L}'_{e\pm}(\mathring{U}^{\pm}, \mathring{\varPhi}^{\pm}) {V}^{\pm}=f^{\pm}-\mathbb{L}_{e\pm}'(\mathring{U}^{\pm},\mathring{\varPhi}^{\pm})V_{\natural}^{\pm}
&\qquad &\textrm{in } \Omega,\\
\label{ELP2b}
&\mathbb{B}'_e(\mathring{U}, \mathring{\varphi}) ({V},\psi)=0
&&\textrm{on } \Sigma,\\[1mm]
\label{ELP2c}
&({V},\psi)=0   &&\textrm{if } t<0,
\end{alignat}
\end{subequations}
where the subscript ``$\flat$'' has been dropped to simplify the notation.

Moreover, to distinguish the noncharacteristic variables from others for the problem \eqref{ELP2}, we shall introduce the vectors
\begin{align} \nonumber
W^{\pm}:=
\big(p^{\pm},\, v^{\pm}\cdot \mathring{N}^{\pm}, \,v_2^{\pm}, \, v_3^{\pm},\, H^{\pm}\cdot\mathring{\tau}_1^{\pm},\, H^{\pm}\cdot\mathring{\tau}_2^{\pm},\,
H^{\pm}\cdot\mathring{N}^{\pm},\, S^{\pm}\big)^{\top},
\end{align}
where
\begin{align}
\nonumber %\label{N.pm.ring:def}
\mathring{N}^{\pm}:=(1,-\p_2\mathring{\varPhi}^{\pm}, -\p_3\mathring{\varPhi}^{\pm})^{\top},
\
\mathring{\tau}_1^{\pm}:=(\p_2\mathring{\varPhi}^{\pm}, 1, 0)^{\top},
\
\mathring{\tau}_2^{\pm}:=(\p_3\mathring{\varPhi}^{\pm},  0,1)^{\top}.
\end{align}
Equivalently, we set
\begin{align}
W^{\pm}:=\mathring{J}_{\pm}^{-1}V^{\pm}\quad
\textrm{with \ }
\label{J.ring}
\mathring{J}_{\pm}:=\mathrm{diag}\,(1,\,\mathring{J}^{v}_{\pm},\,\mathring{J}^{H}_{\pm},\,1),
\end{align}
where
\setlength{\arraycolsep}{8pt}
\begin{gather}
\nonumber %\label{J.ring.v}
\mathring{J}^{v}_{\pm}:=
\begin{pmatrix}
1 & \partial_2\mathring{\varPhi}^{\pm}  & \partial_3\mathring{\varPhi}^{\pm}  \\[1mm]
0& 1 &0 \\[1mm]
\w{0} &\w{0}& \w{1}
\end{pmatrix},
\\[1mm]
\label{JH.ring:def}
\mathring{J}^{H}_{\pm}:=
\frac{1}{{|\mathring{N}^{\pm}|^2}}\begin{pmatrix}
 \p_2\mathring{\varPhi}^{\pm} &  \p_3\mathring{\varPhi}^{\pm}   & 1 \\[1mm]
  1+(\p_3\mathring{\varPhi}^{\pm})^2  & - \p_2\mathring{\varPhi}^{\pm} \p_3\mathring{\varPhi}^{\pm}   & - \p_2\mathring{\varPhi}^{\pm}   \\[1mm]
 - \p_2\mathring{\varPhi}^{\pm} \p_3\mathring{\varPhi}^{\pm}  &  1+(\p_2\mathring{\varPhi}^{\pm})^2   & - \p_3\mathring{\varPhi}^{\pm}
\end{pmatrix}.
\end{gather}
\setlength{\arraycolsep}{4pt}%
\iffalse
\begin{align} \label{J.ring}
 \mathring{J}_{\pm}:=
 \begin{pmatrix}1 & 0 & 0 & 0 & 0 & 0 & 0 & 0  \\[2mm]
  0 & 1 & \partial_2\mathring{\varPhi}^{\pm}  & \partial_3\mathring{\varPhi}^{\pm}  & 0 & 0 & 0 & 0  \\[2mm]
  0 & 0 & 1 & 0 & 0 & 0 & 0 & 0  \\[2mm]
  0 & 0 & 0 & 1 & 0 & 0 & 0 & 0  \\[2mm]
  0 & 0 & 0 & 0 & \dfrac{\p_2\mathring{\varPhi}^{\pm} }{|\mathring{N}^{\pm}|^2}& \dfrac{\p_3\mathring{\varPhi}^{\pm} }{|\mathring{N}^{\pm}|^2} & \dfrac{1 }{|\mathring{N}^{\pm}|^2} & 0  \\[4mm]
  0& 0 & 0 & 0 &\dfrac{1+(\p_3\mathring{\varPhi}^{\pm})^2 }{|\mathring{N}^{\pm}|^2} & -\dfrac{\p_2\mathring{\varPhi}^{\pm} \p_3\mathring{\varPhi}^{\pm} }{|\mathring{N}^{\pm}|^2} & -\dfrac{\p_2\mathring{\varPhi}^{\pm} }{|\mathring{N}^{\pm}|^2} & 0  \\[4mm]
  0& 0 & 0 & 0 & -\dfrac{\p_2\mathring{\varPhi}^{\pm} \p_3\mathring{\varPhi}^{\pm} }{|\mathring{N}^{\pm}|^2} & \dfrac{1+(\p_2\mathring{\varPhi}^{\pm})^2 }{|\mathring{N}^{\pm}|^2}  & -\dfrac{\p_3\mathring{\varPhi}^{\pm} }{|\mathring{N}^{\pm}|^2}& 0  \\[4mm]
  %\w0 & \w0 & \w0 & \w0 & \w0 & \w0 & \w0 & \w{1}
  0 & 0 &0 & 0 &0 & 0 & 0 & 1
 \end{pmatrix},
\end{align}
\fi %
We remark that the matrices $\mathring{J}_{\pm}^v$, $\mathring{J}_{\pm}^H$, and $\mathring{J}_{\pm}$ are all invertible and smooth in $\mathrm{D}_{x'}\mathring{\varPsi}^{\pm}$.
Hence the problem \eqref{ELP2} can be rewritten equivalently as
\begin{subequations}
\label{ELP3}
\begin{alignat}{3}
\label{ELP3a}
&{{\bf L}}^{\pm}W^{\pm}:=\sum_{i=0}^3{\bm{A}}_i^{\pm} \p_i W^{\pm} +{\bm{A}}_4^{\pm}W^{\pm}=\bm{f}^{\pm}
&\quad &\textnormal{in }\Omega_T,\\
\label{ELP3b}
&[W_1]=\mathfrak{s}\mathrm{D}_{x'}\cdot
\bigg(\dfrac{\mathrm{D}_{x'}\psi}{|\mathring{N}|}-
\dfrac{\mathrm{D}_{x'}\mathring{\varphi}\cdot\mathrm{D}_{x'}\psi}{|\mathring{N}|^3}\mathrm{D}_{x'}\mathring{\varphi}
\bigg)+\mathring{a}_1\psi
&\quad &\textnormal{on }\Sigma_T,\\[0.5mm]
\label{ELP3c}
&[W_i]=\mathring{a}_i\psi \qquad \textrm{for }i=2,\ldots,6,
&\quad &\textnormal{on }\Sigma_T,\\[2mm]
\label{ELP3d}
&W_2^+=\mathring{\rm B}\psi
:=(\p_t+\mathring{v}_2^+\p_2+\mathring{v}_3^+\p_3)\psi+\mathring{a}_7 \psi
&\quad &\textnormal{on }\Sigma_T,\\[2mm]
\label{ELP3e}
&(W,\psi)=0 &\quad &\textnormal{if }t<0,
\end{alignat}
\end{subequations}
where $\p_0:=\frac{\p}{\p t}$ denotes the time derivative, $W:=(W^+,W^-)^{\top}$,
the terms $\mathring{a}_1$,  $\mathring{a}_5$,  $\mathring{a}_6$, $\mathring{a}_7$ are defined by \eqref{a.ring:def}, and
\begin{align}
\label{A.bm:def}
\left\{
\begin{aligned}
&{\bm{A}}_1^{\pm}:=\mathring{J}_{\pm}^{\top}\widetilde{A}^{\pm}_1(\mathring{U}^{\pm},\mathring{\varPhi}^{\pm})\mathring{J}_{\pm},\quad
&&{\bm{A}}_{i}^{\pm}:=\mathring{J}_{\pm}^{\top}A^{\pm}_{i}(\mathring{U}^{\pm})\mathring{J}_{\pm}\quad
\textrm{for }i=0,2,3, \\
&
{\bm{A}}_4^{\pm}:= \mathring{J}_{\pm}^{\top}\mathbb{L}_{e\pm}'(\mathring{U}^{\pm},\mathring{\varPhi}^{\pm})\mathring{J}_{\pm},\quad
&&\bm{f}^{\pm}:=\mathring{J}_{\pm}^{\top}\big(f^{\pm}-\mathbb{L}_{e\pm}'(\mathring{U}^{\pm},\mathring{\varPhi}^{\pm})V_{\natural}^{\pm}\big),\\
& \mathring{a}_2:= -\mathring{N}\cdot(\p_1\mathring{v}^+ +\p_1\mathring{v}^-), \quad
&&\mathring{a}_{k+1}:=-\p_1\mathring{v}_k^+ -\p_1\mathring{v}_k^-
\quad \textrm{for } k=2,3.
\end{aligned}
\right.
\end{align}
It is worth pointing out that
the scalars $\mathring{a}_1,\ldots,\mathring{a}_7$ are smooth functions of the traces $(\mathrm{D}\mathring{V},\mathrm{D}_{x'}\mathring{\varPsi})|_{\Sigma_{{T}}}$
and the matrices $\bm{A}_{0}^{\pm},\ldots,\bm{A}_{4}^{\pm}$ are smooth functions of
$(\mathring{V},\mathrm{D}\mathring{V},\mathrm{D}\mathring{\varPsi},\mathrm{D}\mathrm{D}_{x'}\mathring{\varPsi})$.
It should be emphasized that the system \eqref{ELP3a} is still symmetric hyperbolic.

It follows from the identities
\eqref{bas1c} and $\p_1\mathring{\varPhi}^{\pm}|_{\Sigma_T}=\pm 1$ that
\begin{align*}
\widetilde{A}^{\pm}_1(\mathring{U}^{\pm},\mathring{\varPhi}^{\pm})
=\pm \begin{pmatrix}
0& \mathring{N}^{\top} & 0 & 0\\[1mm]
\mathring{N} & O_3 &  \mathring{N} \otimes  \mathring{H}^{\pm} -\mathring{H}_N^{\pm}  I_3&0\\[1mm]
0 & \mathring{H}^{\pm}\otimes \mathring{N}-\mathring{H}_N^{\pm} I_3 &O_3 &0\\[1mm]
\w{0} &\w{0} &\w{0}  &\w{0}
\end{pmatrix}
\quad \textrm{on }\Sigma_T,
\end{align*}
where $\mathring{H}_N^{\pm}:=\mathring{H}^{\pm}\cdot \mathring{N}^{\pm}$.
Then we obtain the decomposition
\begin{align}
\label{decom}
{\bm{A}}_1^{\pm}={\bm{A}}_{(0)}^{\pm}+{\bm{A}}_{(1)}^{\pm},
\quad
{\bm{A}}_{(0)}^{\pm}\big|_{\Sigma_T}=0,
\end{align}
where
\begin{align}
\label{A(1):def}
{\bm{A}}_{(1)}^{\pm}:=
\pm \begin{pmatrix}
0 & 1 &0&0&0&0&0&0 \\[1mm]
1 & 0 & 0&0 &\mathring{H}_2^{\pm}  &\mathring{H}_3^{\pm}&0&0\\[1mm]
0 & 0 & 0 &0&-\mathring{H}_{N}^{\pm} &0&0&0\\[1mm]
0 & 0 & 0 &0 &0&-\mathring{H}_{N}^{\pm} &0&0\\[1mm]
0 & \mathring{H}_2^{\pm}&-\mathring{H}_{N}^{\pm} &0 & 0 &0 &0&0\\[1mm]
0 & \mathring{H}_3^{\pm}&0&-\mathring{H}_{N}^{\pm} & 0 &0 &0&0\\[1mm]
0 & 0&0&0&0&0&0&0 \\
\w{0} & \w{0} &\w{0} & \w{0}&\w{0}&\w{0}&\w{0}&\w{0}
\end{pmatrix}.
\end{align}
According to the kernels of the matrices ${\bm{A}}_1^{\pm}|_{\Sigma_T}$,
we use
\begin{align}
W_{\rm nc}^{\pm}:=(W_1^{\pm},\ldots,W_6^{\pm})^{\top}
\quad \textrm{and} \quad
 W_{\rm c}^{\pm}:=(W_7^{\pm},W_8^{\pm})^{\top}
\label{W.nc:def}
\end{align}
to denote the noncharacteristic and characteristic variables, respectively.
The boundary matrix for the hyperbolic problem \eqref{ELP3},
${\rm diag}\,(-{\bm{A}}_1^+,\,-{\bm{A}}_1^-)$,
has six negative eigenvalues (``incoming characteristics'') on the boundary $\Sigma_T$.
As discussed before,
the correct number of boundary conditions is seven, just the case in \eqref{ELP3b}--\eqref{ELP3d}.
Moreover, for our hyperbolic problem \eqref{ELP3},
the boundary is {\it characteristic of constant multiplicity}
and {\it the maximality condition} is satisfied  in the sense of {\sc Rauch} \cite[Definition 2 and condition (11)]{R85MR0797053}.

\subsection{$H^1$ \textit{a priori} Estimate}
In this subsection, we shall deduce the \textit{a priori} estimate in $H^1$ for solutions $W$ of the reduced problem \eqref{ELP3}.

\subsubsection{$L^2$ estimate of $W$}
Let us first make the $L^2$ estimate of $W$.
Since the matrices $\bm{A}_{0}^{\pm},\ldots,\bm{A}_{3}^{\pm}$  are all symmetric,
we take the scalar product of \eqref{ELP3a} with $W^{\pm}$ respectively and use \eqref{bas1d3} to get
\begin{align} \nonumber
&\sum_{\pm}\int_{\Omega} {\bm{A}}_0^{\pm}W^{\pm}\cdot W^{\pm} \d x
+ \int_{\Sigma_t}  \mathcal{T}_{\rm b}(W)\\
\nonumber &\quad =2\sum_{\pm}\int_{\Omega_t}W^{\pm}\cdot
\big(\bm{f}^{\pm}-\bm{A}_4^{\pm} W^{\pm}\big)+
\sum_{\pm}\sum_{i=0}^{3}\int_{\Omega_t}W^{\pm}\cdot\p_i\bm{A}_i^{\pm}W^{\pm}  \\[1mm]
&\quad \lesssim_K \|(\bm{f},W)\|_{L^2(\Omega_t)}^2,
\label{es1a}
\end{align}
where $K>0$ denotes the upper bound in \eqref{bas1d},
$\bm{f}:=(\bm{f}^+,\bm{f}^-)^{\top}$, and
\begin{align}
\label{T.cal:def}
\mathcal{T}_{\rm b}(U):=
-\sum_{\pm}{\bm{A}}_1^{\pm}U^{\pm}\cdot U^{\pm}
\quad \textrm{for any }U=(U^+,U^-)^{\top}\in\mathbb{R}^{16}.
\end{align}

Utilize the decomposition \eqref{decom}--\eqref{A(1):def},
the identity $[\mathring{H}]|_{\Sigma_{{T}}}=0$ ({\it cf.}~\eqref{bas1c}),
and the boundary conditions \eqref{ELP3c}--\eqref{ELP3d} to obtain
\begin{align}
\mathcal{T}_{\rm b}(W)
&=-2\big[W_2\big(W_1+\mathring{H}_2 W_5 +\mathring{H}_3 W_6\big) -\mathring{H}_N\big(W_3W_5+W_4W_6\big)\big]
\nonumber\\
&  =-2[W_1]W_2^+
+[(W_2,\ldots,W_6)] \mathring{\rm c}_0  \mathcal{U}\nonumber \\
&=-2[W_1]\mathring{\rm B}\psi + \mathring{\rm c}_1\psi  \mathcal{U}
\quad \textrm{ on }\Sigma_T,
 \label{id1a}
\end{align}
where $\mathring{\rm B}$ is the operator defined in \eqref{ELP3d} and
\begin{align}
\mathcal{U}:=
\big(W_1^-,W_2^-,W_3^-,W_4^-,W_2^+,W_5^+,W_6^+\big)^{\top}
\in\mathbb{R}^7.
 \label{U.cal:def}
\end{align}
In all that follows, for any $m\in\mathbb{N}$,
we employ $\mathring{\rm c}_m$ to denote a generic and smooth matrix-valued function of
$\{(\mathrm{D}^{\alpha} \mathring{V},\mathrm{D}^{\alpha}\mathring{\varPsi},
\mathrm{D}^{\alpha}\mathrm{D}_{x'}\mathring{\varPsi}):|\alpha|\leq m\}$.
It follows from \eqref{ELP3b} that
\begin{align}
  -2[W_1]\mathring{\rm B}\psi
   =\;&
 \p_t\bigg\{
 \mathfrak{s}\bigg(\dfrac{|\mathrm{D}_{x'}\psi|^2}{|\mathring{N}|}- \dfrac{|\mathrm{D}_{x'}\mathring{\varphi}\cdot\mathrm{D}_{x'}\psi|^2}{|\mathring{N}|^3}\bigg)
 - \mathring{a}_1    \psi^2   \bigg\}
 +\mathfrak{s}\mathring{\rm c}_2
 \mathrm{D}_{x'}\psi
 \cdot \begin{pmatrix}
  \psi\\ \mathrm{D}_{x'}\psi
 \end{pmatrix}
 \nonumber \\
 \nonumber
 &
  +\mathring{\rm c}_2  \psi^2
 +\sum_{k=2,3}\p_k\bigg\{
 \mathfrak{s} \mathring{v}_k^+ \bigg(\dfrac{|\mathrm{D}_{x'}\psi|^2}{|\mathring{N}|}-
 \dfrac{|\mathrm{D}_{x'}\mathring{\varphi}\cdot\mathrm{D}_{x'}\psi|^2}{|\mathring{N}|^3}
 \bigg)
 - \mathring{a}_1 \mathring{v}_k ^+  \psi^2 \bigg\}
 \\
 & -2\mathfrak{s} \mathrm{D}_{x'}\cdot
 \bigg\{\mathring{\rm B}\psi
 \bigg(\dfrac{\mathrm{D}_{x'}\psi}{|\mathring{N}|}-
 \dfrac{\mathrm{D}_{x'}\mathring{\varphi}\cdot\mathrm{D}_{x'}\psi}{|\mathring{N}|^3}\mathrm{D}_{x'}\mathring{\varphi}
 \bigg)
 \bigg\}
 \quad  \textrm{on } \Sigma_T.
 \label{id1b}
\end{align}

Plugging \eqref{id1a} into \eqref{es1a}, we use \eqref{id1b} and $|\mathring{N}|^2=1+|\mathrm{D}_{x'}\mathring{\varphi}|^2$ to infer
\begin{multline}
 \sum_{\pm}\int_{\Omega} {\bm{A}}_0^{\pm}W^{\pm}\cdot W^{\pm} \d x
+\int_{\Sigma} \bigg( \mathfrak{s}\dfrac{|\mathrm{D}_{x'}\psi|^2}{|\mathring{N}|^3}
-\mathring{a}_1 \psi^2\bigg)\d x'
  \\
 \lesssim_K \|(\bm{f},W)\|_{L^2(\Omega_t)}^2
 +\|\psi\mathcal{U}\|_{L^1(\Sigma_t)}
+\|(\psi, \mathrm{D}_{x'}\psi)\|_{L^2(\Sigma_t)}^2.
\label{es1b}
\end{multline}
Note from integration by parts and the condition \eqref{ELP3e} that
\begin{align} \label{es1c}
 \|\mathrm{D}_{x'}^{\alpha}\psi(t)\|_{L^2(\Sigma)}^2
 =2\int_{\Sigma_t}\mathrm{D}_{x'}^{\alpha}\psi\mathrm{D}_{x'}^{\alpha}\p_t\psi
 \lesssim  \|(\mathrm{D}_{x'}^{\alpha}\psi,\mathrm{D}_{x'}^{\alpha}\p_t\psi)\|_{L^2(\Sigma_t)}^2
\end{align}
for any $\alpha\in \mathbb{N}^2$,
where $\mathrm{D}_{x'}^{\alpha}:=\p_{2}^{\alpha_{2}}\p_{3}^{\alpha_{3}}$ for $\alpha=(\alpha_{2},\alpha_{3})$.
Then we discover
\begin{multline}
\|W(t)\|_{L^2(\Omega)}^2+\| (\psi,\mathrm{D}_{x'}\psi)(t)\|_{L^2(\Sigma)}^2  \\
\lesssim_K \|(\bm{f},W)\|_{L^2(\Omega_t)}^2
 +\|(\psi,\mathrm{D}_{x'}\psi,\p_t\psi, \mathcal{U})\|_{L^2(\Sigma_t)}^2.
\nonumber
\end{multline}
Applying Gr\"{o}nwall's inequality to the last estimate implies
\begin{align}
\|W(t)\|_{L^2(\Omega)}^2+\| (\psi,\mathrm{D}_{x'}\psi)(t)\|_{L^2(\Sigma)}^2
\lesssim_K \|\bm{f}\|_{L^2(\Omega_t)}^2
+\|( \p_t\psi, \mathcal{U})\|_{L^2(\Sigma_t)}^2.
 \label{es1d}
\end{align}
We emphasize that neither the estimate \eqref{es1b} nor \eqref{es1d} for $W$ is closed.

\subsubsection{$L^2$ estimate of $\mathrm{D}_{x'}W$}
We shall close the {\it a priori} estimate in $H^1$.
Let ${\ell}=0,2,3$.
Apply the differential operator $\p_{\ell}$ to \eqref{ELP3a}
and take the scalar product of the resulting equations with $\p_{\ell} W^{\pm}$ respectively to deduce
\begin{align}
 \sum_{\pm}\int_{\Omega} {\bm{A}}_0^{\pm}\p_{\ell}W^{\pm}\cdot \p_{\ell} W^{\pm}\d x
+\int_{\Sigma_t} \mathcal{T}_{\rm b}(\p_{\ell}W)
\lesssim_K \|(\bm{f},W)\|_{H^1(\Omega_t)}^2,
 \label{es2a}
\end{align}
where the operator $\mathcal{T}_{\rm b}$ is defined by \eqref{T.cal:def}.
Similar to \eqref{id1a}, taking advantage of the boundary conditions \eqref{ELP3b}--\eqref{ELP3d}, we obtain
\begin{align}
 \int_{\Sigma_{{t}}} \mathcal{T}_{\rm b}(\p_{\ell}W)
 &= -2\int_{\Sigma_{{t}}}\p_{\ell} [W_1]\p_{\ell} W_2^+
 +\int_{\Sigma_{{t}}}\p_{\ell} ( \mathring{\rm c}_1 \psi)\mathring{\rm c}_0  \p_{\ell} \mathcal{U}
 \nonumber\\
&=\mathcal{J}_{\ell}
-2\int_{\Sigma_{{t}}} \p_{\ell}(\mathring{a}_1\psi) \p_{\ell}W_2^+
 +\int_{\Sigma_{{t}}}\p_{\ell} ( \mathring{\rm c}_1 \psi)\mathring{\rm c}_0  \p_{\ell} \mathcal{U}\nonumber\\
&=\mathcal{J}_{\ell}
+\int_{\Sigma_{{t}}}\p_{\ell} ( \mathring{\rm c}_1 \psi)\mathring{\rm c}_0  \p_{\ell} \mathcal{U}
 \quad \textrm{on }\Sigma_T,
\label{id2a}
\end{align}
where $\mathcal{U}$ is the vector given by \eqref{U.cal:def} and
\begin{align}
\mathcal{J}_{\ell}
:=2\mathfrak{s} \int_{\Sigma_{{t}}}
\p_{\ell}\bigg(\dfrac{\mathrm{D}_{x'}\psi}{|\mathring{N}|}- \dfrac{\mathrm{D}_{x'}\mathring{\varphi}\cdot\mathrm{D}_{x'}\psi}{|\mathring{N}|^3}\mathrm{D}_{x'}\mathring{\varphi} \bigg)\cdot
\mathrm{D}_{x'}\p_{\ell} \mathring{\rm B}\psi .
\label{J.cal:def}
\end{align}

For ${\ell}=0,2,3$,  a lengthy but straightforward computation leads to
\begin{align}
 \mathcal{J}_{\ell}=
\mathcal{J}_{\ell}^a
+\mathcal{J}_{\ell}^b
+ \sum_{|\alpha|\leq 2}\int_{\Sigma_t}
\mathring{\rm c}_2
\begin{pmatrix}
 \mathrm{D}_{x'}^{\alpha}\psi \\ \mathrm{D}_{x'}\p_t \psi
\end{pmatrix}\cdot
\begin{pmatrix}
 \mathrm{D}_{x'}\psi \\   \mathrm{D}_{x'} \p_{\ell}\psi
\end{pmatrix},
\label{J.cal:id}
\end{align}
with
\begin{align}
\nonumber
&\mathcal{J}_{\ell}^a:=\mathfrak{s}\int_{\Sigma}\bigg(
\dfrac{|\mathrm{D}_{x'}\p_{\ell}\psi|^2}{|\mathring{N}|}- \dfrac{|\mathrm{D}_{x'}\mathring{\varphi}\cdot\mathrm{D}_{x'}\p_{\ell}\psi|^2}{|\mathring{N}|^3}
+\mathring{\rm c}_1 \mathrm{D}_{x'}\psi\cdot  \mathrm{D}_{x'}\p_{\ell}\psi
\bigg)\mathrm{d}x'
\\
& \mathcal{J}_{\ell}^b:=
2\mathfrak{s}\int_{\Sigma_t}
 \p_{\ell}\bigg(\dfrac{\mathrm{D}_{x'}\psi}{|\mathring{N}|}- \dfrac{\mathrm{D}_{x'}\mathring{\varphi}\cdot\mathrm{D}_{x'}\psi}{|\mathring{N}|^3}\mathrm{D}_{x'}\mathring{\varphi} \bigg)\cdot
 \mathrm{D}_{x'} \p_{\ell}(\mathring{a}_7\psi).
\nonumber
\end{align}
Utilizing Cauchy's inequality and \eqref{es1c} with $|\alpha|=1$ yields
\begin{align}
\mathcal{J}_{\ell}^a\geq \;&
\mathfrak{s} \int_{\Sigma}   \dfrac{|\mathrm{D}_{x'}\p_{\ell}\psi|^2}{|\mathring{N}|^3}\mathrm{d}x'
-C(K)\int_{\Sigma} | \mathrm{D}_{x'}\p_{\ell} \psi|  |\mathrm{D}_{x'}\psi| \d x'  \nonumber\\
\geq \;&
\frac{\mathfrak{s}}{2} \int_{\Sigma}   \dfrac{|\mathrm{D}_{x'}\p_{\ell}\psi|^2}{|\mathring{N}|^3}\mathrm{d}x'
-C(K)\|(\mathrm{D}_{x'} \psi,\mathrm{D}_{x'}\p_t\psi )\|_{L^2(\Sigma_t)}^2.
\label{Ja.cal:es}
\end{align}
To control the term $\mathcal{J}_{\ell}^b$, we shall use the following classical product estimate.
\begin{lemma} \label{lem:pro}
Let the nonnegative real numbers $s$, $s_1$, and $s_2$ satisfy $s_1,s_2\geq s$ and $s_1+s_2>s+1$.
Then the product mapping $(u,v)\mapsto uv$ is continuous from $H^{s_1}(\mathbb{R}^2)\times H^{s_2}(\mathbb{R}^2)$ to $H^{s}(\mathbb{R}^2)$ and satisfies
\begin{align}
 \label{pro:es}
 \|uv\|_{H^{s}(\mathbb{R}^2)}\lesssim \|u\|_{H^{s_1}(\mathbb{R}^2)}\|v\|_{H^{s_2}(\mathbb{R}^2)}.
\end{align}
\end{lemma}
\noindent
By virtue of \eqref{pro:es} with $s=s_1=1$ and $s_2=\frac{3}{2}$, we deduce
\begin{align}
|\mathcal{J}_{\ell}^b|
&\lesssim
\|\p_{\ell}(\mathring{\rm c}_0 \mathrm{D}_{x'} \psi )\|_{L^2(\Sigma_t)}^2
+\int_{0}^t\| \p_{\ell}(\mathring{a}_7\psi) \|_{H^1(\Sigma)}^2\mathrm{d} \tau
\nonumber \\
&\lesssim_K
\|(\mathrm{D}_{x'} \psi,\mathrm{D}_{x'}\p_{\ell}\psi )\|_{L^2(\Sigma_t)}^2\nonumber
+\int_{0}^t\| ( \psi,\p_{\ell}\psi)\|_{H^1(\Sigma)}^2\| (\mathring{a}_7, \p_{\ell}\mathring{a}_7 ) \|_{H^{3/2}(\Sigma)}^2\mathrm{d}\tau
\\[2mm]
&\lesssim_K
\|(\psi,\p_{\ell}\psi,\mathrm{D}_{x'} \psi,\mathrm{D}_{x'}\p_{\ell}\psi )\|_{L^2(\Sigma_t)}^2,
\label{Jb.cal:es}
\end{align}
where we have used
\begin{multline}
 \| (\mathring{a}_7, \p_{\ell}\mathring{a}_7 ) (t)\|_{H^{3/2}(\Sigma)}
  \lesssim  \|\underline{\mathring{\rm c}}_2(t)\|_{H^{3/2}(\Sigma)}\\
\lesssim  \|\underline{\mathring{\rm c}}_2(t)\|_{H^{2}(\Omega)}
 \lesssim \| \underline{\mathring{\rm c}}_2  \|_{H^{3}(\Omega_T)}
\leq C(K)
\quad  \textrm{for  } 0\leq t \leq T,
\label{bas:es2}
\end{multline}
following from the trace theorem, the Moser-type calculus inequality \eqref{Moser1},
and the relation \eqref{bas1d3}.
Here and below, for any $m\in\mathbb{N}$,
we denote by $\underline{\mathring{\rm c}}_m$ a generic and smooth matrix-valued function of
$\{(\mathrm{D}^{\alpha} \mathring{V},\mathrm{D}^{\alpha}\mathring{\varPsi},
\mathrm{D}^{\alpha}\mathrm{D}_{x'}\mathring{\varPsi}): |\alpha|\leq m\}$
vanishing at the origin.
Substitute \eqref{Ja.cal:es} and \eqref{Jb.cal:es} into \eqref{J.cal:id} to obtain
\begin{align}
 \mathcal{J}_{\ell}
\geq \frac{\mathfrak{s}}{2} \int_{\Sigma}   \dfrac{|\mathrm{D}_{x'}\p_{\ell}\psi|^2}{|\mathring{N}|^3}\mathrm{d}x'
-C(K)\sum_{|\alpha|\leq 2} \|(\mathrm{D}_{x'}^{\alpha}\psi, \p_{\ell}\psi,\mathrm{D}_{x'}\p_t\psi )\|_{L^2(\Sigma_t)}^2
\label{J.cal:es}
\end{align}
for $\ell=0,2,3.$

Regarding the last term in \eqref{id2a} for $\ell=2,3$, we make use of the estimates \eqref{pro:es} and \eqref{bas:es2} to derive
\begin{align}
\left|\int_{\Sigma_{{t}}} \p_{\ell} ( \mathring{\rm c}_1 \psi)\mathring{\rm c}_0  \p_{\ell} \mathcal{U}\right|
&\lesssim \int_{0}^t \|\p_{\ell} ( \mathring{\rm c}_1 \psi)\mathring{\rm c}_0 \|_{H^{1}(\Sigma)}
\|\p_{\ell} \mathcal{U}\|_{H^{-1}(\Sigma)}\d \tau
\nonumber\\
&\lesssim  \int_{0}^t \|\mathring{\rm c}_2( \psi,\p_{\ell}\psi)\|_{H^{1}(\Sigma)}
\| \mathcal{U}\|_{L^{2}(\Sigma)}\d \tau
\nonumber \\
&\lesssim  \int_{0}^t \left(1+\| \underline{\mathring{\rm c}}_2  \|_{H^{3/2}(\Sigma)}\right)\| ( \psi,\p_{\ell}\psi)\|_{H^{1}(\Sigma)}
\| \mathcal{U}\|_{H^{1}(\Omega)}\d \tau
\nonumber \\[1.5mm]
&\lesssim_K  \sum_{|\alpha|\leq 2} \|\mathrm{D}_{x'}^{\alpha}\psi\|_{L^2(\Sigma_t)}^2
+\|W\|_{H^1(\Omega_{{t}})}^2
\quad \textrm{for }\ell=2,3.
\label{es2b}
\end{align}

Plugging \eqref{id2a} into \eqref{es2a} for $\ell=2,3$ and utilizing \eqref{J.cal:es}--\eqref{es2b} imply
\begin{multline}
\|\mathrm{D}_{x'}W(t)\|_{L^2(\Omega)}^2
+\|\mathrm{D}^2_{x'}\psi(t)\|_{L^2(\Sigma)}^2
\\
 \lesssim_K
  \|(\bm{f},W)\|_{H^1(\Omega_t)}^2   +  \sum_{|\alpha|\leq 2} \|(\mathrm{D}_{x'}^{\alpha}\psi, \mathrm{D}_{x'}\p_t\psi)\|_{L^2(\Sigma_t)}^2,
 \label{es2c}
\end{multline}
where $ \mathrm{D}_{x'}^{m}:=(\p_2^m,\p_2^{m-1}\p_3,\ldots,\p_2 \p_3^{m-1}, \p_3^m)^{\top}$ for any integer $m\geq 2$.

\subsubsection{$L^2$ estimate of $\p_t W$}
It follows from \eqref{id2a} that
 \begin{align}
 \int_{\Sigma_t}\mathcal{T}_{\rm b}(\p_t W)
 = \mathcal{J}_0
+\underbrace{\int_{\Sigma_t}   \mathring{\rm c}_1 \p_t\psi\p_t\mathcal{U}}_{\mathcal{I}_1}
+\underbrace{\int_{\Sigma_t}  \mathring{\rm c}_2 \psi\p_t \mathcal{U}}_{\mathcal{I}_2},
 \label{es3a}
\end{align}
where $\mathcal{J}_0$ and $\mathcal{U}$ are defined by \eqref{J.cal:def} and \eqref{U.cal:def}, respectively.

For the integral term $\mathcal{I}_1$, we use \eqref{ELP3d} to deduce
\begin{align}
 \mathcal{I}_1=
 \underbrace{\int_{\Sigma_t}  \mathring{\rm c}_1  \p_t\mathcal{U} W_2^+ }_{\mathcal{I}_{1}^a}
 \underbrace{-\int_{\Sigma_t}  \mathring{\rm c}_1  \p_t\mathcal{U}(\mathring{v}_2^+\p_2\psi+\mathring{v}_3^+\p_3\psi+\mathring{a}_7  \psi) }_{\mathcal{I}_{1}^b}.
 \label{I1:es}
\end{align}
Passing to the volume integral and applying integration by parts imply
\begin{align}
 \mathcal{I}_{1}^a
&= -\int_{\Omega_t}\p_1( \mathring{\rm c}_1 \p_t\mathcal{U} W_2^+  ) \nonumber \\
&  =-\int_{\Omega}  \mathring{\rm c}_1  \p_1 \mathcal{U}W_2^+\,\mathrm{d} x
+\int_{\Omega_t} \mathring{\rm c}_2
\begin{pmatrix}
\mathcal{U}\\ \p_1 \mathcal{U}
\end{pmatrix}\cdot
\begin{pmatrix}
 \mathcal{U} \\ \p_t \mathcal{U}
\end{pmatrix}
 \nonumber\\[2mm]
&  \geq -\boldsymbol{\epsilon} \|\p_1 \mathcal{U}(t)\|_{L^2(\Omega)}^2
-C(\boldsymbol{\epsilon} )C( K)\|\mathcal{U}\|_{H^1(\Omega_t)}^2
\quad \textrm{for all } \boldsymbol{\epsilon}>0,
\label{I1a:es}
\end{align}
and
\begin{align}
\mathcal{I}_{1}^b+\mathcal{I}_{2}
 & =\int_{\Sigma_t}(\psi,\mathrm{D}_{x'}\psi) \mathring{\rm c}_2 \p_t \mathcal{U} \nonumber\\
 &   =\int_{\Sigma}(\psi,\mathrm{D}_{x'}\psi) \mathring{\rm c}_2  \mathcal{U}\d x'
 -\int_{\Sigma_t}\p_t\big((\psi,\mathrm{D}_{x'}\psi) \mathring{\rm c}_2\big)  \mathcal{U}
 \nonumber \\[1mm]
 &   \geq
 -\|\mathcal{U}(t)\|_{L^2(\Sigma)}^2
 -\|\mathcal{U}\|_{L^2(\Sigma_t)}^2
 -C\sum_{ i=0,1}\big\|\p_t^i\big((\psi,\mathrm{D}_{x'}\psi) \mathring{\rm c}_2\big)\big\|_{L^2(\Sigma_t)}^2.
\label{I1b:es1}
\end{align}
In view of the product estimate \eqref{pro:es} with $s=0$, $s_1=1$, and $s_2=\frac{1}{2}$,
for the last term in \eqref{I1b:es1} we obtain
\begin{align*}
& \big\|\p_t\big((\psi,\mathrm{D}_{x'}\psi) \mathring{\rm c}_2\big)\big\|_{L^2(\Sigma_t)}^2\\
&\quad \lesssim   \big\|( \p_t \psi,\mathrm{D}_{x'}\p_t\psi)\mathring{\rm c}_2\big\|_{L^2(\Sigma_t)}^2
+\int_{0}^t \big\|(\psi, \mathrm{D}_{x'} \psi )\p_t\mathring{\rm c}_2\big\|_{L^2(\Sigma)}^2\d \tau\\
&\quad \lesssim_K  \big\|( \p_t \psi,\mathrm{D}_{x'}\p_t\psi)\big\|_{L^2(\Sigma_t)}^2
+\int_{0}^t \big\|(\psi, \mathrm{D}_{x'} \psi ) \big\|_{H^1(\Sigma)}^2
 \big\|\p_t\mathring{\rm c}_2\big\|_{H^{1/2}(\Sigma)}^2 \d \tau\\[1.5mm]
&\quad \lesssim_K
\big\|\big(\p_t \psi,\mathrm{D}_{x'}\p_t\psi\big)\big\|_{L^2(\Sigma_t)}^2
+\sum_{|\alpha|\leq 2}
\big\| \mathrm{D}_{x'}^{\alpha}\psi \big\|_{L^2(\Sigma_t)}^2.
\end{align*}
Substitute the last estimate into \eqref{I1b:es1} to infer
\begin{align}
\left|\mathcal{I}_{1}^b+\mathcal{I}_{2}\right|
\lesssim_K
\sum_{|\alpha|\leq 2}
\|(\mathcal{U}, \mathrm{D}_{x'}^{\alpha}\psi, \p_t \psi,\mathrm{D}_{x'}\p_t\psi)\|_{L^2(\Sigma_t)}^2
+\|W(t)\|_{L^2(\Sigma)}^2.
\label{I1b:es}
\end{align}
Let us pass the last term in \eqref{I1b:es} to the volume integral and get
\begin{align}
 \|W(t)\|_{L^2(\Sigma)}^2
% = -2\int_{\Omega}W\cdot \p_1W
\nonumber & \lesssim \boldsymbol{\epsilon} \|\p_1W(t)\|_{L^2(\Omega)}^2+ \boldsymbol{\epsilon}^{-1} \|W(t)\|_{L^2(\Omega)}^2\\
& \lesssim \boldsymbol{\epsilon} \|\p_1W(t)\|_{L^2(\Omega)}^2+ C(\boldsymbol{\epsilon}) \|W\|_{H^1(\Omega_t)}^2
\quad
\textrm{for all }\boldsymbol{\epsilon}>0.
 \label{es3b}
\end{align}

Plugging \eqref{es3a} into \eqref{es2a} with $\ell=0$, we utilize \eqref{J.cal:es}, \eqref{I1:es}--\eqref{I1a:es}, and \eqref{I1b:es}--\eqref{es3b} to derive
\begin{multline}
 \|\p_tW(t)\|_{L^2(\Omega)}^2+\| (\mathrm{D}_{x'}\p_t\psi,W)(t)\|_{L^2(\Sigma)}^2
 \lesssim_K
C(\boldsymbol{\epsilon})\|(\bm{f},W)\|_{H^1(\Omega_t)}^2 \\[1mm]
 + \boldsymbol{\epsilon}\|\p_1W(t)\|_{L^2(\Omega)}^2  + \sum_{|\alpha|\leq 2} \|(  \mathrm{D}_{x'}^{\alpha}\psi, \p_t\psi,\mathrm{D}_{x'}\p_t\psi,W)\|_{L^2(\Sigma_t)}^2
 \label{es3c}
\end{multline}
for all $\boldsymbol{\epsilon}>0.$

\subsubsection{$L^2$ estimate of $\p_1 W_{\rm nc}$}
Let us estimate the normal derivatives of the noncharacteristic variables $W_{\rm nc}:=(W_{\rm nc}^+,W_{\rm nc}^-)^{\top}$ with $W_{\rm nc}^{\pm}$ given in \eqref{W.nc:def}.
In light of the assumption \eqref{bas1b} and the continuity of the basic state $(\mathring{U},\mathring{\varphi})$,
we can find a small constant $0<\delta<1$, which depends on $\kappa$ and $K$, such that
\begin{align}
 |\mathring{H}_N^{\pm}|\geq  \frac{\kappa}{4}>0
\quad \textrm{in } \Omega_T^{\delta}:=(-\infty,T)\times \Omega^{\delta},
\label{bas2a}
\end{align}
where $\Omega^{\delta}:=\{x\in\mathbb{R}^3: 0<x_1<\delta\}$ denotes the $\delta$-neighbourhood of the boundary $\Sigma$.
Then we can define the matrices
\begin{align}
{\bm{B}}^{\pm}:=\pm \frac{1}{\mathring{H}_N^{\pm}}
\begin{pmatrix}
0 & \mathring{H}_N^{\pm} & \mathring{H}_2^{\pm} &\mathring{H}_3^{\pm} & 0 & 0& 0 & 0\\[0.5mm]
\mathring{H}_N^{\pm} & 0 & 0& 0 &0 & 0& 0 &0\\[0.5mm]
\mathring{H}_2^{\pm} & 0 & 0& 0 &-1 & 0& 0 &0\\[0.5mm]
\mathring{H}_3^{\pm} & 0 & 0& 0 &0 & -1& 0 &0\\[0.5mm]
0 & 0&  -1& 0 &0 &0 &0 & 0\\[0.5mm]
0 & 0 &0 &  -1& 0 &0 &0  & 0\\[0.5mm]
0 & 0 &0 &  0& 0 &0 &0  & 0\\[0.5mm]
\w{0} & \w{0} &\w{0} & \w{0} & \w{0} & \w{0} &\w{0} & \w{0}
\end{pmatrix}
\quad \textrm{in } \Omega_T^{\delta}.
\label{B.bm:def}
\end{align}
Thanks to the equations \eqref{ELP3a} and the decomposition \eqref{decom}--\eqref{A(1):def}, we get
\begin{multline} \label{id4a}
 \big(  \p_1 W_{\rm nc}^{\pm},\, 0,\, 0\big)^{\top}\\
 ={\bm{B}}^{\pm}
 \bigg(\bm{f}^{\pm}-{\bm{A}}_4^{\pm} W^{\pm}-\sum_{\ell=0,2,3}{\bm{A}}_{\ell}^{\pm}\p_{\ell} W^{\pm}-{\bm{A}}_{(0)}^{\pm} \p_1W^{\pm}\bigg)
\quad \textrm{in } \Omega_T^{\delta}.
\end{multline}
It follows from \eqref{bas1d3} and the second identity in \eqref{decom} that
\begin{align} \label{bas2b}
 \sup_{(t,x')\in (-\infty,T)\times\mathbb{R}^{2}}\big| {\bm{A}}_{(0)}^{\pm}(t,x_1,x')\big|
 \lesssim_K \sigma(x_1)
 \ \ \textrm{for all } x_1\geq 0,
\end{align}
where $\sigma=\sigma(x_1)$ is an increasing $C^{\infty}(\mathbb{R})$--function satisfying
\begin{align}
\label{sigma:def}
\sigma(x_1)=
\left\{\begin{aligned}
&x_1\quad &&\textrm{if } 0\leq x_1\leq 1\\
&2 \quad  &&\textrm{if } x_1\geq 4.
\end{aligned}\right.
\end{align}
By virtue of \eqref{id4a}--\eqref{bas2b}, we obtain
\begin{align}
 \|\p_1 W_{\rm nc}(t)\|_{L^2(\Omega^{\delta})}
 &\lesssim_K \|(W,\mathrm{D}_{\rm tan}W, \sigma\p_1 W,\bm{f})(t)\|_{L^2(\Omega)},
 \label{es4a}
\end{align}
where $\mathrm{D}_{\rm tan}:=(\p_t,\p_2,\p_3)^{\top}$.
Since the weight $\sigma$ vanishes on the boundary $\Sigma_T$,
we apply the operator $\sigma\p_1$ to \eqref{ELP3a} and multiply the resulting equations with $\sigma \p_1 W^{\pm}$ to derive
\begin{align}
 \|\sigma \p_1W(t)\|_{L^2(\Omega)}\lesssim_K \|(\bm{f},W )\|_{H^1(\Omega_t)}.
\label{es4b}
\end{align}
Moreover, the weight $\sigma$ is away from zero outside the boundary; more precisely, $\sigma(x_1)\geq \delta>0$ for all $x_1\geq \delta$.
Hence from the estimate \eqref{es4b} we infer
\begin{align}
 \|  \p_1W(t)\|_{L^2(\Omega\setminus \Omega^{\delta})}
 \lesssim_K
 \|(\bm{f},W )\|_{H^1(\Omega_t)},
 \label{es4c}
\end{align}
which combined \eqref{es4a}--\eqref{es4b} gives
\begin{align}
 \|\p_1 W_{\rm nc}(t)\|_{L^2(\Omega)}^2
 &\lesssim_K\|(W,\mathrm{D}_{\rm tan}W)(t)\|_{L^2(\Omega)}^2
 +\|(\bm{f},W)\|_{H^1(\Omega_t)} ^2.
 \label{es4d}
\end{align}

\subsubsection{$L^2$ estimate of $\p_1 W_{\rm c}$}\label{subsec:Wc}
It remains to control the normal derivatives of
the characteristic variables $W_{\rm c}:=(W_{\rm c}^{+},W_{\rm c}^{-})^{\top}$ with
$W_{\rm c}^{\pm}$ given in \eqref{W.nc:def}.

Since the matrices $\mathcal{C}_{\pm}(\mathring{U}^{\pm},\mathring{\varPhi}^{\pm})$
are smooth functions of $(\mathring{V},\mathrm{D}\mathring{V}, \mathrm{D}\mathring{\varPsi})$,
the equations for $W_8^{\pm}=S^{\pm}$ in \eqref{ELP2a} read as ({\it cf.}~\eqref{J.ring} and \eqref{A.bm:def})
\begin{align}
 \big(\p_t  +\mathring{w}^{\pm}_{1}\p_{1} +\mathring{v}_2^{\pm}\p_{2}+  \mathring{v}_3^{\pm}\p_{3}\big)W_8^{\pm}
 = \bm{f}_8^{\pm}+   \mathring{\rm c}_1 W \quad \textrm{in } \Omega_T,
\label{S:equ}
\end{align}
where
$
\mathring{w}^{\pm}_1
:=(\mathring{v}^{\pm}\cdot \mathring{N}^{\pm}-\p_t\mathring{\varPhi}^{\pm})/\p_1\mathring{\varPhi}^{\pm}
$
satisfy
\begin{align}
\mathring{w}^{\pm}_1 =0 \quad \textrm{on }\Sigma_{T},
 \label{w1.ring}
\end{align}
resulting from the assumption \eqref{bas1c}.
Differentiate \eqref{S:equ} with respect  to $x_1$ and use the identity \eqref{w1.ring} to deduce
\begin{align}
\|\p_1 W_8^{\pm}(t)\|_{L^2(\Omega)}^2\lesssim_K \|(\bm{f},W)\|_{H^1(\Omega_t)}^2.
\label{es5a}
\end{align}

In order to estimate the normal derivative of $W_7^{\pm}=H^{\pm}\cdot\mathring{N}^{\pm}$, we introduce the linearized divergences of the magnetic fields, that is,
\begin{align}
\xi^{\pm}:=\nabla^{\mathring{\varPhi}^{\pm}}\cdot H^{\pm},
\label{xi:def}
\end{align}
where the operators $\nabla^{\mathring{\varPhi}^{+}}$ and $\nabla^{\mathring{\varPhi}^{-}}$
are defined by \eqref{nabla.varPhi}--\eqref{differential}.
A direct computation shows
\begin{align}
\xi^{\pm}=\frac{1}{\p_1 \mathring{\varPhi}^{\pm}} \p_1 W_7^{\pm}
+\sum_{k=2,3} \bigg(
\frac{\p_1\p_k \mathring{\varPhi}^{\pm}}{ \p_1 \mathring{\varPhi}^{\pm}} H_k^{\pm}
+\p_k H_k^{\pm}
\bigg)
.
\label{xi:iden}
\end{align}
Hence it is sufficient to obtain the $L^2$ estimate of $\xi^{\pm}$.
For this purpose, we write down the equations for $H_j^{\pm}$ in \eqref{ELP2a} as follows:
\begin{align}
\p_t^{\mathring{\varPhi}^{\pm}}H_j^{\pm} + \mathring{v}^{\pm}\cdot\nabla^{\mathring{\varPhi}^{\pm}}H_j^{\pm}
- \mathring{H}^{\pm}\cdot\nabla^{\mathring{\varPhi}^{\pm}} v_j^{\pm}
+\mathring{H}_j^{\pm} \nabla^{\mathring{\varPhi}^{\pm}} \cdot v^{\pm}
= \mathring{\rm c}_1 \bm{f}+\mathring{\rm c}_1  W.
\nonumber
\end{align}
Applying the operators $\p_j^{\mathring{\varPhi}^{\pm}}$ to the last equations respectively,
we find
\begin{align}
 \big(\p_t  +\mathring{w}^{\pm}_{1}\p_{1} +\mathring{v}_2^{\pm}\p_{2}+  \mathring{v}_3^{\pm}\p_{3}\big)\xi^{\pm}
=\mathring{\rm c}_1 \mathrm{D}\bm{f}+ \mathring{\rm c}_2 \bm{f}+\mathring{\rm c}_1 \mathrm{D}W + \mathring{\rm c}_2 W,
\label{xi:equ}
\end{align}
which together with \eqref{w1.ring} leads to
\begin{align}
 \|\xi^{\pm}(t)\|_{L^2(\Omega)}\lesssim_K \|(\bm{f},W)\|_{H^1(\Omega_t)}.
\nonumber
\end{align}
Combine the last estimate with \eqref{xi:iden} to deduce
\begin{align}
 \|\p_1 W_7^{\pm}(t)\|_{L^2(\Omega)}^2\lesssim_K\|(W,\mathrm{D}_{\rm tan}W)(t)\|_{L^2(\Omega)}^2
 +\|(\bm{f},W)\|_{H^1(\Omega_t)}^2 .
\label{es5b}
\end{align}

\subsubsection{Conclusion}
Taking a suitable linear combination of \eqref{es1d}, \eqref{es2c}, \eqref{es3c}, \eqref{es4d}, \eqref{es5a}, and \eqref{es5b}, we choose $\boldsymbol{\epsilon}>0$ sufficiently small to derive
\begin{multline}
% \nonumber &
\|(W,\mathrm{D}W)(t)\|_{L^2(\Omega)}^2+
\sum_{|\alpha|\leq 2}\|(\mathrm{D}_{x'}^{\alpha}\psi, \mathrm{D}_{x'}\p_t\psi,W)(t)\|_{L^2(\Sigma)}^2
\\
 \qquad \lesssim_K
 \|(\bm{f},W)\|_{H^1(\Omega_t)}^2+
\sum_{|\alpha|\leq 2}\|(\mathrm{D}_{x'}^{\alpha}\psi, \p_t \psi,\mathrm{D}_{x'}\p_t\psi,W)\|_{L^2(\Sigma_t)}^2.
\label{es6a}
\end{multline}
Note from the boundary condition \eqref{ELP3d} that
\begin{align}
 \|\p_t\psi\|_{L^2(\Sigma_t)}
 \lesssim_K\|(W_2^+,\psi,\mathrm{D}_{x'}\psi)\|_{L^2(\Sigma_t)} .
 \label{es6b}
\end{align}
Substitute \eqref{es6b} into \eqref{es6a} and utilize Gr\"{o}nwall's inequality to obtain
\begin{align}
\|(W,\mathrm{D}W)(t)\|_{L^2(\Omega)}^2+
\sum_{|\alpha|\leq 2}\|(\mathrm{D}_{x'}^{\alpha}\psi, \mathrm{D}_{x'}\p_t\psi,W)(t)\|_{L^2(\Sigma)}^2
\lesssim_K
\|\bm{f}\|_{H^1(\Omega_t)}^2,
\nonumber
\end{align}
which combined with \eqref{es6b} implies the desired $H^1$ estimate
\begin{align}
\|W\|_{H^1(\Omega_t)}
 +\|W\|_{L^2(\Sigma_t)}  + \|(\psi, \mathrm{D}_{x'}\psi)\|_{H^1(\Sigma_t)}
\lesssim_K
 \|\bm{f}\|_{H^1(\Omega_t)}
\label{es6c}
\end{align}
for all $0\leq t\leq T$.

\subsection{Well-posedness of the $\varepsilon$--Regularization}

For the linearized problem \eqref{ELP3}, we introduce the $\varepsilon$--regularization
\begin{subequations} \label{Reg}
 \begin{alignat}{3}
  \label{Reg.a}
  &{{\bf L}}_{\varepsilon}^{\pm}W^{\pm}:={{\bf L}} ^{\pm}W^{\pm} -\varepsilon\bm{J}_{\pm}\p_1 W^{\pm} =\bm{f}^{\pm}
  &\quad &\textnormal{in }\Omega_T,\\
  \label{Reg.b}
  &
  [W_1]=\mathfrak{s}\mathrm{D}_{x'}\cdot
  \bigg(\dfrac{\mathrm{D}_{x'}\psi}{|\mathring{N}|}-
  \dfrac{\mathrm{D}_{x'}\mathring{\varphi}\cdot\mathrm{D}_{x'}\psi}{|\mathring{N}|^3}\mathrm{D}_{x'}\mathring{\varphi}
  \bigg)+ \mathring{a}_1 \psi
  &\quad &\textnormal{on }\Sigma_T,\\
  \label{Reg.c}
&[W_i]=\mathring{a}_i\psi \qquad \textrm{for }i=2,\ldots,6,
&\qquad &\textnormal{on }\Sigma_T,\\[2mm]
\label{Reg.d}
&W_2^+=(\p_t+\mathring{v}_2^+\p_2+\mathring{v}_3^+\p_3)\psi+\mathring{a}_7 \psi
+\varepsilon(\p_2^4+\p_3^4)\psi
&\qquad &\textnormal{on }\Sigma_T,\\[2mm]
  \label{Reg.e}
&(W,\psi)=0 &\quad &\textnormal{if }t<0,
 \end{alignat}
\end{subequations}
where $\varepsilon>0$ denotes the small parameter,
$W:=(W^+,W^-)^{\top}$,
the operators ${{\bf L}} ^{\pm}$ and
the scalars $\mathring{a}_1,\ldots,\mathring{a}_7$ are given in \eqref{ELP3a}, \eqref{a.ring:def}, and \eqref{A.bm:def}.
To derive a uniform--in--$\varepsilon$ estimate in $H^1$ for  solutions $W$ of the problem \eqref{Reg}, we design the following symmetric matrices:
\begin{alignat}{3}  \label{J.bm:def}
 \bm{J}_{+}:=
 \mathrm{diag}\,\big(0,\,1,\,0,\,0,\,\bm{J}_{+}^H,\,0\big),
 \quad
 \bm{J}_{-}:=\mathrm{diag}\,\big(1,\,1,\,1,\,1,\,0,\,0,\,0,\,0\big),
\end{alignat}
where the matrix $\bm{J}_{+}^H$ is related with $\mathring{J}_{+}^H$ ({\it cf.}~\eqref{JH.ring:def}) through
\setlength{\arraycolsep}{8pt}
\begin{align}
 \bm{J}_{+}^H:=(\mathring{J}_{+}^H)^{\top} \mathring{J}_{+}^H=
 \frac{1}{{|\mathring{N}^{+}|^2}}
 \begin{pmatrix}
  1+(\p_3\mathring{\varPhi}^{+})^2
  &  -\p_2\mathring{\varPhi}^{+} \p_3\mathring{\varPhi}^{+}   & 0 \\[1mm]
  - \p_2\mathring{\varPhi}^{+} \p_3\mathring{\varPhi}^{+}
  &1+(\p_2\mathring{\varPhi}^{+})^2  & 0   \\[1mm]
  0&  0   & 1
 \end{pmatrix}.
\label{J.bm.H:def}
\end{align}
\setlength{\arraycolsep}{4pt}

It is important to point out that there is some $\varepsilon_0>0$ depending on $K$, such that if $0<\varepsilon\leq \varepsilon_0$,
then the boundary matrix for the problem \eqref{Reg}, {\it i.e.},
$$\mathrm{diag}\,\left(\varepsilon\bm{J}_{+} -\bm{A}_1^+,\,
\varepsilon\bm{J}_{-}-\bm{A}_1^-\right),$$
has six negative eigenvalues on the boundary $\Sigma_{T}$.
As analyzed for \eqref{ELP3}, the hyperbolic problem \eqref{Reg} has a correct number of boundary conditions provided $0<\varepsilon\leq \varepsilon_0$.

In this subsection, we are going to deduce the $\varepsilon$-dependent $L^2$ {\it a priori} estimates for the regularized problem \eqref{Reg} and its dual problem (see \S \ref{sec.dual1}) for any fixed parameter $\varepsilon\in(0,\varepsilon_1)$ with $\varepsilon_1\leq \varepsilon_0$ small enough,
which allows us to solve the problem \eqref{Reg} in $L^2$ by the duality argument.

\subsubsection{$L^2$ \textit{a priori} estimate}
Take the scalar product of \eqref{Reg.a} with $W^{\pm}$ to obtain
\begin{align} \label{Res1a}
\sum_{\pm}\int_{\Omega} {\bm{A}}_0^{\pm}W^{\pm}\cdot W^{\pm} \d x
+\int_{\Sigma_t} \mathcal{T}_{\rm b}^{\varepsilon} (W)
 \lesssim_K \|(\bm{f},W)\|_{L^2(\Omega_t)}^2,
\end{align}
where we denote
\begin{align}
\label{Tb.bf}
\mathcal{T}_{\rm b}^{\varepsilon} (U):=\sum_{\pm}(\varepsilon\bm{J}_{\pm} -{\bm{A}}_1^{\pm})U^{\pm}\cdot U^{\pm}
\ \ \textrm{for any } U:=(U^+,U^-)^{\top}\in\mathbb{R}^{16}.
\end{align}
As for \eqref{id1a}, we get from \eqref{J.bm:def}--\eqref{J.bm.H:def} and \eqref{Reg.c}--\eqref{Reg.d} that
\begin{align}
\mathcal{T}_{\rm b}^{\varepsilon} (W)
&=\varepsilon\sum_{\pm} \bm{J}_{\pm}W^{\pm}\cdot W^{\pm}
-2[W_1]W_2^+
+[(W_2,\ldots,W_6)]\mathring{\rm c}_0\mathcal{U}
\nonumber\\
&=\varepsilon |\mathcal{U}_{\rm reg}|^2
-2[W_1]\mathring{\rm B}\psi -2\varepsilon\sum_{k=2,3}[W_1]\p_k^4\psi
+\mathring{\rm c}_1 \psi \mathcal{U}
\quad \textrm{on }\Sigma_T,
\label{Rid1a}
\end{align}
with $\mathring{\rm B}$ and $\mathcal{U}$ given in \eqref{ELP3d} and \eqref{U.cal:def},
where
\begin{align}
 &
 \mathcal{U}_{\rm reg}:=
 \mathrm{diag}\,\big(1,\,1,\,1,\,1,\,1,\,\mathring{J}_{+}^H\big)
\begin{pmatrix}
\mathcal{U} \\ W_7^+
\end{pmatrix} \in\mathbb{R}^8.
 \label{U.cal.r:def}
\end{align}
Since the matrix $\mathring{J}^H_+$ is invertible and smooth in $\mathrm{D}_{x'}\mathring{\varPsi}^{+}$ ({\it cf.}~\eqref{JH.ring:def}), we have
\begin{align}
|\mathcal{U}|+|W_7^+|\lesssim_K |\mathcal{U}_{\rm reg}|\lesssim_K |\mathcal{U}|+|W_7^+|.
\label{U.Ur:id}
\end{align}
For ${k}=2,3$, it follows from \eqref{Reg.b} that
\begin{align}
 -\int_{\Sigma_t} [W_1] \p_{k}^4\psi
 &=\mathfrak{s}\int_{\Sigma_t}\p_{k}^2
 \bigg(\dfrac{\mathrm{D}_{x'}\psi}{|\mathring{N}|}-
 \dfrac{\mathrm{D}_{x'}\mathring{\varphi}\cdot\mathrm{D}_{x'}\psi}{|\mathring{N}|^3}\mathrm{D}_{x'}\mathring{\varphi}
 \bigg)
 \cdot \p_{k}^2\mathrm{D}_{x'}\psi
- \int_{\Sigma_t}  \mathring{a}_1 \psi \p_{k}^4\psi
 \nonumber
 \\
 &  \geq
 \mathfrak{s}\int_{\Sigma_t} \dfrac{|\p_{k}^2\mathrm{D}_{x'}\psi|^2}{|\mathring{N}|^3}
 -\int_{\Sigma_t}  \left|[\p_{k}^2,\mathring{\rm c}_0]\mathrm{D}_{x'}\psi \cdot \p_{k}^2\mathrm{D}_{x'}\psi
 +\p_{k}(\mathring{a}_1 \psi )\p_{k}^3\psi\right|
 \nonumber \\
 &  \geq
 \frac{\mathfrak{s}}{2}\int_{\Sigma_t} \dfrac{|\p_{k}^2\mathrm{D}_{x'}\psi|^2}{|\mathring{N}|^3}
 -C(K)\sum_{|\alpha|\leq 2}\|\mathrm{D}_{x'}^{\alpha}\psi\|_{L^2(\Sigma_t)}^2.
 \label{Res1b}
\end{align}
Substituting \eqref{Rid1a} into \eqref{Res1a},
we utilize \eqref{id1b} and \eqref{U.Ur:id}--\eqref{Res1b} to infer
\begin{multline}
 \|W(t)\|_{L^2(\Omega)}^2+\|\mathrm{D}_{x'}\psi(t)\|_{L^2(\Sigma)}^2
 +\varepsilon\sum_{k=2,3}\|(\mathcal{U},W_7^+,\p_k^2\mathrm{D}_{x'}\psi)\|_{L^2(\Sigma_t)}^2  \\
 \lesssim_K
 \|(\bm{f},W)\|_{L^2(\Omega_t)}^2+\|\psi(t)\|_{L^2(\Sigma)}^2
+\big\|\psi \mathcal{U}\big\|_{L^1(\Sigma_t)}   \\[2.5mm]
+\|(\psi, \mathrm{D}_{x'}\psi)\|_{L^2(\Sigma_t)}^2
 +\varepsilon\|\mathrm{D}^2_{x'} \psi\|_{L^2(\Sigma_t)}^2 ,
 \label{Res1c}
\end{multline}
where $\mathcal{U}$ is the vector defined by \eqref{U.cal:def}.

From the boundary condition \eqref{Reg.d}, we employ the standard argument of the energy method to deduce
\begin{align}
\nonumber \|\psi(t)\|_{L^2(\Sigma)}^2
 + \varepsilon \sum_{k=2,3}\|\p_k^2\psi\|_{L^2(\Sigma_t)}^2
&\lesssim_K \|\psi\|_{L^2(\Sigma_t)}^2+\|\psi W_2^+\|_{L^1(\Sigma_t)}\\
 &\lesssim_K  \boldsymbol{\epsilon} \varepsilon \|W_2^+\|_{L^2(\Sigma_t)}^2
 +C(\boldsymbol{\epsilon} \varepsilon)\|\psi\|_{L^2(\Sigma_t)}^2
 \label{Res1d}
\end{align}
for all $\boldsymbol{\epsilon}>0$, where $C(\boldsymbol{\epsilon}\varepsilon)\to +\infty$ as $\boldsymbol{\epsilon}\varepsilon\to 0$.
Use integration by parts to get
\begin{align}
 \|\mathrm{D}^2_{x'}\mathrm{D}_{x'}^{\alpha}\psi\|_{L^2(\Sigma)}^2
 \lesssim  \sum_{k=2,3}\|\p_k^2\mathrm{D}_{x'}^{\alpha}\psi\|_{L^2(\Sigma)}^2
\quad \textrm{for any }\alpha\in\mathbb{N}^2.
 \label{Res1e}
\end{align}
Then it follows from  \eqref{Res1d} and \eqref{Res1e} with $\alpha=0$ that
\begin{align}
\|\psi(t)\|_{L^2(\Sigma)}^2
 + \varepsilon \|\mathrm{D}^2_{x'} \psi\|_{L^2(\Sigma_t)}^2
   \lesssim_K  \boldsymbol{\epsilon} \varepsilon \|W_2^+\|_{L^2(\Sigma_t)}^2
 +C(\boldsymbol{\epsilon} \varepsilon)\|\psi\|_{L^2(\Sigma_t)}^2.
 \label{Res1d2}
\end{align}
Plugging \eqref{Res1d2} into \eqref{Res1c}, taking $\boldsymbol{\epsilon}>0$ small enough, and making use of \eqref{Res1e} with $|\alpha|=1$ imply
\begin{multline}
 \|W(t)\|_{L^2(\Omega)}^2+\|(\psi,\mathrm{D}_{x'}\psi)(t)\|_{L^2(\Sigma)}^2
 +\varepsilon \|(\mathcal{U}, W_7^+,\mathrm{D}^2_{x'} \psi,\mathrm{D}^3_{x'} \psi)\|_{L^2(\Sigma_t)}^2
 \nonumber\\
 \lesssim_K
 \|(\bm{f},W)\|_{L^2(\Omega_t)}^2
 +C(\varepsilon)\|(\psi, \mathrm{D}_{x'}\psi)\|_{L^2(\Sigma_t)}^2.
\end{multline}
Apply Gr\"{o}nwall's inequality and take into account the boundary conditions \eqref{Reg.b}--\eqref{Reg.c} to derive
\begin{multline}
 \|W(t)\|_{L^2(\Omega)}^2
 +\|(\psi,\mathrm{D}_{x'}\psi)(t)\|_{L^2(\Sigma)}^2
 \\
  +
  \|(W_{\rm nc},W_7^+,\mathrm{D}^2_{x'}\psi,\mathrm{D}^3_{x'} \psi)\|_{L^2(\Sigma_t)}^2 \leq
 C(K,\varepsilon)   \|\bm{f}\|_{L^2(\Omega_t)}^2,
 \label{Res1f}
\end{multline}
where $W_{\rm nc}:=(W_1^{+},\ldots,W_6^{+},W_1^{-},\ldots,W_6^{-})^{\top}$ and
$C(K,\varepsilon)\to +\infty$ as $\varepsilon\to 0$.

Furthermore, we apply the operator $\p_{i}^2$, for ${i}=2,3$, to the boundary condition \eqref{Reg.d},
multiply the resulting equation with $\p_{i}^2\psi$, and employ Cauchy's inequality to infer
\begin{align}
 \|\p_{i}^2\psi(t)\|_{L^2(\Sigma)}^2+  \sum_{k=2,3}\|\p_k^2\p_{i}^2\psi\|_{L^2(\Sigma_t)}^2   \leq C(K,\varepsilon)  \sum_{|\alpha|\leq 2}\|(W_2^+,\mathrm{D}_{x'}^{\alpha}\psi)\|_{L^2(\Sigma_t)}^2.
 \nonumber
\end{align}
Taking a suitable combination of the last estimate with \eqref{Res1e} and \eqref{Res1f} yields
\begin{multline}
 \|W(t)\|_{L^2(\Omega)}^2+ \|(\psi,\mathrm{D}_{x'}\psi\mathrm{D}^2_{x'}\psi)(t)\|_{L^2(\Sigma)}^2
 \\
 + \|(W_{\rm nc},W_7^+,\mathrm{D}^2_{x'}\psi,\mathrm{D}^3_{x'}\psi,\mathrm{D}^4_{x'}\psi)\|_{L^2(\Sigma_t)}^2
 \leq
 C(K,\varepsilon)   \|\bm{f}\|_{L^2(\Omega_t)}^2.
\label{Res1h}
\end{multline}
Thanks to the condition \eqref{Reg.d} and the inequality \eqref{Res1h}, we obtain the following $\varepsilon$-dependent $L^2$ {\it a priori} estimate for the regularized problem \eqref{Reg}:
\begin{align}
 \|W\|_{L^2(\Omega_{t})}+\sum_{ |\alpha|\leq 4}\|(W_{\rm nc},W_7^+,\mathrm{D}_{x'}^{\alpha} \psi,\p_t\psi)\|_{L^2(\Sigma_{t})}
 \leq
 C(K,\varepsilon)   \|\bm{f}\|_{L^2(\Omega_{t})}
\label{Res1h2}
\end{align}
for all $0\leq t\leq T$,
where $C(K,\varepsilon)\to +\infty$ as $\varepsilon\to 0$.

\subsubsection{Existence of solutions} \label{sec.dual1}
We shall construct the solutions of the regularized problem \eqref{Reg} by means of the duality argument.
To this end, it suffices to show a suitable $L^2$ {\it a priori} estimate for solutions $\widetilde{W}:=(\widetilde{W}^+,\widetilde{W}^-)^{\top}$ of the following dual problem of \eqref{Reg}:
\begin{subequations} \label{dual1}
 \begin{alignat}{3}
\label{dual1a}
 &\widetilde{{\bf L}}_{\varepsilon}^{\pm}\widetilde{W}^{\pm}=\widetilde{\bm{f}}^{\pm}
 &\quad &\textnormal{in }\Omega_T,\\
\label{dual1b}
&\varepsilon \bm{J}_{+}^H\begin{pmatrix}
\widetilde{W}_5^+\\[1mm]
\widetilde{W}_6^+\\[1mm]
\widetilde{W}_7^+
\end{pmatrix}
=\begin{pmatrix}
\mathring{H}_2^+\big[\widetilde{W}_2\big]
-\mathring{H}_N^+\big[\widetilde{W}_3\big]\\[1mm]
\mathring{H}_3^+\big[\widetilde{W}_2\big]
-\mathring{H}_N^+\big[\widetilde{W}_4\big]\\[1mm]
0
\end{pmatrix}&\quad &\textnormal{on }\Sigma_T,\\
  \label{dual1c} &\varepsilon\widetilde{W}_1^-=\big[\widetilde{W}_2\big],\quad \varepsilon\widetilde{W}_3^-=-\mathring{H}_N^+\big[\widetilde{W}_5\big],\quad
  \varepsilon\widetilde{W}_4^-=-\mathring{H}_N^+\big[\widetilde{W}_6\big]
  &\quad &\textnormal{on }\Sigma_T,\\
 &\p_tw+\p_2(\mathring{v}_2^+w)+\p_3(\mathring{v}_3^+w)-\varepsilon(\p_2^4+\p_3^4)w
 -\mathring{a}_7 w&\quad &
 \nonumber\\
 \label{dual1d}
 &\ \ -\mathfrak{s}\mathrm{D}_{x'}\cdot
 \bigg(\dfrac{\mathrm{D}_{x'}\widetilde{W}_2^+}{|\mathring{N}|}-
 \dfrac{\mathrm{D}_{x'}\mathring{\varphi}\cdot\mathrm{D}_{x'}\widetilde{W}_2^+}{|\mathring{N}|^3}\mathrm{D}_{x'}\mathring{\varphi}\bigg) +\mathcal{T}_{\rm dual}=0
 &\quad &\textnormal{on }\Sigma_T, \\
 \label{dual1e}
 &\widetilde{W} =0&\quad &\textnormal{if }t>T,
 \end{alignat}
\end{subequations}
where $\widetilde{{\bf L}}_{\varepsilon}^{\pm}$ are the formal adjoint operators of ${{\bf L}}_{\varepsilon}^{\pm}$ ({\it cf.}~\eqref{Reg.a}), {\it i.e.},
\begin{align}
 \widetilde{{\bf L}}_{\varepsilon}^{\pm}:=-\sum_{i=0}^{3}\bm{A}_i^{\pm}\p_i
 +\varepsilon\bm{J}_{\pm}\p_1-\sum_{i=0}^{3}\p_i\bm{A}_i^{\pm}
 +\varepsilon\p_1\bm{J}_{\pm}
 +(\bm{A}_4^{\pm})^{\top},
 \nonumber
\end{align}
the matrix $\bm{J}_{+}^H$ is defined by \eqref{J.bm.H:def}, and
\begin{align}
\label{w:def}
 w:=\,&\big[\widetilde{W}_1 \big]+\mathring{H}_2^+\big[\widetilde{W}_5 \big]
 +\mathring{H}_3^+\big[\widetilde{W}_6 \big]-\varepsilon \widetilde{W}_2^+ -\varepsilon \widetilde{W}_2^- ,
\\
\mathcal{T}_{\rm dual}:=\,&
- \big(\mathring{H}_2^+ \mathring{a}_5 +\mathring{H}_3^+ \mathring{a}_6 \big)\widetilde{W}_2^-
- \mathring{a}_2 \big(\widetilde{W}_1^- +\varepsilon \widetilde{W}_2^-
+\mathring{H}_2^+\widetilde{W}_5^- +\mathring{H}_3^+\widetilde{W}_6^-\big)
\nonumber\\
&  - \mathring{a}_1 \widetilde{W}_2^+
+ \mathring{H}_N^+\big( \mathring{a}_3\widetilde{W}_5^+  +\mathring{a}_4\widetilde{W}_6^+
+\mathring{a}_5\widetilde{W}_3^- +\mathring{a}_6\widetilde{W}_4^-\big).
\label{T.cal.dual}
\end{align}
The boundary conditions \eqref{dual1b}--\eqref{dual1d} are imposed to ensure that
\begin{align*}
&\sum_{\pm}\int_{\Omega_{T}}\left({{\bf L}}_{\varepsilon}^{\pm}W^{\pm}\cdot \widetilde{W}^{\pm} -W\cdot \widetilde{{\bf L}}_{\varepsilon}^{\pm} \widetilde{W}^{\pm}\right)
 =\sum_{\pm}\int_{\Sigma_{T}}\left(\varepsilon\bm{J}_{\pm}-\bm{A}_1^{\pm} \right)W^{\pm}\cdot \widetilde{W}^{\pm}=0,
\end{align*}
where we have used the conditions \eqref{Reg.b}--\eqref{Reg.e} and \eqref{dual1e}.

Let us define
$\widetilde{W}_{\flat}^{\pm}(t,x):=\widetilde{W}^{\pm}(T-t,x)$ and
$\widetilde{\bm{f}}_{\flat}^{\pm}(t,x):= \widetilde{\bm{f}}^{\pm}(T-t,x).
$
Then the dual problem \eqref{dual1} is reduced to
\begin{subequations} \label{dual2}
	\begin{alignat}{3}
		\nonumber
		&\bm{A}_0^{\pm}\p_t\widetilde{W}^{\pm}-\sum_{i=1}^{3}\bm{A}_i^{\pm}\p_i \widetilde{W}^{\pm}
		+\p_t\bm{A}_0^{\pm}\widetilde{W}^{\pm}
		-\sum_{i=1}^{3}\p_i\bm{A}_i^{\pm}\widetilde{W}^{\pm}
		&\quad & \\
		& \label{dual2a} \qquad\qquad\quad +\varepsilon\bm{J}_{\pm}\p_1\widetilde{W}^{\pm}
		+\varepsilon\p_1\bm{J}_{\pm}\widetilde{W}^{\pm}
		+(\bm{A}_4^{\pm})^{\top}\widetilde{W}^{\pm}=\widetilde{\bm{f}}^{\pm}
		&\quad &\textnormal{in }\Omega_T,\\[0.5mm]
		\label{dual2b}&\varepsilon \bm{J}_{+}^H\begin{pmatrix}
			\widetilde{W}_5^+\\[1mm]
			\widetilde{W}_6^+\\[1mm]
			\widetilde{W}_7^+
		\end{pmatrix}
		=\begin{pmatrix}
			\mathring{H}_2^+\big[\widetilde{W}_2\big]
			-\mathring{H}_N^+\big[\widetilde{W}_3\big]\\[1mm]
			\mathring{H}_3^+\big[\widetilde{W}_2\big]
			-\mathring{H}_N^+\big[\widetilde{W}_4\big]\\[1mm]
			0
		\end{pmatrix}&\quad &\textnormal{on }\Sigma_T,\\[0.5mm]
		\label{dual2c} &\varepsilon\widetilde{W}_1^-=\big[\widetilde{W}_2\big],\quad \varepsilon\widetilde{W}_3^-=-\mathring{H}_N^+\big[\widetilde{W}_5\big],\quad
		\varepsilon\widetilde{W}_4^-=-\mathring{H}_N^+\big[\widetilde{W}_6\big]
		&\quad &\textnormal{on }\Sigma_T,\\[0.5mm]
		&\p_tw-\p_2(\mathring{v}_2^+w)-\p_3(\mathring{v}_3^+w)+\varepsilon(\p_2^4+\p_3^4)w
		+\mathring{a}_7 w&\quad &
		\nonumber\\
		&\ \ +\mathfrak{s}\mathrm{D}_{x'}\cdot
		\bigg(\dfrac{\mathrm{D}_{x'}\widetilde{W}_2^+}{|\mathring{N}|}-
		\dfrac{\mathrm{D}_{x'}\mathring{\varphi}\cdot\mathrm{D}_{x'}\widetilde{W}_2^+}{|\mathring{N}|^3}\mathrm{D}_{x'}\mathring{\varphi}\bigg) -\mathcal{T}_{\rm dual}=0
		&\quad &\textnormal{on }\Sigma_T, \label{dual2d}\\[0.5mm]
		\label{dual2e}
		&\widetilde{W} =0&\quad &\textnormal{if }t<0,
	\end{alignat}
\end{subequations}
where the subscript ``$\flat$'' has been dropped for notational simplicity, and the terms $w$, $\mathcal{T}_{\rm dual}$ are defined by \eqref{w:def}--\eqref{T.cal.dual}.

Take the scalar product of \eqref{dual2a} with $\widetilde{W}^{\pm}$ to obtain
\begin{align}
\sum_{\pm}\int_{\Omega} {\bm{A}}_0^{\pm}\widetilde{W}^{\pm}\cdot \widetilde{W}^{\pm} \d x
 -\int_{\Sigma_{t}}
\mathcal{T}_{\rm b}^{\varepsilon}(\widetilde{W})
 \lesssim_K \big\|\big(\widetilde{\bm{f}},\widetilde{W}\big)\big\|_{L^2(\Omega_{t})}^2,
\label{Res2a}
\end{align}
where $\mathcal{T}_{\rm b}^{\varepsilon}$ is defined by \eqref{Tb.bf}.
Using \eqref{decom}--\eqref{A(1):def} and \eqref{J.bm:def}--\eqref{J.bm.H:def} yields
\begin{align}
\nonumber
\mathcal{T}_{\rm b}^{\varepsilon}(\widetilde{W})=\,&
 \varepsilon \big|\widetilde{\mathcal{U}}_{\rm reg}\big|^2-2\big[\widetilde{W}_2 \big(\widetilde{W}_1+\mathring{H}_2\widetilde{W}_5+\mathring{H}_3\widetilde{W}_6 \big)\big]\\
& +2\mathring{H}_N^+\big[\widetilde{W}_3\widetilde{W}_5+\widetilde{W}_4\widetilde{W}_6\big]
\qquad \textrm{ on }\Sigma_T,
 \nonumber
\end{align}
where
\begin{align}
 \widetilde{\mathcal{U}}_{\rm reg}:=\big(\widetilde{W}_2^+,\big(\widetilde{W}_5^+,\widetilde{W}_6^+,\widetilde{W}_7^+\big)\mathring{J}^H_+,\widetilde{W}_1^-,\widetilde{W}_2^-,\widetilde{W}_3^-,\widetilde{W}_4^-\big)^{\top}\in\mathbb{R}^8.\label{U.cal.t:def}
\end{align}
Thanks to \eqref{w:def} and \eqref{dual2b}--\eqref{dual2c}, we compute
\begin{align}
-\mathcal{T}_{\rm b}^{\varepsilon}(\widetilde{W})
=\,&\varepsilon \big|\widetilde{\mathcal{U}}_{\rm reg}\big|^2
+2 \widetilde{W}_2^-  w-2 \varepsilon \big[\widetilde{W}_2\big]\widetilde{W}_2^+
+2\big[\widetilde{W}_2\big]\big[\widetilde{W}_1\big]
\nonumber\\
=\,&
\varepsilon \big|\widetilde{\mathcal{U}}_{\rm reg}\big|^2
+2 \widetilde{W}_2^+  w
+2 \varepsilon^2  \frac{\widetilde{W}_1^-}{\mathring{H}_N^+}
\big(\mathring{H}_2^+ \widetilde{W}_3^- +  \mathring{H}_3^+ \widetilde{W}_4^-
+\mathring{H}_N^+ \widetilde{W}_2^-  \big)
\
\textrm{on }\Sigma_T.
\nonumber
\end{align}
Plug the last identity into \eqref{Res2a} and choose $\varepsilon>0$ sufficiently small to obtain
\begin{align}
\big\|\widetilde{W}(t)\big\|_{L^2(\Omega)}^2
+\varepsilon \big\|\widetilde{\mathcal{U}}_{\rm reg}\big\|_{L^2(\Sigma_t)}^2
\lesssim_K \big\|\big(\widetilde{\bm{f}},\widetilde{W}\big)\big\|_{L^2(\Omega_{t})}^2
+\varepsilon^{-1}\|w\|_{L^2(\Sigma_t)}^2.
\label{Res2b}
\end{align}
Noting from \eqref{T.cal.dual} and \eqref{dual2c} that
$\mathcal{T}_{\rm dual}=\mathring{\rm c}_1 \widetilde{\mathcal{U}}^{r}$ on $\Sigma_{{T}}$,
we multiply the boundary condition \eqref{dual2d} with $w$ to get
\begin{align}
 &\| w(t)\|_{L^2(\Sigma)}^2+\varepsilon\sum_{k=2,3}\| \p_k^2 w\|_{L^2(\Sigma_{t})}^2
 \nonumber \\
 &\quad
 \lesssim_K \big\|\big( w,\widetilde{\mathcal{U}}_{\rm reg}\big)\big\|_{L^2(\Sigma_{t})}^2
 +\left| \int_{\Sigma_{t}} \widetilde{W}^+_2 \mathrm{D}_{x'}\cdot
 \bigg(\dfrac{\mathrm{D}_{x'} w}{|\mathring{N}|}-
 \dfrac{\mathrm{D}_{x'}\mathring{\varphi}\cdot\mathrm{D}_{x'} w}{|\mathring{N}|^3}\mathrm{D}_{x'}\mathring{\varphi}
 \bigg) \right|
 \nonumber\\[1.5mm]
 &\quad
\lesssim_K  \boldsymbol{\epsilon}\varepsilon\sum_{ |\alpha|\leq 2}\big\|\mathrm{D}_{x'}^{\alpha} w\big\|_{L^2(\Sigma_{t})}^2
 +C(\boldsymbol{\epsilon}\varepsilon) \big\|\big( w,\widetilde{\mathcal{U}}_{\rm reg}\big)\big\|_{L^2(\Sigma_{t})}^2
 \ \ \textrm{for all }\boldsymbol{\epsilon}>0.
 \label{Res2c}
\end{align}
Substitute the estimates
\begin{align}
 \|\mathrm{D}^2_{x'}w\|_{L^2(\Sigma_{t})}^2
 \lesssim  \|(\p_2^2w,\p_3^2w)\|_{L^2(\Sigma_{t})}^2,
 \quad
 \|\mathrm{D}_{x'}w\|_{L^2(\Sigma_{t})}^2
 \lesssim
 \|(w,\mathrm{D}^2_{x'}w)\|_{L^2(\Sigma_{t})}^2
 \nonumber
\end{align}
into \eqref{Res2c} and choose $\boldsymbol{\epsilon}>0$ sufficiently small to derive
\begin{align}
 \|w(t)\|_{L^2(\Sigma)}^2+ \sum_{   |\alpha|\leq 2}\|\mathrm{D}_{x'}^{\alpha}w\|_{L^2(\Sigma_{t})}^2
 \leq  C(K,\varepsilon) \big\|\big( w,\widetilde{\mathcal{U}}_{\rm reg}\big)\big\|_{L^2(\Sigma_{t})}^2.
\nonumber
\end{align}
Then it follows by combining the last estimate with \eqref{Res2b} and applying Gr\"{o}nwall's inequality that
\begin{align}
 \|\widetilde{W}(t)\|_{L^2(\Omega)}^2
 +\|w(t)\|_{L^2(\Sigma)}^2
 +\sum_{  |\alpha|\leq 2}\|(\widetilde{\mathcal{U}}^{r},\mathrm{D}_{x'}^{\alpha}w)\|_{L^2(\Sigma_{t})}^2
 \nonumber
 \leq
 C(K,\varepsilon)
 \big\|\widetilde{\bm{f}}\big\|_{L^2(\Omega_{t})}^2,
\end{align}
for some positive constant $C(K,\varepsilon)\to +\infty$ as $\varepsilon\to 0$,
where $\widetilde{\mathcal{U}}^{r}$ is the vector given by \eqref{U.cal.t:def}.
In view of \eqref{dual1b}--\eqref{dual1c} and \eqref{w:def},
we deduce the following $L^2$ estimate for the dual problem \eqref{dual1}:
\begin{align}
 \|\widetilde{W}\|_{L^2(\Omega_{{t}})}
 +\sum_{  |\alpha|\leq 2}\|(\widetilde{W}_{\rm nc},\widetilde{W}_7^+,\mathrm{D}_{x'}^{\alpha}{w})\|_{L^2(\Sigma_{{t}})}
 \leq
 C(K,\varepsilon)
 \big\|\widetilde{\bm{f}}\big\|_{L^2(\Omega_{{t}})},
 \label{Res2d}
\end{align}
for all $0\leq t\leq T$,
where
$\widetilde{W}_{\rm nc} :=(\widetilde{W}_1^{+},\ldots,\widetilde{W}_6^{+},\widetilde{W}_1^{-},\ldots,\widetilde{W}_6^{-})^{\top}$.

Having the $L^2$ estimates \eqref{Res1h2} and \eqref{Res2d} in hand, we can prove the existence and uniqueness of a weak solution $W^{\varepsilon}\in L^2(\Omega_T)$ to the problem \eqref{Reg} with  $W_{\rm nc}^{\varepsilon}|_{x_1=0}\in L^2(\Sigma_T)$ for any {\it fixed} and sufficiently small parameter $\varepsilon>0$ by the classical duality argument in {\sc Chazarain--Piriou} \cite{CP82MR0678605}. %\cite{LP60MR0118949}.

Then we shall consider \eqref{Reg.d} as a fourth-order parabolic equation for $\psi $ with given source term $W_2^{+\varepsilon}|_{x_1=0}\in L^2(\Sigma_T)$ and zero initial data.
Referring to \cite[Theorem 5.2]{CP82MR0678605},  we can conclude that the Cauchy problem for this parabolic equation has a unique solution
$\psi^{\varepsilon}\in C ([0,T],H^4(\mathbb{R}^2))\bigcap C^1 ([0,T],L^2(\mathbb{R}^2)),$
implying $\psi^{\varepsilon}\in L^2((-\infty ,T],H^4(\mathbb{R}^2))$ and $\partial_t\psi^{\varepsilon}\in L^2(\Sigma_T)$.
As a matter of fact, we have already derived the {\it a priori} estimate for solutions $\psi^{\varepsilon}$ of this Cauchy problem in \eqref{Res1h2}.

Therefore, we have constructed the unique solution $(W^{\varepsilon},\psi^{\varepsilon} )\in L^2(\Omega_T)\times L^2((-\infty ,T],H^4(\mathbb{R}^2))$ to the regularized problem \eqref{Reg} for any fixed and sufficiently small parameter $\varepsilon>0$ with $W_{\rm nc}^{\varepsilon}|_{x_1=0}\in L^2(\Sigma_T)$ and $\partial_t\psi^{\varepsilon}\in L^2(\Sigma_T)$.
Moreover, applying tangential differentiation leads to the existence and uniqueness of solutions  $(W^{\varepsilon},\psi^{\varepsilon})\in H^1(\Omega_T)\times H^1((-\infty ,T],H^4(\mathbb{R}^2))$, again for any fixed and sufficiently small parameter $\varepsilon>0$.

\subsection{Uniform-in-$\varepsilon$ Estimate and Passing to the Limit}

This subsection is devoted to showing the uniform-in-$\varepsilon$ estimate in $H^1$ for solutions $W^{\varepsilon}$ of the regularized problem \eqref{Reg},
from which we can show the existence of solutions to the linearized problem \eqref{ELP3} by passing to the limit $\varepsilon\to 0$.
In the following calculations, to simplify the notation, we drop the superscript ``$\varepsilon$'' in $W^{\varepsilon}$, $\psi^{\varepsilon}$, etc.

\subsubsection{$L^2$ estimate of $W$}
We plug \eqref{es1c} with $\alpha=0$ into \eqref{Res1c} and use \eqref{Res1e} with $|\alpha|=1$ to deduce
\begin{multline}
 \|W(t)\|_{L^2(\Omega)}^2+\|(\psi,\mathrm{D}_{x'}\psi)(t)\|_{L^2(\Sigma)}^2
 +\varepsilon \|(\mathcal{U}, W_7^+,  \mathrm{D}^3_{x'}\psi)\|_{L^2(\Sigma_t)}^2
 \\
 \lesssim_K
 \|(\bm{f},W)\|_{L^2(\Omega_t)}^2
 +  \|(\psi, \mathrm{D}_{\rm tan}\psi,\mathcal{U})\|_{L^2(\Sigma_t)}^2
 + {\varepsilon}\|\mathrm{D}^2_{x'}\psi\|_{L^2(\Sigma_t)}^2,
 \label{uni1a}
\end{multline}
where $\mathcal{U}$ is given in \eqref{U.cal:def}.

\subsubsection{$L^2$ estimate of $\mathrm{D}_{x'} W$}
For $\ell=0,2,3,$ applying the differential operator $\p_{\ell}$ to the interior equations \eqref{Reg.a}, we have
\begin{align}
\sum_{\pm}\int_{\Omega} {\bm{A}}_0^{\pm}\p_{\ell}W^{\pm}\cdot \p_{\ell}W^{\pm} \d x
 +\int_{\Sigma_t} \mathcal{T}_{\rm b}^{\varepsilon}(\p_{\ell}W)
 \lesssim_K \|(\bm{f},W)\|_{H^1(\Omega_t)}^2,
 \label{uni2a}
\end{align}
where $\mathcal{T}_{\rm b}^{\varepsilon}$ is defined by \eqref{Tb.bf}.

Similar to \eqref{Rid1a}, we use the boundary conditions \eqref{Reg.b}--\eqref{Reg.d} to get
\begin{align}
\mathcal{T}_{\rm b}^{\varepsilon}(\p_{\ell}W)
=\;&
-2\p_{\ell}[W_1]\p_{\ell}W_2^+
+\p_{\ell}[(W_2,\ldots,W_6)] \mathring{\rm c}_0  \p_{\ell}\mathcal{U}
+\varepsilon |\p_{\ell}\mathcal{U}_{\rm reg}|^2
\nonumber \\[1mm]
=\;&
-2\mathfrak{s} \p_{\ell}\mathrm{D}_{x'}\cdot\left(\dfrac{\mathrm{D}_{x'}\psi}{|\mathring{N}|}-
\dfrac{\mathrm{D}_{x'}\mathring{\varphi}\cdot\mathrm{D}_{x'}\psi}{|\mathring{N}|^3}\mathrm{D}_{x'}\mathring{\varphi}
\right) \p_{\ell}\bigg(\mathring{\rm B}\psi+\varepsilon\sum_{k=2,3}\p_k^4\psi\bigg)\nonumber\\
&-2\p_{\ell}(\mathring{a}_1\psi)\p_{\ell} W_2^+
+\p_{\ell}(\mathring{\rm c}_1 \psi) \mathring{\rm c}_0  \p_{\ell}\mathcal{U}
+\varepsilon |\p_{\ell}\mathcal{U}_{\rm reg}|^2
\quad \textrm{on }\Sigma_T.
\nonumber %\label{uni2b}
\end{align}
Then we obtain
\begin{align}
\int_{\Sigma_t} \mathcal{T}_{\rm b}^{\varepsilon}(\p_{\ell}W)
=\varepsilon \int_{\Sigma_t} \big|\p_{\ell} \mathcal{U}_{\rm reg}\big|^2
+ \mathcal{J}_{\ell}+ \varepsilon\mathcal{I}_3^{(\ell)}
+\int_{\Sigma_t} \p_{\ell}(\mathring{\rm c}_1 \psi) \mathring{\rm c}_0  \p_{\ell}\mathcal{U},
\label{uni2c}
\end{align}
where $\mathcal{J}_{\ell}$ is defined by \eqref{J.cal:def}, and
\begin{align}
 &\mathcal{I}_3^{(\ell)}:=2  \mathfrak{s} \sum_{k=2,3}\int_{\Sigma_t}
 \p_{\ell}\p_k^2
 \left(\dfrac{\mathrm{D}_{x'}\psi}{|\mathring{N}|}-
 \dfrac{\mathrm{D}_{x'}\mathring{\varphi}\cdot\mathrm{D}_{x'}\psi}{|\mathring{N}|^3}\mathrm{D}_{x'}\mathring{\varphi}
 \right)\cdot\p_{\ell}\p_k^2\mathrm{D}_{x'}\psi.
 \label{I3.cal:def}
\end{align}

Employ Cauchy's inequality to infer
\begin{align*}
\mathcal{I}_3^{(\ell)}
 \geq \,&   \mathfrak{s} \sum_{k=2,3}
 \int_{\Sigma_t}\dfrac{|\p_{\ell}\p_k^2\mathrm{D}_{x'}\psi|^2}{|\mathring{N}|^3}
- C(K)  \sum_{k=2,3}
\left\|\left[\p_{\ell}\p_k^2, \mathring{\rm c}_0\right] \mathrm{D}_{x'}\psi\right\|_{L^2(\Sigma_t)}^2 .
\end{align*}
In view of the decomposition
\begin{align*}
\left[\p_{\ell}\p_k^2, \mathring{\rm c}_0\right] \mathrm{D}_{x'}\psi
=\p_k^2 ( \mathring{\rm c}_1\mathrm{D}_{x'} \psi)
+[\p_k^2, \mathring{\rm c}_0] \mathrm{D}_{x'}\p_{\ell}\psi,
\end{align*}
we utilize the Moser-type calculus inequalities \eqref{Moser3}--\eqref{Moser4}, the embedding theorem, and the estimate \eqref{bas1d3} to derive
\begin{align*}
\big\|\big[\p_{\ell}\p_k^2, \mathring{\rm c}_0\big] \mathrm{D}_{x'}\psi\big\|_{L^2(\Sigma)}^2
&\lesssim  \big\|\mathring{\rm c}_1 \mathrm{D}_{x'}\psi\big\|_{H^2(\Sigma)}^2
+ \big\|\big[\p_k^2, \mathring{\rm c}_0\big] \mathrm{D}_{x'}(\p_{\ell}\psi)\big\|_{L^2(\Sigma)}^2 \\[1mm]
&\lesssim_K
\|(  \mathrm{D}_{x'}\psi, \p_{\ell}\psi)\|_{H^2(\Sigma)}^2.
\end{align*}
Hence we discover
\begin{align}
 \mathcal{I}_3^{(\ell)}
 \geq \,&   \mathfrak{s} \sum_{k=2,3}\int_{\Sigma_t}\dfrac{|\p_{\ell}\p_k^2\mathrm{D}_{x'}\psi|^2}{|\mathring{N}|^3}
 -  C(K) \sum_{|\alpha|\leq 2}
 \|( \mathrm{D}_{x'}^{\alpha}\mathrm{D}_{x'}\psi, \mathrm{D}_{x'}^{\alpha}\p_{\ell}\psi)\|_{L^2(\Sigma_t)}^2.
 \label{I3.cal:es}
\end{align}

Substituting \eqref{uni2c} into \eqref{uni2a} with $\ell=2,3$,
we make use of \eqref{J.cal:es}--\eqref{es2b},  \eqref{I3.cal:es}, and \eqref{Res1e} with $|\alpha|=2$  to deduce
 \begin{multline}
 \|\mathrm{D}_{x'}W(t)\|_{L^2(\Omega)}^2
 + \|\mathrm{D}^2_{x'}\psi(t)\|_{L^2(\Sigma)}^2
 +\varepsilon \|(\mathrm{D}_{x'}\mathcal{U}_{\rm reg},\mathrm{D}^4_{x'}\psi)\|_{L^2(\Sigma_t)}^2
 \\[1mm]
 \lesssim_K\|(\bm{f},W)\|_{H^1(\Omega_t)}^2
 + \sum_{|\alpha|\leq 2}
 \|(\mathrm{D}_{x'}^{\alpha}\psi,  \mathrm{D}_{x'}\p_t\psi,\sqrt{\varepsilon}\mathrm{D}_{x'}^{\alpha}\mathrm{D}_{x'}\psi)\|_{L^2(\Sigma_t)}^2.
 \label{uni2f}
\end{multline}

\subsubsection{$L^2$ estimate of $\p_t W$}
In view of \eqref{uni2c}, we find
\begin{align}
 \int_{\Sigma_t} \mathcal{T}_{\rm b}^{\varepsilon}(\p_{t}W)
 =\varepsilon \int_{\Sigma_t} |\p_t \mathcal{U}_{\rm reg}|^2
 +  \mathcal{J}_0+ \varepsilon\mathcal{I}_3^{(0)}
  +{\mathcal{I}_1}+{\mathcal{I}_2},
 \label{uni3a}
\end{align}
where the terms $\mathcal{J}_0$, $\mathcal{I}_3^{(0)}$, $\mathcal{I}_1$, and $\mathcal{I}_2$ are defined in \eqref{J.cal:def}, \eqref{I3.cal:def}, and \eqref{es3a}.

Thanks to the boundary condition \eqref{Reg.d}, we get
\begin{align}
 \mathcal{I}_1=\int_{\Sigma_t} \mathring{\rm c}_1\p_t\psi\p_t\mathcal{U}
 =\mathcal{I}_{1}^a+\mathcal{I}_{1}^b
 -\underbrace{\varepsilon\sum_{k=2,3} \int_{\Sigma_t} \mathring{\rm c}_1\p_k^4\psi\p_t\mathcal{U}}_{\mathcal{I}_{1c}}
 \label{uni3b}
\end{align}
with $\mathcal{I}_{1}^a$ and $\mathcal{I}_{1}^b$ given in \eqref{I1:es}.
It follows from the definition \eqref{U.cal.r:def} that
\begin{align}
 |\mathcal{I}_{1c}|\leq \frac{\varepsilon}{2}\int_{\Sigma_t}|\p_t \mathcal{U}_{\rm reg}|^2
 +\varepsilon C(K) \|(\mathrm{D}^4_{x'}\psi,\mathcal{U}_{\rm reg})\|_{L^2(\Sigma_t)}^2.
 \label{uni3c}
\end{align}
Plugging \eqref{uni3a}--\eqref{uni3b} into \eqref{uni2a} for $\ell=0$,
and utilizing \eqref{J.cal:es}, \eqref{I3.cal:es},
\eqref{I1a:es}, \eqref{I1b:es}--\eqref{es3b}, and  \eqref{uni3c} imply
\begin{multline}
 \|\p_tW(t)\|_{L^2(\Omega)}^2 +\|(\mathrm{D}_{x'}\p_t\psi, W)(t)\|_{L^2(\Sigma)}^2
 +\varepsilon \|(\p_t\mathcal{U}_{\rm reg},\mathrm{D}^3_{x'}\p_t\psi)\|_{L^2(\Sigma_t)}^2
 \\[2.5mm]
   \lesssim_K
  \sum_{|\alpha|\leq 2} \big\|\big(
 \mathrm{D}_{x'}^{\alpha}\psi, \p_t \psi,\mathrm{D}_{x'}\p_t\psi,W,
 \sqrt{\varepsilon}\mathrm{D}_{x'}^{\alpha}\mathrm{D}_{\rm tan}\psi,
 \sqrt{\varepsilon}\mathrm{D}^4_{x'}\psi\big)\big\|_{L^2(\Sigma_t)}^2
 \\[0.5mm]
 +C(\boldsymbol{\epsilon})\|(\bm{f},W)\|_{H^1(\Omega_{{t}})}^2
+\boldsymbol{\epsilon} \|\p_1W(t)\|_{L^2(\Omega)}^2
\quad
\textrm{for all } \boldsymbol{\epsilon}>0.
 \label{uni3d}
\end{multline}

\subsubsection{$L^2$ estimate of $\p_1 W_{\rm nc}$}
Multiply the equations \eqref{Reg.a} with ${\bm{B}}^{\pm}$ respectively and use the decomposition \eqref{decom}--\eqref{A(1):def} to deduce
\begin{multline} \label{uni4a}
\big(\p_1 W_{\rm nc}^{\pm}-\varepsilon \mathring{\rm c}_0\p_1 W_{\rm nc}^{\pm}, \,0, \,0\big)^{\top}
\\
 ={\bm{B}}^{\pm}
 \bigg(\bm{f}^{\pm}-{\bm{A}}_4^{\pm} W^{\pm}-\sum_{\ell=0,2,3}{\bm{A}}_{\ell}^{\pm}\p_{\ell} W^{\pm}-{\bm{A}}_{(0)}^{\pm} \p_1W^{\pm}\bigg)
 \ \ \textrm{in } \Omega_T^{\delta},
\end{multline}
where ${\bm{B}}^{\pm}$ are defined by \eqref{B.bm:def}.
Hence for suitably small $\varepsilon>0$, we get
\begin{align}
 \|\p_1 W_{\rm nc}(t)\|_{L^2(\Omega^{\delta})}
 &\lesssim_K \|(W,\mathrm{D}_{\rm tan}W, \sigma\p_1 W,\bm{f})(t)\|_{L^2(\Omega)},
 \label{uni4b}
\end{align}
Similar to the derivation of \eqref{es4b}--\eqref{es4c}, for solutions $W$ of the regularized problem \eqref{Reg}, we find
\begin{align}
 \|\sigma \p_1W(t)\|_{L^2(\Omega)}
 + \|  \p_1W(t)\|_{L^2(\Omega\setminus \Omega^{\delta})}
 \lesssim_K\|(\bm{f},W )\|_{H^1(\Omega_t)},
\nonumber % \label{uni4c}
\end{align}
which together with \eqref{uni4b} leads to
\begin{align}
 \|\p_1 W_{\rm nc}(t)\|_{L^2(\Omega)}^2
 &\lesssim_K\|(W,\mathrm{D}_{\rm tan}W)(t)\|_{L^2(\Omega)}^2
 +\|(\bm{f},W)\|_{H^1(\Omega_t)}^2,
 \label{uni4d}
\end{align}
provided $\varepsilon>0$ is sufficiently small.

\subsubsection{$L^2$ estimate of $\p_1 W_{\rm c}$}
It follows from \eqref{Reg.a} and \eqref{J.bm:def} that the characteristic variables $W_8^{\pm}=S^{\pm}$
also satisfy the equations \eqref{S:equ}, which allows us to deduce
\begin{align}
 \|\p_1 W_8^{\pm}(t)\|_{L^2(\Omega)}^2\lesssim_K \|(\bm{f},W)\|_{H^1(\Omega_t)}^2.
 \label{uni5a}
\end{align}

Let us estimate the normal derivative of the remaining characteristic variables  $W_7^{\pm}=H^{\pm}\cdot\mathring{N}^{\pm}$ ({\it cf.}~the transformation \eqref{J.ring}).
For this purpose, we introduce the linearized divergences $\xi^{\pm}$
that are defined by \eqref{xi:def} and satisfy the identities \eqref{xi:iden}.
Remark that the equations of $H_j^{\pm}$ for the regularized problem \eqref{Reg} are different from those for the linearized problem \eqref{ELP3}, due to the presence of the regularized terms $-\varepsilon \bm{J}_{\pm}\p_1 W^{\pm}$ in \eqref{Reg.a}.
Nevertheless, for the matrices $\bm{J}_{\pm}$ defined by \eqref{J.bm:def}--\eqref{J.bm.H:def},
we can still show the energy estimate of $\xi^{\pm}$ for the regularized problem \eqref{Reg}.
More precisely,
taking advantage of the explicit form \eqref{J.bm:def}--\eqref{J.bm.H:def} and the identity $$(H_1^{\pm},H_2^{\pm},H_3^{\pm})^{\top}=\mathring{J}_{\pm}^H (W_5^{\pm},W_6^{\pm},W_7^{\pm})^{\top},$$
we calculate from \eqref{Reg.a} that
\begin{align}
\nonumber &  \p_t^{\mathring{\varPhi}^{+}}H_j^{+} + \mathring{v}^{+}\cdot\nabla^{\mathring{\varPhi}^{+}}H_j^{+}
-\varepsilon\p_1 H_j^+
- \mathring{H}^{+}\cdot\nabla^{\mathring{\varPhi}^{+}} v_j^{+}
+\mathring{H}_j^{+} \nabla^{\mathring{\varPhi}^{+}} \cdot v^{+}
 = \mathring{\rm c}_1 \bm{f}+\mathring{\rm c}_1  W,\\
\nonumber &  \p_t^{\mathring{\varPhi}^{-}}H_j^{-} + \mathring{v}^{-}\cdot\nabla^{\mathring{\varPhi}^{-}} H_j^{-}
- \mathring{H}^{-}\cdot\nabla^{\mathring{\varPhi}^{-}} v_j^{-}
+\mathring{H}_j^{-} \nabla^{\mathring{\varPhi}^{-}} \cdot v^{-}
= \mathring{\rm c}_1 \bm{f}+\mathring{\rm c}_1  W.
\end{align}
Applying the operators $\p_j^{\mathring{\varPhi}^{\pm}}$ respectively to the last equations yields
\begin{align}
& \big(\p_t  +\mathring{w}^{+}_{1}\p_{1} +\mathring{v}_2^{+}\p_{2}+  \mathring{v}_3^{+}\p_{3}-\varepsilon\p_1 \big)\xi^{+} =\mathring{\rm c}_1 \mathrm{D}\bm{f}+ \mathring{\rm c}_2 \bm{f}+\mathring{\rm c}_1 \mathrm{D}W + \mathring{\rm c}_2 W, \label{xi+:eq:R}\\
& \big(\p_t  +\mathring{w}^{-}_{1}\p_{1} +\mathring{v}_2^{-}\p_{2}+  \mathring{v}_3^{-}\p_{3}\big)\xi^{-} =\mathring{\rm c}_1 \mathrm{D}\bm{f}+ \mathring{\rm c}_2 \bm{f}+\mathring{\rm c}_1 \mathrm{D}W + \mathring{\rm c}_2 W. \label{xi-:eq:R}
\end{align}
Use the equations \eqref{xi+:eq:R}--\eqref{xi-:eq:R} and the identity \eqref{w1.ring} to infer
\begin{align}
 \|\xi^{\pm}(t)\|_{L^2(\Omega)}+\varepsilon\|\xi^{+}\|_{L^2(\Sigma_t)} \lesssim_K \|(\bm{f},W)\|_{H^1(\Omega_t)},
 \nonumber
\end{align}
which together with \eqref{xi:iden} implies
\begin{align}
 \|\p_1 W_7^{\pm}(t)\|_{L^2(\Omega)}^2\lesssim_K\|(W,\mathrm{D}_{\rm tan}W)(t)\|_{L^2(\Omega)}^2
 +\|(\bm{f},W)\|_{H^1(\Omega_t)}^2 .
 \label{uni5b}
\end{align}

\subsubsection{Proof of Theorem \ref{thm:lin}}
It follows from \eqref{Reg.d} that
\begin{align}
 \| \p_t \psi \|_{L^2(\Sigma_t)}^2
 \lesssim_K \|(W_2^+,\psi,\mathrm{D}_{x'} \psi,\varepsilon  \mathrm{D}^4_{x'} \psi)\|_{L^2(\Sigma_t)}^2.
 \label{uni2e}
\end{align}
Using the basic estimate
\begin{align*}
 \| \mathrm{D}^2_{x'}\mathrm{D}_{\rm tan}\psi\|_{L^2(\Sigma_t)}^2
 \leq \boldsymbol{\epsilon}  \|  \mathrm{D}^3_{x'}\mathrm{D}_{\rm tan}\psi\|_{L^2(\Sigma_t)}^2
 +C(\boldsymbol{\epsilon} )\| \mathrm{D}_{x'}\mathrm{D}_{\rm tan}\psi\|_{L^2(\Sigma_t)}^2,
\end{align*}
and taking a suitable linear combination of \eqref{uni1a}, \eqref{uni2f}, \eqref{uni3d}, \eqref{uni4d}--\eqref{uni5a}, and \eqref{uni5b}, we choose $\boldsymbol{\epsilon}>0$ suitably small to discover
 \begin{multline}
  \|(W,\mathrm{D}W)(t)\|_{L^2(\Omega)}^2+
  \sum_{|\alpha|\leq 2}\|(\mathrm{D}_{x'}^{\alpha}\psi, \mathrm{D}_{x'}\p_t\psi,W)(t)\|_{L^2(\Sigma)}^2\\
  +\|\p_t \psi\|_{L^2(\Sigma_t)}^2
  +\varepsilon \sum_{ 2\leq |\alpha|\leq 3}
  \|( \mathrm{D}_{\rm tan}\mathcal{U}_{\rm reg}, \mathrm{D}_{x'}^{\alpha}\mathrm{D}_{\rm tan}\psi)\|_{L^2(\Sigma_{{t}})}^2
  \\
  \qquad \lesssim_K
  \|(\bm{f},W)\|_{H^1(\Omega_t)}^2+
  \sum_{|\alpha|\leq 2}\|(\mathrm{D}_{x'}^{\alpha}\psi, \mathrm{D}_{x'}\p_t\psi,W)\|_{L^2(\Sigma_t)}^2.
\nonumber
 \end{multline}
We apply Gr\"{o}nwall's inequality to the last estimate and compute
 \begin{multline}
  \|(W,\mathrm{D}W)(t)\|_{L^2(\Omega)}^2+
  \sum_{|\alpha|\leq 2}\|(\mathrm{D}_{x'}^{\alpha}\psi, \mathrm{D}_{x'}\p_t\psi,W)(t)\|_{L^2(\Sigma)}^2 \\
    +\|\p_t \psi\|_{L^2(\Sigma_t)}^2
  +\varepsilon \sum_{ 2\leq |\alpha|\leq 3}
  \|(\mathrm{D}_{\rm tan}\mathcal{U}_{\rm reg}, \mathrm{D}_{x'}^{\alpha}\mathrm{D}_{\rm tan}\psi)\|_{L^2(\Sigma_{{t}})}^2
  \lesssim_K
  \|\bm{f}\|_{H^1(\Omega_t)}^2.
\nonumber %\label{uni6a}
 \end{multline}
Consequently, we derive
 \begin{multline}
  \|W\|_{H^1(\Omega_t)}
  +\|W\|_{L^2(\Sigma_t)}  + \|(\psi, \mathrm{D}_{x'}\psi)\|_{H^1(\Sigma_t)}\\
  +\sqrt{\varepsilon} \sum_{ 2\leq |\alpha|\leq 3}
 \|( \mathrm{D}_{\rm tan}\mathcal{U}_{\rm reg}, \mathrm{D}_{x'}^{\alpha}\mathrm{D}_{\rm tan}\psi)\|_{L^2(\Sigma_{{t}})}
  \lesssim_K
  \|\bm{f}\|_{H^1(\Omega_t)}
  \label{uni6b}
 \end{multline}
 for all $0\leq t\leq T$,
 where $\mathcal{U}_{\rm reg}$ is defined by \eqref{U.cal.r:def}.

The uniform-in-$\varepsilon$ estimate \eqref{uni6b} allows us to construct the unique solution of the linearized problem \eqref{ELP3} by passing to the limit $\varepsilon\to 0$.
As a matter of fact, in view of \eqref{uni6b}, we can extract a subsequence weakly convergent to
$(W,\psi )\in H^1(\Omega_T)\times H^1((-\infty,T],H^2(\mathbb{R}^2))$
with $\partial_1W\in L^2(\Omega_T)$ and $W|_{x_1=0}\in L^2(\Sigma_T)$.
Since $\partial_1W$ and $\sqrt{\varepsilon} (\p_2^4+\p_3^4)\psi$ are uniformly bounded in $L^2(\Omega_T)$ and $L^2(\Sigma_T)$ respectively ({\it cf.}~\eqref{uni6b}),
the passage to the limit $\varepsilon\to 0$ in \eqref{Reg} verifies that
$(W,\psi )$ solves the linearized problem \eqref{ELP3}.
Moreover, the uniqueness of solutions follows from the {\it a priori} estimate \eqref{es6c}.

Thanks to the existence and uniqueness of solutions $(W,\psi)$ in $H^1(\Omega_{T})\times H^1(\Sigma_{{T}})$ of the reduced problem \eqref{ELP3}, we can show that there exists a unique solution $(\dot{V},\psi)\in H^1(\Omega_{T})\times H^1(\Sigma_{{T}})$ to the effective linear problem \eqref{ELP1}. Moreover, the $H^1$ estimate \eqref{H1:es} follows by combining the estimate \eqref{es6c} with \eqref{V.natural:es} and \eqref{J.ring}.

%\newpage
\section{Tame Estimate} \label{sec:tame}
This section is dedicated to showing the following theorem, that is,
the tame \textit{a priori} estimate in the usual Sobolev spaces $H^m$ for the effective linear problem \eqref{ELP1} with $m\in\mathbb{N}$ large enough.
\begin{theorem}
\label{thm:tame}
Let $K>0$ and $m\in\mathbb{N}$ with $m\geq 3$ be fixed.
Then there exist constants $T_0>0$ and $C(K)>0$ such that if
the basic state $(\mathring{U}(t,x),\mathring{\varphi}(t,x'))$ satisfies \eqref{bas1a}--\eqref{bas1d} and $(\mathring{V}^{\pm},\mathring{\varphi})\in H^{{m+2}}(\Omega_T)\times  H^{{m+2}}(\Sigma_T)$
for $\mathring{V}^{\pm}:=\mathring{U}^{\pm}-\widebar{U}^{\pm}$,
and the source terms $f^{\pm}\in H^{m}(\Omega_T)$, $g\in H^{m+1/2}(\Sigma_T)$ vanish in the past, for some $0<T\leq T_0$,
then the problem \eqref{ELP1} admits a unique solution $(\dot{V}^{\pm},\psi)\in H^{m}(\Omega_T)\times H^{m}(\Sigma_T)$ satisfying the tame estimate
\begin{align}
 &{\|}(\dot{V},\varPsi,\mathrm{D}_{x'}\varPsi){\|}_{H^{m}(\Omega_T)}+\|(\psi,\mathrm{D}_{x'}\psi)\|_{H^{m}(\Sigma_T)}
 \nonumber \\
 &
 \quad \leq   C(K)\Big\{  {\|}f^{\pm}{\|}_{H^{m}(\Omega_T)} +\|g\|_{H^{m+1/2}(\Sigma_T)}
 \nonumber \\
 &\qquad \qquad\quad \, \  +{\|}(\mathring{V},\mathring{\varPsi},\mathrm{D}_{x'}\mathring{\varPsi}){\|}_{H^{m+2}(\Omega_T)}\left({\|}f^{\pm}{\|}_{H^{3}(\Omega_T)}+\|g\|_{H^{7/2}(\Sigma_T)}\right) \Big\}.
 \label{tame:es}
\end{align}
\end{theorem}
To derive the tame estimate \eqref{tame:es}, it is sufficient to obtain an analogous tame estimate for solutions $(W,\psi)$ of the reduced problem \eqref{ELP3}.
We shall first make the estimate of the normal derivatives of $W$ through its tangential ones.
Then we will control the tangential derivatives by using the spatial regularity enhanced by the surface tension.

\subsection{Estimate of the Normal Derivatives}
The normal derivatives of solutions $W$ to the problem \eqref{ELP3} can be estimated through the tangential ones as follows.
\begin{proposition}   \label{pro:normal}
 If the assumptions in Theorem \ref{thm:tame} are satisfied, then
 \begin{align}
  \VERT  W (t)\VERT_{m}^2
  \lesssim_K \VERT W(t)\VERT_{{\rm tan},\,m}^2
  + \mathcal{M}_1(t),
  \label{normal.est}
 \end{align}
where $ \VERT \cdot \VERT_{{\rm tan},m}$ and $\VERT  \cdot \VERT_{m}$ are defined by \eqref{VERT.tan}, and
\begin{align}
\mathcal{M}_1(t):=\,&\|    (   \bm{f},W)\|_{H^m(\Omega_t)}^2 +
\|    ( \bm{f},W)\|_{L^{\infty}(\Omega_t)}^2
\nonumber\\
&+\| (\mathring{V},  \mathring{\varPsi}, \mathrm{D}_{x'}\mathring{\varPsi})\|_{H^{m+2}(\Omega_T)}^2
\|    ( \bm{f},W)\|_{L^{\infty}(\Omega_t)}^2.
\label{M1.cal:def}
\end{align}
\end{proposition}
\begin{proof}
We divide the proof into two steps.

\vspace*{2mm}
\noindent {\bf 1. Estimate of the noncharacteristic variables.}
Let the multi-index ${\beta}=({\beta}_0,{\beta}_2,{\beta}_3)\in\mathbb{N}^3$ and the integer $k\geq 1$ satisfy $|{\beta}|+k\leq m$.
Applying the differential operator $\p_1^{k-1}\mathrm{D}_{\rm tan}^{\beta}:=\p_1^{k-1}\p_t^{{\beta}_0}\p_2^{{\beta}_2}\p_3^{{\beta}_3}$ to the identity \eqref{id4a} implies
\begin{align}
 \|\p_1^{k}\mathrm{D}_{\rm tan}^{\beta}W_{\rm nc} (t)\|_{L^2(\Omega^{\delta})}^2
 \lesssim \mathcal{I}_{4}^{a}+\mathcal{I}_{4}^{b}+\mathcal{I}_{4}^{c},
 \label{nor1a}
\end{align}
where
\begin{align} \nonumber
 \left\{
 \begin{aligned}
  &\mathcal{I}_{4}^{a}:=\big\|\p_1^{k-1}\mathrm{D}_{\rm tan}^{\beta}\big( \bm{B}^{\pm}\bm{f}^{\pm}-\bm{B}^{\pm}\bm{A}_4^{\pm}W^{\pm} \big) (t)\big\|_{L^2(\Omega^{\delta})}^2,
  \\
  &\mathcal{I}_{4}^{b}:=\big\|\p_1^{k-1}\mathrm{D}_{\rm tan}^{\beta}\big(\mathring{\rm c}_1 \mathrm{D}_{\rm tan} W \big) (t)\big\|_{L^2(\Omega)}^2,
  \\
  &\mathcal{I}_{4}^{c}:=\big\|\p_1^{k-1}\mathrm{D}_{\rm tan}^{\beta}\big( \bm{B}^{\pm}\bm{A}_{(0)}^{\pm}\p_1W^{\pm} \big)(t) \big\|_{L^2(\Omega^{\delta})}^2.
 \end{aligned}
 \right.
\end{align}

Since  $\bm{A}_4^{\pm}$ are $C^{\infty}$--functions of $(\mathring{V},\mathrm{D}\mathring{\varPsi},\mathrm{D}\mathring{V},
\mathrm{D}\mathrm{D}_{x'}\mathring{\varPsi})$
and $ \bm{B}^{\pm}$ are $C^{\infty}$--functions of $(\mathring{V},\mathrm{D}\mathring{\varPsi})$,
we use the Moser-type calculus inequality \eqref{Moser3} to obtain
\begin{align}
 \mathcal{I}_{4}^{a}
 &\lesssim
 \big\|\p_1^{k-1}\mathrm{D}_{\rm tan}^{\beta}
 \big(  \mathring{\rm c}_1\bm{f} +\mathring{\rm c}_1W \big) \big\|_{H^1(\Omega_t)}^2
 \lesssim
 \big\| \mathring{\rm c}_1\bm{f} +\mathring{\rm c}_1W  \big\|_{H^m(\Omega_t)}^2
 \lesssim_K \mathcal{M}_1(t),
 \label{I4a:es}
\end{align}
where $\mathcal{M}_1(t)$ is defined by \eqref{M1.cal:def}.

By virtue of the Moser-type calculus inequality \eqref{Moser4}, we get
\begin{align}
 \mathcal{I}_{4}^{b} &\lesssim
 \big\|
 \mathring{\rm c}_1 \p_1^{k-1}\mathrm{D}_{\rm tan}^{\beta} \mathrm{D}_{\rm tan} W (t)
 \big\|_{L^2(\Omega)}^2
 +\big\|
 \big[\p_1^{k-1}\mathrm{D}_{\rm tan}^{\beta},\,\mathring{\rm c}_1\big]\mathrm{D}_{\rm tan} W (t)
 \big\|_{L^2(\Omega)}^2
 \nonumber\\
 &\lesssim_K
 \VERT \p_1^{k-1}W(t)\VERT_{{\rm tan},\,m-k+1}^2
 + \big\|
 \big[\p_1^{k-1}\mathrm{D}_{\rm tan}^{\beta},\,\mathring{\rm c}_1\big]\mathrm{D}_{\rm tan} W
 \big\|_{H^1(\Omega_t)}^2 \nonumber\\
 &\lesssim_K
 \VERT \p_1^{k-1}W(t)\VERT_{{\rm tan},\,m-k+1}^2
 +\mathcal{M}_1(t).
 \label{I4b:es}
\end{align}

It follows from \eqref{bas2b} and \eqref{Moser4} that
\begin{align}
 \mathcal{I}_{4}^{c}&\lesssim
 \big\|\bm{B}^{\pm}\bm{A}_{(0)}^{\pm} \p_1^{k}\mathrm{D}_{\rm tan}^{\beta} W^{\pm}(t)\big\|_{L^2(\Omega)}^2
 +\big\|
 \big[\p_1^{k-1}\mathrm{D}_{\rm tan}^{\beta},\,\mathring{\rm c}_1\big]\p_1 W (t)
 \big\|_{L^2(\Omega)}^2
 \nonumber\\
 &\lesssim_K
 \big\|\sigma \p_1^{k}\mathrm{D}_{\rm tan}^{\beta} W (t)\big\|_{L^2(\Omega)}^2
 +\big\|
 \big[\p_1^{k-1}\mathrm{D}_{\rm tan}^{\beta},\,\mathring{\rm c}_1\big]\p_1 W
 \big\|_{H^1(\Omega_t)}^2  \nonumber\\
 &\lesssim_K
 \big\|\sigma \p_1^{k}\mathrm{D}_{\rm tan}^{\beta} W (t)\big\|_{L^2(\Omega)}^2
 +\mathcal{M}_1(t),
 \label{I4c:es}
\end{align}
where the $C^{\infty}$--function $\sigma=\sigma(x_1)$ satisfies \eqref{sigma:def}.
In particular, $\sigma(0)=0$.

Regarding the first term on the right-hand side of \eqref{I4c:es}, we apply the operator $\sigma\p_1^{k}\mathrm{D}_{\rm tan}^{\beta}$ to the equations \eqref{ELP3a} and employ the standard argument of the energy method to derive
\begin{align}
 \|\sigma \p_1^{k}\mathrm{D}_{\rm tan}^{\beta} W (t)\|_{L^2(\Omega)}^2
 \lesssim_K \mathcal{M}_1(t).
 \label{nor1b}
\end{align}
Since the weight $\sigma$ is away from zero outside the boundary $\Sigma_{{T}}$, we have
\begin{align}
 \| \p_1^{k}\mathrm{D}_{\rm tan}^{\beta} W (t)\|_{L^2(\Omega\setminus \Omega^{\delta})}^2
 \lesssim_K \mathcal{M}_1(t).
 \label{nor1c}
\end{align}

Plug \eqref{I4a:es}--\eqref{nor1b} into \eqref{nor1a} and combine the resulting estimate with \eqref{nor1c} to infer
\begin{align}
 \| \p_1^{k}\mathrm{D}_{\rm tan}^{\beta} W_{\rm nc} (t)\|_{L^2(\Omega)}^2
 \lesssim_K \VERT \p_1^{k-1}W(t)\VERT_{{\rm tan},\,m-k+1}^2+\mathcal{M}_1(t).
 \nonumber
\end{align}
Since the last estimate holds for all ${\beta}\in\mathbb{N}^3$ with $|{\beta}|\leq m-k$, we derive
\begin{align}
 \VERT \p_1^{k}  W_{\rm nc} (t)\VERT_{{\rm tan},\, m-k }^2
 \lesssim_K \VERT \p_1^{k-1}W(t)\VERT_{{\rm tan},\,m-k+1}^2+\mathcal{M}_1(t)
 \label{nor1e}
\end{align}
for $1\leq k\leq m$.

\vspace*{2mm}
\noindent {\bf 2. Estimate of the characteristic variables.}
We first consider the characteristic variables $W_8^{\pm}=S^{\pm}$ (entropies).
Let $\alpha:=(\alpha_0,\alpha_1,\alpha_2,\alpha_3)\in\mathbb{N}^{4}$ be any multi-index
with $|\alpha|\leq  m$.
Apply the operator $\mathrm{D}^{\alpha}:=\p_t^{\alpha_0}\p_1^{\alpha_1}\p_2^{\alpha_2}\p_3^{\alpha_3}$ to the equations \eqref{S:equ}
and multiply the resulting identities by $\mathrm{D}^{\alpha} W_8^{\pm}$ respectively to find
\begin{align*}
 &\p_t \Big(\big|\mathrm{D}^{\alpha}W_8^{\pm}\big|^2\Big)
 +\p_{1} \Big(\mathring{w}^{\pm}_{1} \big|\mathrm{D}^{\alpha}W_8^{\pm}\big|^2\Big)\\[1.5mm]
 &\qquad +\sum_{k=2,3}\p_k \Big(\mathring{v}^{\pm}_{k} \big|\mathrm{D}^{\alpha}W_8^{\pm}\big|^2\Big)
 -\Big(\p_{1} \mathring{w}^{\pm}_{1}+\p_{2} \mathring{v}^{\pm}_{2} +\p_{3} \mathring{v}^{\pm}_{3}  \Big)\big|\mathrm{D}^{\alpha}W_8^{\pm}\big|^2\\
 &\quad =2\mathrm{D}^{\alpha} W_8^{\pm}\Big(
 \mathrm{D}^{\alpha} \bm{f}^{\pm}_8
 +  \mathrm{D}^{\alpha}(\mathring{\rm c}_1 W)
 -\big[\mathrm{D}^{\alpha},\, \mathring{w}^{\pm}_{1}\big]\p_{1} W_8^{\pm}
 -\sum_{k=2,3}[\mathrm{D}^{\alpha},\, \mathring{v}^{\pm}_{k}]\p_{k} W_8^{\pm}
 \Big).
\end{align*}
Integrating the last identities over $\Omega_t$ and employing \eqref{Moser3}--\eqref{Moser4} yield
\begin{align}
 \VERT W_8^{\pm}(t)\VERT_{s}^2\lesssim_K \mathcal{M}_1(t).
 \label{S:es}
\end{align}

Next we recover the normal derivatives of the characteristic variables $W_7^{\pm}$ from the estimate of the linearized divergences $\xi^{\pm}$ defined by \eqref{xi:def}.
More precisely,
we apply the differential operator $\mathrm{D}^{\alpha}$ with $|\alpha|\leq m-1$ to the equations \eqref{xi:equ} and
multiply the resulting identities by $\mathrm{D}^{\alpha} \xi^{\pm}$ respectively to deduce
\begin{align}
 \nonumber
 \|\mathrm{D}^{\alpha} \xi^{\pm}(t)\|_{L^2 (\Omega)}^2
 \lesssim_K \, &
 \sum_{k=2,3}\big\|\big(\mathrm{D}^{\alpha}\xi^{\pm},\big[\mathrm{D}^{\alpha},\mathring{w}_{1}^{\pm} \big]\p_{1} \xi^{\pm} ,\,\big[\mathrm{D}^{\alpha},\mathring{v}_{k}^{\pm} \big]\p_{k} \xi^{\pm}\big) \big\|_{L^{2}(\Omega_t)}^2\\
 \label{xi:es1}&
 +\big\|\mathrm{D}^{\alpha}\big(
 \mathring{\rm c}_1 \mathrm{D}\bm{f}+ \mathring{\rm c}_2 \bm{f}+\mathring{\rm c}_1 \mathrm{D}W + \mathring{\rm c}_2 W\big) \big\|_{L^{2}(\Omega_t)}^2.
\end{align}
It follows from \eqref{xi:iden} that
\begin{align}
 \xi^{\pm}=\mathring{\rm c}_1 W+\mathring{\rm c}_1  \mathrm{D}W.
 \label{xi:id2}
\end{align}
Then we utilize the Moser-type calculus inequalities \eqref{Moser3}--\eqref{Moser4} to infer
\begin{align}
 \nonumber&\sum_{k=2,3}\big\|\big(\big[\mathrm{D}^{\alpha},\mathring{w}_{1}^{\pm} \big]\p_{1} \xi^{\pm} ,\,\big[\mathrm{D}^{\alpha},\mathring{v}_{k}^{\pm} \big]\p_{k} \xi^{\pm}\big) \big\|_{L^{2}(\Omega_t)}^2
 +\big\|\mathrm{D}^{\alpha}\big(\mathring{\rm c}_1 \mathrm{D}W + \mathring{\rm c}_2 W\big) \big\|_{L^{2}(\Omega_t)}^2
 \\
 &  \lesssim_K
\sum_{|\gamma|=2} \big\|\!\big(
 \big[\mathrm{D}^{\alpha},\, \mathring{\rm c}_2 \big] W,\,
 \big[\mathrm{D}^{\alpha},\, \mathring{\rm c}_2 \big] \mathrm{D}W,\,
 \big[\mathrm{D}^{\alpha},\, \mathring{\rm c}_1 \big] \mathrm{D}^{\gamma} W,\,
 \mathring{\rm c}_1 \mathrm{D}^{\alpha}\mathrm{D}W,\,
 \mathring{\rm c}_2 \mathrm{D}^{\alpha}W \big)\!\big\|_{L^{2}(\Omega_t)}^2
 \nonumber\\
 &  \lesssim_K
 \|    W\|_{H^m(\Omega_t)}^2 +\Big(1+\| (\mathring{V},  \mathring{\varPsi}, \mathrm{D}_{x'}\mathring{\varPsi})\|_{H^{m+2}(\Omega_T)}^2\Big)
 \|    W\|_{L^{\infty}(\Omega_t)}^2  .
 \label{xi:es2}
\end{align}
Similarly, we have
\begin{align}
 \nonumber
 &\big\|\mathrm{D}^{\alpha}\big(
 \mathring{\rm c}_1 \mathrm{D}\bm{f}+ \mathring{\rm c}_2 \bm{f}\big) \big\|_{L^{2}(\Omega_t)}^2\\
 \nonumber&\quad \lesssim_K
 \big\|\big(
 \mathring{\rm c}_1 \mathrm{D}^{\alpha}\mathrm{D}\bm{f},\,
 \mathring{\rm c}_2\mathrm{D}^{\alpha} \bm{f},\,
 [\mathrm{D}^{\alpha},\mathring{\rm c}_1]\mathrm{D}\bm{f},\,
 [\mathrm{D}^{\alpha},\mathring{\rm c}_2]\bm{f}\big) \big\|_{L^{2}(\Omega_t)}^2\\
 &\quad \lesssim_K
 \|    \bm{f}\|_{H^m(\Omega_t)}^2 +\Big(1+\| (\mathring{V},  \mathring{\varPsi}, \mathrm{D}_{x'}\mathring{\varPsi})\|_{H^{m+2}(\Omega_T)}^2\Big)
 \|    \bm{f}\|_{L^{\infty}(\Omega_t)}^2  .
 \label{xi:es3}
\end{align}
Plugging \eqref{xi:es2}--\eqref{xi:es3} into \eqref{xi:es1} and using Gr\"{o}nwall's inequality imply
\begin{align}
 \VERT \xi^{\pm}(t) \VERT_{m-1}^2
 =\sum_{|\alpha|\leq m-1}\|\mathrm{D}^{\alpha} \xi^{\pm}(t)\|_{L^2 (\Omega)}^2\lesssim_K \mathcal{M}_1(t).
 \label{xi:es}
\end{align}
Moreover, it follows from \eqref{xi:iden} that
\begin{align}
 \p_1 W_7^{\pm}=\mathring{\rm c}_1 \xi^{\pm}+\mathring{\rm c}_1 \mathrm{D}_{\rm tan} W+\mathring{\rm c}_1  W.
 \nonumber
\end{align}
Then for any multi-index $\beta\in\mathbb{N}^3$ and integer $k\geq 1$ with $|\beta|+k\leq m$, we take advantage of the identity \eqref{xi:id2}, the estimate \eqref{xi:es}, and the inequalities \eqref{Moser3}--\eqref{Moser4} to get
\begin{align}
 \nonumber
 \| \p_1^{k}\mathrm{D}_{\rm tan}^{\beta} W_{7}^{\pm} (t)\|_{L^2(\Omega)}^2
 \lesssim_K\,& \VERT \xi^{\pm}(t) \VERT_{m-1}^2
 +\VERT \p_1^{k-1}W(t) \VERT_{{\rm tan},\,m-k+1}^2\\
 \nonumber
 &+\big\|\big(\big[\p_1^{k-1}\mathrm{D}_{\rm tan}^{\beta}, \mathring{\rm c}_1\big]W,\,
 \big[\p_1^{k-1}\mathrm{D}_{\rm tan}^{\beta}, \mathring{\rm c}_1\big]\mathrm{D}W\big)\big\|_{H^1(\Omega_t)}^2\\
 \lesssim_K\,&  \VERT \p_1^{k-1}W(t) \VERT_{{\rm tan},\,m-k+1}^2
 +\mathcal{M}_1(t).
 \label{W7:es}
\end{align}
Combining \eqref{nor1e}--\eqref{S:es} and \eqref{W7:es} gives
\begin{align}
 \nonumber
 \VERT \p_1^k W(t)\VERT_{{\rm tan},\, m-k}^2
 \lesssim_K
 \VERT \p_1^{k-1} W(t)\VERT_{{\rm tan},\, m-k+1}^2
 +\mathcal{M}_1(t)
 \ \ \  \textrm{for }1\leq k\leq m. \nonumber
\end{align}
Since
$
 \VERT u \VERT_{m}^2=\sum_{k=0}^m \VERT \p_1^k u \VERT_{{\rm tan},\,m-k}^2,
$
we can derive \eqref{normal.est} by induction.
 This completes the proof.
\end{proof}

\subsection{Estimate of the Tangential Derivatives}
The following proposition concerns the estimate of the tangential derivatives.
\begin{proposition}   \label{pro:tan}
 If the assumptions in Theorem \ref{thm:tame} are satisfied, then
 \begin{multline}
  \VERT  W (t)\VERT_{{\rm tan},\,m}^2
  +\sum_{|\beta|\leq m}\|(\mathrm{D}_{\rm tan}^{\beta}\psi,
  \mathrm{D}_{\rm tan}^{\beta}\mathrm{D}_{x'}\psi)(t)\|_{L^2(\Sigma)}^2
  \\
  \lesssim_K \boldsymbol{\epsilon}\VERT W(t)\VERT_{m}^2
  +C(\boldsymbol{\epsilon}) \mathcal{M}_1(t)
  +C(\boldsymbol{\epsilon}) \mathcal{M}_2(t)
  \label{tan:est}
 \end{multline}
for all $\boldsymbol{\epsilon}>0$,
where $ \VERT \cdot \VERT_{{\rm tan},m}$, $\VERT  \cdot \VERT_{m}$,
and $\mathcal{M}_1(t)$ are defined in \eqref{VERT.tan} and \eqref{M1.cal:def}, and
\begin{align}
 \mathcal{M}_2(t):=
\,& \|( \psi,\mathrm{D}_{x'}\psi)\|_{H^m(\Sigma_t)}^2 +
  \|( \psi,\mathrm{D}_{x'}\psi)\|_{L^{\infty}(\Sigma_t)}^2\nonumber\\
& +\| (\mathring{V},  \mathring{\varPsi}, \mathrm{D}_{x'}\mathring{\varPsi})\|_{H^{m+2}(\Omega_T)}^2
 \|( \psi,\mathrm{D}_{x'}\psi)\|_{L^{\infty}(\Sigma_t)}^2.
 \label{M2.cal}
\end{align}
\end{proposition}
\begin{proof}
Let $\beta=(\beta_0,\beta_2,\beta_3)\in\mathbb{N}^3$ satisfy $|\beta|\leq m$.
Applying the differential operator $\mathrm{D}_{\rm tan}^{\beta}$ to the equations \eqref{ELP3a} implies
\begin{align}
\sum_{\pm}\int_{\Omega} {\bm{A}}_0^{\pm}\mathrm{D}_{\rm tan}^{\beta}W^{\pm}\cdot \mathrm{D}_{\rm tan}^{\beta}W^{\pm}\d x
+\int_{\Sigma_{{t}}} \mathcal{T}_{\rm b}(\mathrm{D}_{\rm tan}^{\beta} W)
 =\mathcal{I}_{5},
\label{tan:es1}
\end{align}
where the operator $\mathcal{T}_{\rm b}$ is defined by \eqref{T.cal:def}, and
\begin{align}
\nonumber \mathcal{I}_{5}:=\,&
2\sum_{\pm}\int_{\Omega_t}\mathrm{D}_{\rm tan}^{\beta}W^{\pm}\cdot
\mathrm{D}_{\rm tan}^{\beta} \left( \bm{f}^{\pm}
-\bm{A}_4^{\pm}W^{\pm}\right)\\
&\!-\sum_{\pm}\sum_{i=0}^3\int_{\Omega_t}\mathrm{D}_{\rm tan}^{\beta}W^{\pm}\cdot
\left(2[\mathrm{D}_{\rm tan}^{\beta},  \bm{A}_i^{\pm}]\p_i W^{\pm}
-\p_i{\bm{A}}_i^{\pm} \mathrm{D}_{\rm tan}^{\beta} W^{\pm} \right).
\nonumber
\end{align}
A standard calculation with an application of \eqref{Moser3}--\eqref{Moser4} leads to
\begin{align}
\mathcal{I}_{5}\lesssim_K \mathcal{M}_1(t).
\label{I5.cal:es}
\end{align}

Similar to \eqref{id1a}, we can derive from the boundary conditions \eqref{ELP3b}--\eqref{ELP3d} that
\begin{align}
\int_{\Sigma_{{t}}} \mathcal{T}_{\rm b}(\mathrm{D}_{\rm tan}^{\beta} W)
=\,& -2 \int_{\Sigma_{{t}}}  \mathrm{D}_{\rm tan}^{\beta} [W_1] \mathrm{D}_{\rm tan}^{\beta} W_2^+
+\int_{\Sigma_{{t}}} \mathrm{D}_{\rm tan}^{\beta} [(W_2,\ldots,W_6)]\mathring{\rm c}_0 \mathrm{D}_{\rm tan}^{\beta}\mathcal{U} \nonumber\\
=\,&
\mathcal{I}_{6}^{a}+\mathcal{I}_{6}^{b}+\mathcal{I}_{6}^{c}+\mathcal{I}_{6}^{d}
+\int_{\Sigma_{{t}}} \mathrm{D}_{\rm tan}^{\beta} (\mathring{\rm c}_1 \psi )\mathring{\rm c}_0 \mathrm{D}_{\rm tan}^{\beta}\mathcal{U},
\label{tan:id1}
\end{align}
where
\begin{align}
\nonumber &\mathcal{I}_{6}^{a}:= 2\mathfrak{s}\int_{\Sigma_t} \mathrm{D}_{\rm tan}^{{\beta}}\bigg(\dfrac{\mathrm{D}_{x'}\psi}{|\mathring{N}|}- \dfrac{\mathrm{D}_{x'}\mathring{\varphi}\cdot\mathrm{D}_{x'}\psi}{|\mathring{N}|^3}\mathrm{D}_{x'}\mathring{\varphi}\bigg) \cdot (\p_t+\mathring{v}_2^+\p_2+\mathring{v}_3^+\p_3)\mathrm{D}_{\rm tan}^{{\beta}} \mathrm{D}_{x'}\psi ,\\
\nonumber &\mathcal{I}_{6}^{b}:= 2\mathfrak{s}\int_{\Sigma_t}
\mathrm{D}_{\rm tan}^{{\beta}}\bigg(\dfrac{\mathrm{D}_{x'}\psi}{|\mathring{N}|}- \dfrac{\mathrm{D}_{x'}\mathring{\varphi}\cdot\mathrm{D}_{x'}\psi}{|\mathring{N}|^3}\mathrm{D}_{x'}\mathring{\varphi}\bigg)
\\
\nonumber &\qquad \qquad \ \; \quad    \cdot \Big(\big[\mathrm{D}_{\rm tan}^{{\beta}}\mathrm{D}_{x'}, \mathring{v}_2^+\p_2+\mathring{v}_3^+\p_3\big]\psi+\mathrm{D}_{\rm tan}^{{\beta}}\mathrm{D}_{x'}\big(\mathring{a}_7 \psi\big) \Big),\\
\nonumber &\mathcal{I}_{6}^{c}:=-2 \int_{\Sigma_{{t}}}
\mathrm{D}_{\rm tan}^{{\beta}}\big(\mathring{a}_1\psi \big) (\p_t+\mathring{v}_2^+\p_2+\mathring{v}_3^+\p_3)\mathrm{D}_{\rm tan}^{{\beta}}\psi,\\
\nonumber &\mathcal{I}_{6}^{d}:= -2\int_{\Sigma_{{t}}}  \mathrm{D}_{\rm tan}^{{\beta}}\big(\mathring{a}_1\psi \big)
\Big( \big[\mathrm{D}_{\rm tan}^{{\beta}} , \mathring{v}_2^+\p_2+\mathring{v}_3^+\p_3\big]\psi+\mathrm{D}_{\rm tan}^{{\beta}} \big(\mathring{a}_7 \psi\big)\Big)
.
\end{align}
By a direct computation, we obtain
\begin{align*}
\mathcal{I}_{6}^{a}
= \,&\,
\mathfrak{s}\int_{\Sigma}
\bigg(
\frac{|\mathrm{D}_{\rm tan}^{{\beta}}\mathrm{D}_{x'}\psi|^2}{|\mathring{N}|}-
\frac{|\mathrm{D}_{x'}\mathring{\varphi}\cdot\mathrm{D}_{\rm tan}^{{\beta}}\mathrm{D}_{x'}\psi|^2}{|\mathring{N}|^3}\bigg)\mathrm{d}x'\\
&+\int_{\Sigma}
[\mathrm{D}_{\rm tan}^{{\beta}},\mathring{\rm c}_0]\mathrm{D}_{x'}\psi\cdot\mathrm{D}_{\rm tan}^{{\beta}}\mathrm{D}_{x'}\psi \mathrm{d}x'
+\int_{\Sigma_t}
\mathring{\rm c}_2\mathrm{D}_{\rm tan}^{{\beta}}\mathrm{D}_{x'}\psi \cdot\mathrm{D}_{\rm tan}^{{\beta}}\mathrm{D}_{x'}\psi\\
&+\int_{\Sigma_t}
\mathrm{D}_{\rm tan}^{{\beta}}\mathrm{D}_{x'}\psi \cdot
\bigg(\p_t[\mathrm{D}_{\rm tan}^{{\beta}},\mathring{\rm c}_0]\mathrm{D}_{x'}\psi+\sum_{k=2,3} \p_k\big(\mathring{v}_k^+[\mathrm{D}_{\rm tan}^{{\beta}},\mathring{\rm c}_0]\mathrm{D}_{x'}\psi \big) \bigg).
\end{align*}
Then using Cauchy's inequality, integration by parts, and the Moser-type calculus inequalities \eqref{Moser3}--\eqref{Moser4}, we discover
\begin{align}
&-\mathcal{I}_{6}^{a}
+\frac{\mathfrak{s}}{2}\int_{\Sigma}
\frac{|\mathrm{D}_{\rm tan}^{{\beta}}\mathrm{D}_{x'}\psi|^2}{|\mathring{N}|^3}\d x'
\nonumber \\
&\qquad \lesssim_K
\|\mathrm{D}_{\rm tan}^{{\beta}}\mathrm{D}_{x'}\psi\|_{L^2(\Sigma_t)}^2
+\|[\mathrm{D}_{\rm tan}^{{\beta}},\mathring{\rm c}_0] \mathrm{D}_{x'}\psi \|_{H^1(\Sigma_t)}^2
\nonumber\\[1mm]
&\qquad \lesssim_K
\|\mathrm{D}_{x'}\psi\|_{H^m(\Sigma_t)}^2+\Big(1+\| (\mathring{V},  \mathring{\varPsi}, \mathrm{D}_{x'}\mathring{\varPsi})\|_{H^{m+2}(\Omega_T)}^2\Big) \|\mathrm{D}_{x'}\psi\|_{L^{\infty}(\Sigma_t)}^2.
\label{I6.a:es}
\end{align}
In view of \eqref{Moser3}--\eqref{Moser4}, we estimate the integral term $\mathcal{I}_{6}^{b}$ as
\begin{align}
|\mathcal{I}_{6}^{b}|
&\lesssim
\sum_{k=2,3}\big\|  \big(\mathrm{D}_{\rm tan}^{{\beta}}(\mathring{\rm c}_1\mathrm{D}_{x'}\psi),
\mathrm{D}_{\rm tan}^{{\beta}}(\mathring{\rm c}_2\psi),
\big[\mathrm{D}_{\rm tan}^{{\beta}}\mathrm{D}_{x'}, \mathring{v}_k^+\big]\p_k\psi\big)\big\|_{L^{2}(\Sigma_t)}^2 \nonumber \\
&\lesssim
\big\| \big(\mathring{\rm c}_1\mathrm{D}_{x'}\psi,\mathring{\rm c}_2\psi\big)\big\|_{H^{m}(\Sigma_t)}^2
+\big\|\big[\mathrm{D}_{\rm tan}^{{\beta}}\mathrm{D}_{x'}, \mathring{\rm c}_0\big]\mathrm{D}_{x'}\psi\big\|_{L^{2}(\Sigma_t)}^2
\lesssim_K \mathcal{M}_2(t),
\label{I6.b:es}
\end{align}
where $\mathcal{M}_2(t)$ is defined by \eqref{M2.cal}.
Regarding the term $\mathcal{I}_{6}^{c}$, we have
\begin{align}
\nonumber
|\mathcal{I}_{6}^{c}|
\lesssim_K\;&
\int_{\Sigma}   \big|\mathring{a}_1\big|\big |\mathrm{D}_{\rm tan}^{\beta} \psi\big|^2\mathrm{d}x'
+2\int_{\Sigma}\big|[\mathrm{D}_{\rm tan}^{\beta}, \mathring{a}_1]\psi \mathrm{D}_{\rm tan}^{\beta} \psi\big|\mathrm{d}x' \\[1mm]
& +\sum_{k=2,3}\big\| \big(\mathrm{D}_{\rm tan}^{\beta} \psi,\,
\p_t [\mathrm{D}_{\rm tan}^{\beta}, \mathring{a}_1] \psi,\,
\p_k (\mathring{v}_k^+ [\mathrm{D}_{\rm tan}^{\beta}, \mathring{a}_1] \psi
 )\big)\big \|_{L^2(\Sigma_t)}^2\nonumber\\
\lesssim_K\;&
\big\|\mathrm{D}_{\rm tan}^{\beta} \psi(t)\big\|^2_{L^2(\Sigma)}
+\mathcal{M}_2(t).
\label{I6.c:es}
\end{align}
Applying the Moser-type calculus inequalities \eqref{Moser3}--\eqref{Moser4} yields
\begin{align}
 |\mathcal{I}_{6}^{d}|
 \lesssim_K
\|\psi\|_{H^m(\Sigma_t)}^2+\Big(1+\| (\mathring{V},  \mathring{\varPsi}, \mathrm{D}_{x'}\mathring{\varPsi})\|_{H^{m+2}(\Omega_T)}^2\Big)\|\psi\|_{L^{\infty}(\Sigma_t)}^2 .
 \label{I6.d:es}
\end{align}

Now let us estimate the first term on the right-hand side of \eqref{I6.c:es}.
If $|\beta|\leq m-1$ or $\beta_2+\beta_3\geq 1$, then
\begin{align}
 \|\mathrm{D}_{\rm tan}^{\beta}\psi(t)\|_{L^2(\Sigma)}^2
 \lesssim \int_{\Sigma_t} |\mathrm{D}_{\rm tan}^{\beta}\psi||\p_t\mathrm{D}_{\rm tan}^{\beta}\psi|
 \lesssim \|( \psi,\mathrm{D}_{x'}\psi)\|_{H^m(\Sigma_t)}^2.
 \label{tan:es2}
\end{align}
Otherwise, $\beta_2=\beta_3=0$ and $\beta_0=m$.
For this case, it follows from the boundary condition \eqref{ELP3d} and integration by parts that
\begin{align}
 \nonumber
 \|\p_t^{m}\psi(t)\|_{L^2(\Sigma)}^2 & \nonumber
 \lesssim \|\p_t^{m-1}W_2^+(t)\|_{L^2(\Sigma)}^2
 +\|\mathring{v}_2^+\p_2\psi+\mathring{v}_3^+\p_3\psi+\mathring{a}_7\psi\|_{H^m(\Sigma_t)}^2\\
 & \lesssim   \boldsymbol{\epsilon} \| \p_t^{m-1}\p_1 W(t)\|_{L^2(\Omega)}^2
+\boldsymbol{\epsilon}^{-1}\| \p_t^{m-1} W(t)\|_{L^2(\Omega)}^2
+\mathcal{M}_2(t)\nonumber\\
 & \lesssim   \boldsymbol{\epsilon} \VERT  W(t)\VERT_{m}^2
+\boldsymbol{\epsilon}^{-1}\| W\|_{H^m(\Omega_t)}^2
+\mathcal{M}_2(t)
 \label{tan:es3}
\end{align}
for all $\boldsymbol{\epsilon}>0$.

It remains to make the estimate of  the last term in \eqref{tan:id1}.

If $|\beta|\leq m-1$, then using the trace theorem implies
\begin{align}
\left|\int_{\Sigma_{{t}}} \mathrm{D}_{\rm tan}^{\beta} (\mathring{\rm c}_1 \psi )\mathring{\rm c}_0 \mathrm{D}_{\rm tan}^{\beta}\mathcal{U}\right|
&\lesssim_K
\|\mathring{\rm c}_1 \psi\|_{H^{m-1}(\Sigma_{{t}})}
\|\mathcal{U}\|_{H^{m}(\Omega_{{t}})} \nonumber\\
&\lesssim_K \mathcal{M}_1(t)+\mathcal{M}_2(t).
\label{tan:es4}
\end{align}

If $\beta=(\beta_0,\beta_2,\beta_3)$ with $\beta_2\geq 1$ or $\beta_3\geq 1$,
then it follows from integration by parts and Moser-type calculus inequalities that
\begin{align}
\left|\int_{\Sigma_{{t}}} \mathrm{D}_{\rm tan}^{\beta} (\mathring{\rm c}_1 \psi )\mathring{\rm c}_0 \mathrm{D}_{\rm tan}^{\beta}\mathcal{U}\right|
&\lesssim
 \int_{\Sigma_{{t}}}
 \big|\p_k\big(\mathrm{D}_{\rm tan}^{\beta}  (\mathring{\rm c}_1 \psi)\mathring{\rm c}_0\big)\big|
 \big|\mathrm{D}_{\rm tan}^{\beta-\bm{e}_k}\mathcal{U}\big| \nonumber \\
&\lesssim_K   \mathcal{M}_2(t)+\|W\|_{H^m(\Omega_{{t}})}^2,
\label{tan:es5}
\end{align}
where $\bm{e}_2:=(0,1,0)^{\top}$ and $\bm{e}_3:=(0,0,1)^{\top}$.

If $\beta=(m,0,0)$, then
\begin{align}
\int_{\Sigma_{{t}}} \mathrm{D}_{\rm tan}^{\beta} (\mathring{\rm c}_1 \psi )\mathring{\rm c}_0 \mathrm{D}_{\rm tan}^{\beta}\mathcal{U}
=\int_{\Sigma_{{t}}} \p_t^{m} (\mathring{\rm c}_1 \psi )\mathring{\rm c}_0 \p_t^{m}\mathcal{U}
=\mathcal{I}_{7}+\mathcal{I}_{8}+\mathcal{I}_{9},
\label{tan:id3}
\end{align}
with
\begin{align}
\nonumber
&\mathcal{I}_{7}:=\int_{\Sigma}\p_t^{m}(\mathring{\rm c}_1 \psi)\mathring{\rm c}_0 \p_t^{m-1}\mathcal{U}\mathrm{d}x' ,\quad
\mathcal{I}_{8}:=\int_{\Sigma_t} \mathring{\rm c}_1\p_t^{m+1} \psi  \p_t^{m-1}\mathcal{U}, \\ \nonumber
&\mathcal{I}_{9}=\int_{\Sigma_t}
\Big(\big[\p_t^{m+1}, \mathring{\rm c}_1\big] \psi \mathring{\rm c}_0 \p_t^{m-1}\mathcal{U}
+\p_t^{m}(\mathring{\rm c}_1 \psi)\mathring{\rm c}_1 \p_t^{m-1}\mathcal{U}\Big).
\end{align}
For the integral term $\mathcal{I}_{7}$, we utilize the estimate \eqref{tan:es3} and the calculus inequality \eqref{Moser4} to infer
\begin{align}
|\mathcal{I}_{7}| \lesssim \,&
\|\p_t^m(\mathring{\rm c}_1 \psi)(t)\|_{L^2(\Sigma)}^2
+\|\p_t^{m-1} \mathcal{U}(t)\|_{L^2(\Sigma)}^2\nonumber\\
\lesssim \,&
\|\p_t^m \psi(t)\|_{L^2(\Sigma)}^2
+\|[\p_t^m, \mathring{\rm c}_1 ]\psi\|_{H^1(\Sigma_{{t}})}^2
+\|\p_t^{m-1} W(t)\|_{L^2(\Sigma)}^2\nonumber\\
 \lesssim  \,&
\mathcal{M}_2(t)
+\boldsymbol{\epsilon} \VERT  W(t)\VERT_{m}^2
+\boldsymbol{\epsilon}^{-1}\| W\|_{H^m(\Omega_t)}^2
\quad \textrm{for all }\boldsymbol{\epsilon}>0.
\label{I7:es}
\end{align}
Thanks to the boundary condition \eqref{ELP3d}, we get
\begin{align}
\mathcal{I}_{8}=
\underbrace{\int_{\Sigma_t} \mathring{\rm c}_1\p_t^{m} W_2^+  \p_t^{m-1}\mathcal{U}}_{\mathcal{I}_{8}^{a}}
+\underbrace{\int_{\Sigma_t} \mathring{\rm c}_1\p_t^{m}(\mathring{\rm c}_0 \mathrm{D}_{x'}\psi +\mathring{\rm c}_1 \psi)  \p_t^{m-1}\mathcal{U}}_{\mathcal{I}_{8}^{b}}.
\label{I8:id}
\end{align}
Passing the boundary integral $\mathcal{I}_{8}^a$ to the volume one yields
\begin{align}
\mathcal{I}_{8}^{a}=\,&-\int_{\Omega_t} \p_1\left(\mathring{\rm c}_1\p_t^{m} W_2^+  \p_t^{m-1}\mathcal{U}\right)\nonumber\\
=\,& -\int_{\Omega} \mathring{\rm c}_1\p_1\p_t^{m-1} W_2^+  \p_t^{m-1}\mathcal{U}\mathrm{d}x
+\int_{\Omega_t}
\mathring{\rm c}_2
\begin{pmatrix}
 \p_t^{m-1} \mathcal{U} \\ \p_t^{m} \mathcal{U} \\  \p_1\p_t^{m-1} \mathcal{U}
\end{pmatrix}\cdot
\begin{pmatrix}
 \p_t^m W_2^+\\  \p_1\p_t^{m-1} W_2^+
\end{pmatrix}
\nonumber\\
\geq \,&
-\boldsymbol{\epsilon} \VERT W(t)\VERT_{m}^2
-C(\boldsymbol{\epsilon}) C(K)\|W\|_{H^m(\Omega_t)}^2.
\label{I8:es1}
\end{align}
Apply the trace theorem and
the Moser-type calculus inequalities \eqref{Moser3}--\eqref{Moser4} to obtain
\begin{align}
\left|\mathcal{I}_{8}^{b}+\mathcal{I}_{9}\right|\lesssim \mathcal{M}_2(t)+\|W\|_{H^m(\Omega_{{t}})}^2.
\label{I8:es2}
\end{align}

We conclude the estimate \eqref{tan:est} by plugging \eqref{I5.cal:es}--\eqref{tan:id1} into \eqref{tan:es1} and using \eqref{I6.a:es}--\eqref{I8:es2}.
The proof is thus complete.
\end{proof}

\subsection{Proof of Theorem \ref{thm:tame}}
Combining the estimate \eqref{normal.est} with \eqref{tan:est},
we choose $\boldsymbol{\epsilon}>0$ small enough to derive
 \begin{align}
 \VERT  W (t)\VERT_{m}^2
 +\sum_{|\beta|\leq m}\|(\mathrm{D}_{\rm tan}^{\beta}\psi,
 \mathrm{D}_{\rm tan}^{\beta}\mathrm{D}_{x'}\psi)(t)\|_{L^2(\Sigma)}^2
 \lesssim_K   \mathcal{M}_1(t)  +\mathcal{M}_2(t),
 \label{fin:es1}
\end{align}
where $\mathcal{M}_1(t)$ and $\mathcal{M}_2(t)$ are defined by \eqref{M1.cal:def} and \eqref{M2.cal}, respectively.
By virtue of the Gr\"{o}nwall's inequality, from \eqref{fin:es1} we obtain
\begin{multline}
  \VERT  W (t)\VERT_{m}^2
 +\sum_{|\beta|\leq m}\|(\mathrm{D}_{\rm tan}^{\beta}\psi,
 \mathrm{D}_{\rm tan}^{\beta}\mathrm{D}_{x'}\psi)(t)\|_{L^2(\Sigma)}^2
\lesssim_K
\| \bm{f} \|_{H^m(\Omega_t)}^2 \\
  +\Big(1+\| (\mathring{V},  \mathring{\varPsi}, \mathrm{D}_{x'}\mathring{\varPsi})\|_{H^{m+2}(\Omega_T)}^2\Big)
\left(\|    ( \bm{f},W)\|_{L^{\infty}(\Omega_t)}^2+\|( \psi,\mathrm{D}_{x'}\psi)\|_{L^{\infty}(\Sigma_t)}^2\right).
\nonumber %
\end{multline}
Integrating the last estimate over $[0,T]$,
we use the embedding $H^3(\Omega_{T}) \hookrightarrow  L^{\infty}(\Omega_{T})$, $H^2(\Sigma_{T}) \hookrightarrow  L^{\infty}(\Sigma_{T})$
and take $T>0$ sufficiently small to infer
\begin{multline}
\|    W\|_{H^m(\Omega_T)}^2+\|( \psi,\mathrm{D}_{x'}\psi)\|_{H^{m}(\Sigma_T)}^2\\
 \lesssim_K
 T
\Big\{
\| \bm{f} \|_{H^m(\Omega_T)}^2 +\| (\mathring{V},  \mathring{\varPsi},\mathrm{D}_{x'}\mathring{\varPsi})\|_{H^{m+2}(\Omega_T)}^2{\mbox{\qquad\qquad \qquad} }
 \\
\times\left(\|    \bm{f}\|_{H^3(\Omega_T)}^2+\| W\|_{H^3(\Omega_T)}^2
+\|( \psi,\mathrm{D}_{x'}\psi)\|_{H^{2}(\Sigma_T)}^2\right)
\Big\}
 \label{fin:es2}
\end{multline}
for $m\geq 3$.
In view of \eqref{fin:es2} with $m=3$, we can find a sufficiently small constant $T_0>0$, depending on $K$ ({\it cf.}~\eqref{bas1d3}), such that if $0<T\leq T_0$, then
\begin{align}
 \|    W\|_{H^3(\Omega_T)}^2+\|( \psi,\mathrm{D}_{x'}\psi)\|_{H^{3}(\Sigma_T)}^2   \lesssim_{K}
 \| \bm{f} \|_{H^3(\Omega_T)}^2.
 \nonumber
\end{align}
Plugging the above estimate into \eqref{fin:es2} implies
\begin{align}
 &\nonumber\|    W\|_{H^m(\Omega_T)}^2+\|( \psi,\mathrm{D}_{x'}\psi)\|_{H^{m}(\Sigma_T)}^2\\
 &\quad  \lesssim_{K}
 \| \bm{f} \|_{H^m(\Omega_T)}^2 +\| (\mathring{V},  \mathring{\varPsi},\mathrm{D}_{x'}\mathring{\varPsi})\|_{H^{m+2}(\Omega_T)}^2\|  \bm{f} \|_{H^3(\Omega_T)}^2
 \quad\textrm{for } m\geq 3.\label{fin:es3}
\end{align}

In Section \ref{sec:lin1}, we have proved that for $(f^{\pm},g)\in H^1(\Omega_T)\times H^{3/2}(\Sigma_T)$ vanishing in the past, there exists a unique solution $(W,\psi)\in H^1(\Omega_T)\times H^1(\Sigma_T)$ to the reduced problem \eqref{ELP3}.
Using the arguments in \cite[Chapter 7]{CP82MR0678605} and the energy estimate \eqref{fin:es3},
one can establish the existence and uniqueness of solutions $(W,\psi)$ of {the} problem \eqref{ELP3} in $H^m(\Omega_T)\times H^{m}(\Sigma_T)$ for any $m\geq 3$.
As a consequence,  the problem \eqref{ELP1} admits a unique solution $(\dot{V}^{\pm},\psi)$ in $H^m(\Omega_T)\times H^{m}(\Sigma_T)$.
The tame estimate \eqref{tame:es} for {the} problem \eqref{ELP1}
follows by combining \eqref{fin:es3} with \eqref{V.natural:es}.
The proof of Theorem \ref{thm:tame} is finished.

%\newpage
\section{Nash--Moser Iteration} \label{sec:NM}
This section is devoted to showing the nonlinear stability of MHD contact discontinuities with surface tension, or equivalently, solving the nonlinear problem \eqref{MCD0}.
Our analysis is based on a modified Nash--Moser iteration scheme.

\subsection{Reducing to Zero Initial Data}
To apply Theorems \ref{thm:lin} and \ref{thm:tame}, we will reduce the nonlinear problem \eqref{MCD0} to that with zero initial data via the approximate solutions.
For this purpose, we need to impose suitable compatibility conditions on the initial data.

Take $m\in\mathbb{N}$ with $m\geq 3$.
Assume that the initial data $(U_0^+, U_0^-,\varphi_0)$ satisfy
$\widetilde{U}_0^{\pm}:=U_0^{\pm}-\widebar{U}^{\pm} \in H^{m+3/2}(\Omega)$
and $\varphi_0\in H^{m+2}(\mathbb{R}^{2})$,
where $\widebar{U}^{\pm}$ are the constant states defined by \eqref{U.bar:def}.
We assume without loss of generality that
$\|\varphi_0\|_{L^{\infty}(\mathbb{R}^{2})}\leq \frac{1}{4}$,
and hence
\begin{align}\label{CA1}
 \pm \p_1 \varPhi_0^{\pm}(x) \geq  \frac{3}{4}>0\quad  \textrm{in } \Omega,
\quad
\end{align}
where $\varPhi_0^{\pm}(x):={\pm}x_1+\widetilde{\varPhi}_0^{\pm}(x)$
with $\widetilde{\varPhi}_0^{\pm}(x) :=\chi(\pm x_1)\varphi_0(x')$.
Let us define the perturbations
$(\widetilde{U}^{\pm},\widetilde{\varPhi}^{\pm}) : = (U^{\pm} -\widebar{U}^{\pm},
\varPhi^{\pm} {\mp} x_1)$, and
\begin{align*}
\widetilde{U}_{(\ell)}^{\pm}:=\p_t^{\ell}\widetilde{U}^{\pm}\big|_{t=0},\ \
\varphi_{(\ell)}:=\p_t^{\ell}\varphi\big|_{t=0},\ \
\widetilde{\varPhi}_{(\ell)}^{\pm}:= \p_t^{\ell}\widetilde{\varPhi}^{\pm}\big|_{t=0}
\quad\textrm{for } \ell \in\mathbb{N}.
\end{align*}
It follows from \eqref{varPhi:def} that
\begin{align*}
\widetilde{\varPhi}_{(\ell)}^{\pm}(x)=\chi(\pm x_1)\varphi_{(\ell)}(x'),
\quad
\big(\widetilde{U}_{(0)}^{\pm},\,\varphi_{(0)}^{\pm},\,\widetilde{\varPhi}_{(0)}^{\pm}\big)
=\big(\widetilde{U}_0^{\pm},\,\varphi_0^{\pm},\,\widetilde{\varPhi}_{0}^{\pm}\big).
\end{align*}
Applying Leibniz's rule to the last condition in \eqref{MCD0b} yields
\begin{align}
 \varphi_{(\ell+1)} = v_{1(\ell)}^{+}|_{x_1=0}-\sum_{i=0}^{\ell} \sum_{j=2,3}
 \begin{pmatrix}
   \ell \\ i
 \end{pmatrix}  \p_j\varphi_{(\ell-i)}v_{j(i)}^+|_{x_1=0},
 \label{trace.id1}
\end{align}
where $\left(\begin{smallmatrix}
 \ell \\[0.3mm] i
\end{smallmatrix}\right):=\frac{\ell !}{i!(\ell-i)!}$ is the binomial coefficient.
Under the hyperbolicity condition \eqref{hyperbolicity},
we can rewrite the equations \eqref{MCD0} as
\begin{align}\nonumber %\label{tilde.U.Phi}
 \p_t \widetilde{U}^{\pm}=\bm{G}\big(\mathcal{W}^{\pm}\big)
\qquad \textrm{for }
\mathcal{W}^{\pm}:=(\widetilde{U}^{\pm},\nabla \widetilde{U}^{\pm},\mathrm{D} \widetilde{\varPhi}^{\pm})^{\top}\in\mathbb{R}^{36},
\end{align}
where $\bm{G}$ is a certain $C^{\infty}$--function vanishing at the origin.
Employ the generalized Fa\`a di Bruno's formula (see \cite[Theorem 2.1]{M00MR1781515})
to find
\begin{align} \label{trace.id2}
&\widetilde{U}_{(\ell+1)}^{\pm}
 =\sum_{\substack{\alpha_{i}\in\mathbb{N}^{36}\\ |\alpha_1|+\cdots+\ell |\alpha_{\ell}|=\ell}}
 \mathrm{D}_{\mathcal{W}}^{\alpha_1+\cdots+\alpha_\ell}\bm{G}\big(\mathcal{W}_{(0)}^{\pm}\big)
 \ell!\prod_{i=1}^\ell\frac{1}{\alpha_{i}!}
 \left(\frac{\mathcal{W}_{(i)}^{\pm} }{i!}\right)^{\alpha_{i}},
\end{align}
where $\mathcal{W}_{(i)}^{\pm}:=(\widetilde{U}_{(i)}^{\pm},\,\nabla \widetilde{U}_{(i)}^{\pm},\,\mathrm{D}\widetilde{\varPhi}_{(i)}^{\pm})^{\top}$.

By virtue of the relations \eqref{trace.id1}--\eqref{trace.id2}, we can determine the traces $\widetilde{U}_{(\ell)}^{\pm}$ and ${\varphi}_{(\ell)}$ inductively in the following lemma (see \cite[Lemma 4.2.1]{M01MR1842775} for the proof).

\begin{lemma}\label{lem:CA1}
Let $m\geq 3$ be an integer.
Assume that the initial data $(U_0^+, U_0^-,\varphi_0)$ satisfy
the hyperbolicity condition \eqref{thm:H1},
$\|\varphi_0\|_{L^{\infty}(\mathbb{R}^{2})}\leq \frac{1}{4}$,
and
$(\widetilde{U}_0^{\pm} ,\varphi_0 ) \in H^{m+3/2}(\Omega)\times H^{m+2}(\mathbb{R}^{2})$
for $\widetilde{U}_0^{\pm}:=U_0^{\pm}-\widebar{U}^{\pm}$.
Then equations \eqref{trace.id1}--\eqref{trace.id2} determine
$\widetilde{U}_{(\ell)}^{\pm} \in H^{m+3/2-\ell}(\Omega)$
and $\varphi_{(\ell)} \in H^{m+2-\ell}(\mathbb{R}^{2})$ for $\ell=1,\ldots,m$.
Moreover,
\begin{align}
 \sum_{\ell=0}^{m}\sum_{\pm}\Big(\|\widetilde{U}_{(\ell)}^{\pm} \|_{H^{m+3/2-\ell}(\Omega)} +\|{\varphi}_{(\ell)}\|_{H^{m+2-\ell}(\mathbb{R}^{2})}\Big)
 \nonumber   \leq C M_0
\end{align}
for
\begin{align}
 \label{M0}
 M_0:=\| (\widetilde{U}_0^+,\widetilde{U}_0^- ) \|_{H^{m+3/2}(\Omega)}
 +\|\varphi_0\|_{H^{m+2}(\mathbb{R}^{2})},
\end{align}
where $C>0$ depends only on $m$, $\|\widetilde{U}_{0}^{\pm}\|_{W^{1,\infty}(\Omega)}$, and $\|{\varphi}_{0}\|_{W^{1,\infty}(\mathbb{R}^{2})}$.
\end{lemma}

In light of the definition of $\mathcal{H}(\varphi)$ in \eqref{H.cal:def}, we set
\begin{align*}
\zeta:=\mathrm{D}_{x'}\varphi\in\mathbb{R}^2 \quad
\textrm{and} \quad
\mathfrak{f}(\zeta):=\frac{\zeta}{\sqrt{1+|\zeta|^2}}.
\end{align*}
Then it follows from the first condition in \eqref{MCD0b} that for $\zeta_{(i)}:=\mathrm{D}_{x'}\varphi_{(i)}$,
\begin{align}
\label{compat1}
 \big[p_{(\ell)}\big]=
 \sum_{\substack{\alpha_{i}\in\mathbb{N}^{2}\\ |\alpha_1|+\cdots+\ell |\alpha_{\ell}|=\ell}} \mathfrak{s}\mathrm{D}_{x'}\cdot
 \left(\mathrm{D}_{\zeta}^{\alpha_1+\cdots+\alpha_\ell}\mathfrak{f}\big(\zeta_{(0)}\big)\ell! \prod_{i=1}^\ell\frac{1}{\alpha_{i}!}
 \left(\frac{\zeta_{(i)}}{i!}\right)^{\alpha_{i}}\right).
\end{align}
Setting $ H_{\tau_1}^{\pm}:=H_1^{\pm}\p_2\varPhi^{\pm}+H_2^{\pm}$
and $H_{\tau_2}^{\pm}:=H_1^{\pm}\p_3\varPhi^{\pm}+H_3^{\pm}$,
we have
\begin{align}
\nonumber &
H_{\tau_1(\ell)}^{\pm}:=\p_t^{\ell}H_{\tau_1}^{\pm}\big|_{t=0}
=\sum_{i=0}^{\ell} \begin{pmatrix}
 \ell \\ i
\end{pmatrix} H_{1(i)}^{\pm}\p_2\varPhi_{(\ell-i)}^{\pm}+H_{2(\ell)}^{\pm},\\
\nonumber &
H_{\tau_2(\ell)}^{\pm}:=\p_t^{\ell}H_{\tau_2}^{\pm}\big|_{t=0}
=\sum_{i=0}^{\ell} \begin{pmatrix}
 \ell \\ i
\end{pmatrix} H_{1(i)}^{\pm}\p_3\varPhi_{(\ell-i)}^{\pm}+H_{3(\ell)}^{\pm}.
\end{align}

According to the boundary conditions \eqref{MCD0b}, we introduce the following terminology.
\begin{definition}\label{def:1}
Assume that all the conditions of Lemma \ref{lem:CA1} are satisfied.
Then the initial data $(U_0^+,U_0^-,\varphi_0)$ are said to be compatible up to order $m$
if for $\ell=0,\ldots,m$,
the traces $\widetilde{U}_{(\ell)}^{\pm}$ and $\varphi_{(\ell)}$ determined by \eqref{trace.id1}--\eqref{trace.id2} satisfy the boundary conditions \eqref{compat1} and
\begin{align}
\label{compat2}
\big[v_{(\ell)} \big]=0,\quad
\big[H_{\tau_1(\ell)}\big]=0,\quad
\big[H_{\tau_2(\ell)}\big]=0
\quad \textrm{on }\Sigma.
\end{align}
\end{definition}

We are now ready to construct the approximate solutions.

\begin{lemma}\label{lem:app}
Suppose that all the conditions of Lemma \ref{lem:CA1} are satisfied.
Suppose further that the initial data $(U_0^+, U_0^-,\varphi_0)$ are compatible up to order $m$
and satisfy the constraints \eqref{BC3a} and \eqref{inv2b}.
Then there are positive constants $T_1(M_0)$ and $C(M_0)$ depending on $M_0$ ({\it cf.}~\eqref{M0}), such that if $0<T\leq T_1(M_0)$, then there exist functions $U^{a\pm}$ and $\varphi^a$ satisfying
\begin{subequations}
\label{app1}
\begin{alignat}{3}
\label{app1a}
&\mathbb{B}(U^{a+},U^{a-},\varphi^a)=0,
\quad [H^{a}]=0 \quad
\textrm{on }  \Sigma_T,\\
\label{app1b}
&U^{a\pm}\big|_{t=0}=U_0^{\pm}\ \   \textrm{in } \Omega,
\quad \varphi^a\big|_{t=0}=\varphi_0 \ \   \textrm{on } \Sigma.
\end{alignat}
\end{subequations}
Moreover,
\begin{subequations}
\label{app2}
 \begin{alignat}{3}
 \label{app2a}
&\p_t^{\ell}\mathbb{L}_{\pm}(U^{a\pm},\varPhi^{a\pm})\big|_{t=0}=0\quad
\textrm{in } \Omega
\quad \textrm{for } \ell=0,\ldots,m-1,\\[1.5mm]
&\|(\widetilde{U}^{a+},\widetilde{U}^{a-} )\|_{H^{m+1}(\Omega_T)}
+\|\varphi^a\|_{H^{m+5/2}(\Sigma_T)}  \leq C (M_0) ,
\label{app2b}\\[0.5mm]
\label{app2c}
&\rho_*<\inf_{\Omega_T} \rho^{\pm}(U^{a\pm}) \leq \sup_{\Omega_T} \rho^{\pm}(U^{a\pm})<\rho^*,
\quad \big|\p_1\varPhi^{a\pm}\big|\geq  \frac{5}{8}\ \ \textrm{in }  \Omega_T,\\
\label{app2d}
& \left|H_1^{a\pm}-H_2^{a\pm}\p_2\varphi^a-H_3^{a\pm}\p_3\varphi^a \right|\geq  \frac{3}{4}\kappa>0\quad  \textrm{on }  \Sigma_T,
 \end{alignat}
\end{subequations}
where $\widetilde{U}^{a\pm}:=U^{a\pm}-\widebar{U}^{\pm}$ and
$\varPhi^{a\pm} :=\pm x_1 +\varPsi^{a\pm} $ with
$\varPsi^{a\pm} :=\chi(\pm x_1)\varphi^{a}$.
\end{lemma}
\begin{proof}
Since $\|\varphi_0\|_{L^{\infty}(\mathbb{R}^2)}\leq \frac{1}{4}$, we can take $\varphi^a\in H^{m+5/2}(\mathbb{R}^3)$ to satisfy
\begin{align*}
 \|\varphi^a\|_{L^{\infty}(\mathbb{R}^3)}\leq \frac{3}{8},\quad
 \p_t^{\ell}\varphi^a\big|_{t=0}=\varphi_{(\ell)}\in  H^{m+2-\ell}(\mathbb{R}^2)
 \textrm{ \ for \ }\ell=0,\ldots,m.
\end{align*}
Define $\varPhi^{a\pm}:=\pm x_1+\chi(\pm x_1)\varphi^a(t,x')$, so that
$|\p_1\varPhi^{a\pm} |\geq \frac{5}{8}$ in $ \mathbb{R}\times \Omega$.

Using the compatibility conditions \eqref{compat2} and the initial constraint $[H_0]=0$, we can prove as in \cite[Lemma 3]{MTT18MR3766987} that
\begin{align*}
\big[H_{(\ell)}\big]=0\quad\textrm{on }\Sigma \quad \textrm{for }\ell=0,\ldots,m.
\end{align*}
Then we apply the lifting result in \cite[Theorem 2.3]{LM72MR0350178} to find
$\tilde{p}^{a\pm}\in H^{m+1}(\mathbb{R}\times \Omega)$ and
$(\tilde{v}_2^{a\pm},\tilde{v}_3^{a\pm},\widetilde{H}^{a\pm},\widetilde{S}^{a\pm})\in H^{m+2}(\mathbb{R}\times \Omega)$ such that
\begin{align*}
&[\tilde{p}^a]=\mathfrak{s}\mathcal{H}(\varphi^a)%\in H^{m+1/2}(\mathbb{R}\times \Sigma)
, \quad  [\tilde{v}_2^a]= [\tilde{v}_3^a]= 0,\quad \big[\widetilde{H}^a\big]= 0
\quad \textrm{on }\Sigma,\\[1mm]
&\p_t^{\ell}(\tilde{p}^{a\pm},\tilde{v}_2^{a\pm},\tilde{v}_3^{a\pm},\widetilde{H}^{a\pm},\widetilde{S}^{a\pm})\big|_{t=0}\\
&\quad =(\tilde{p}_{(\ell)} ^{\pm},\tilde{v}_{2(\ell)} ^{\pm},\tilde{v}_{3(\ell)} ^{\pm},\widetilde{H}_{(\ell)} ^{\pm},\widetilde{S}_{(\ell)} ^{\pm})
\quad \textrm{in }\Omega \quad
\textrm{for }\ell=0,\ldots,m.
\end{align*}
Set $(p^{a\pm},v_2^{a\pm},v_3^{a\pm}, {H}^{a\pm}, {S}^{a\pm}):=
(\tilde{p}^{a\pm},\tilde{v}_2^{a\pm},\tilde{v}_3^{a\pm},\widetilde{H}^{a\pm},\widetilde{S}^{a\pm})
+(\bar{p},\bar{v}_2,\bar{v}_3,\widebar{H},\widebar{S}^{\pm}).$

By virtue of the trace theorem,
the first condition in \eqref{compat2}, and the relation \eqref{trace.id1}, we can choose
${v}_1^{a\pm}=\tilde{v}_1^{a\pm}\in H^{m+2}(\mathbb{R}\times \Omega)$ to satisfy
\begin{align*}
&{v}_1^{a\pm}|_{x_1=0}=\p_t\varphi^a+ \p_2\varphi^a v_2^{a+}|_{x_1=0}
 +\p_3\varphi^a v_3^{a+}|_{x_1=0}\in H^{m+3/2}(\mathbb{R}^3),\\
&\p_t^{\ell} \tilde{v}_{1(\ell)}^{a\pm}\big|_{t=0}=\tilde{v}_{1(\ell)}^{\pm}
\quad \textrm{in }\Omega
\qquad \textrm{for }\ell=0,\ldots,m.
\end{align*}
We have already obtained \eqref{app1} and the second relation in \eqref{app2c}.

The equations \eqref{app2a} follow directly from \eqref{trace.id2}.
Use Lemma \ref{lem:CA1} and the continuity of the lifting operators to derive the inequality \eqref{app2b}.
The inequality \eqref{app2d} and the first relation in \eqref{app2c} follow from \eqref{app2b} by taking $T>0$ sufficiently small. The proof of the lemma is complete.
\end{proof}

We call the vector-valued function $(U^{a+},U^{a-},\varphi^a)$ constructed in Lemma~\ref{lem:app}
the approximate solution to the problem \eqref{MCD0}.
Define
\begin{align}\label{f^a}
 f^{a\pm}:=\left\{\begin{aligned}
  & -\mathbb{L}_{\pm}(U^{a\pm},\varPhi^{a\pm}) \quad &\textrm{if \;}t>0,\\
  & 0 \quad &\textrm{if \;}t<0.
 \end{aligned}\right.
\end{align}
Utilize the Moser-type calculus and embedding inequalities to deduce
\begin{align}\label{f^a:est}
f^{a\pm}\in H^{m}(\Omega_T),
\quad
 \|f^{a\pm}\|_{ H^{m}(\Omega_T)}\leq
 \delta_0\left(T\right),
\end{align}
where $\delta_0(T)$ tends to zero as $T\to 0$.
In view of \eqref{app1}--\eqref{f^a},
we infer that  $(U^+,U^-,\varphi)$ is a solution of the nonlinear problem \eqref{MCD0} on $[0,T]\times \Omega$,
if $(V^+,V^-,\psi)=(U^+,U^-,\varphi)-(U^{a+},U^{a-},\varphi^a)$ solves
\begin{align} \label{P.new}
 \left\{
 \begin{aligned}
  &\mathcal{L}(V,\varPsi )=f^{a }:= (f^{a+},\,f^{a-} )^{\top}
  \ &&\textrm{in }\Omega_T,\\
  &\mathcal{B}(V,\psi):=\mathbb{B}(U^{a+}+V^+,U^{a-}+V^-,\varphi^a+\psi)=0
  \ &&\textrm{on }\Sigma_T,\\
  &(V,\psi)=0,\ &&\textrm{if }t< 0,
 \end{aligned}\right.
\end{align}
where $V:=(V^+,V^-)^{\top}$,
$\varPsi:=(\varPsi^+,\varPsi^-)^{\top}$ with $\varPsi^{\pm}:=\chi(\pm x_1)\psi$, and
\begin{align*}
\mathcal{L}(V,\varPsi ):=\begin{pmatrix}
\mathbb{L}_{+}(U^{a+}+V^{+},\varPhi^{a+}+\varPsi^{+})-\mathbb{L}_{+}(U^{a+},\varPhi^{a+})\\
\mathbb{L}_{-}(U^{a-}+V^{-},\varPhi^{a-}+\varPsi^{-})-\mathbb{L}_{-}(U^{a-},\varPhi^{a-})
\end{pmatrix}.
\end{align*}
It follows from \eqref{app1a} that $(V,\psi)\equiv 0$ satisfies \eqref{P.new} for $t<0$.
Therefore, the original problem on $[0,T]\times \Omega$ is reformulated as a problem in $\Omega_T$ whose solutions vanish in the past.

\subsection{Iteration Scheme and Inductive Hypothesis}

We first list the important properties of smooth operators \cite{A89MR976971,CS08MR2423311,T09CPAMMR2560044}.
\begin{proposition}\ \label{pro:smooth}
Let $T>0$ and $m\in \mathbb{N}$ with $m\geq 4$.
Denote by $\mathscr{F}^s(\Omega_T)$ the class of $H^s(\Omega_T)$--functions vanishing in the past.
Then there exists a family $\{\mathcal{S}_{\theta}\}_{\theta\geq 1}$ of smoothing operators
$u=(u^+,u^-)  \mapsto \mathcal{S}_{\theta}u= ((\mathcal{S}_{\theta}u)^+,(\mathcal{S}_{\theta}u)^-)$
from $ \mathscr{F}^3(\Omega_T)\times\mathscr{F}^3(\Omega_T)$ to
$\bigcap_{s\geq 3}\mathscr{F}^s(\Omega_T)\times\mathscr{F}^s(\Omega_T)$,
such that
 \begin{subequations}\label{smooth.p1}
  \begin{alignat}{2}
   \label{smooth.p1a}&\|\mathcal{S}_{\theta} u\|_{H^{\ell}(\Omega_T)}\lesssim_m \theta^{(\ell-j)_+}\|u\|_{H^{j}(\Omega_T)}
   &&\ \textrm{ for }j,\ell\in[1,m],\\[1.5mm]
   \label{smooth.p1b}&\|\mathcal{S}_{\theta} u-u\|_{H^{\ell}(\Omega_T)}\lesssim_m \theta^{\ell-j}\|u\|_{H^{j}(\Omega_T)}
   &&\  \textrm{ for }1\leq \ell\leq j\leq m,\\
   \label{smooth.p1c}&\left\|\frac{\d}{\d \theta}\mathcal{S}_{\theta} u\right\|_{H^{\ell}(\Omega_T)}
   \lesssim_m \theta^{\ell-j-1}\|u\|_{H^{j}(\Omega_T)}&&\   \textrm{ for }j,\ell\in[1,m],
  \end{alignat}
 \end{subequations}
where $j,\ell \in \mathbb{N}$ and $(\ell-j)_+:=\max\{0,\ell-j\}$. Moreover,
\begin{align}
\label{smooth.p2}
\left\| \left[\mathcal{S}_{\theta} u\right]\right\|_{H^{\ell}(\Sigma_T)}
\lesssim_m \theta^{(\ell+1-j)_+}\|[u]\|_{H^{j}(\Sigma_T)} \   \textrm{ for }j,\ell=1,\ldots,m,
\end{align}
where $\left[\mathcal{S}_{\theta} u\right]:=(\mathcal{S}_{\theta} u)^+-(\mathcal{S}_{\theta} u)^-$
and $[u]:=u^+-u^-$ on $\Sigma_{{T}}$.

There is another family of smoothing operators {(}still denoted by $\mathcal{S}_{\theta}${)} acting on functions that are defined on the boundary $\Sigma_T$ and satisfy the properties \eqref{smooth.p1} with norms $\|\cdot\|_{H^{j}(\Sigma_T)}$.
\end{proposition}

\vspace*{1mm}
Let us follow \cite{CS08MR2423311,T09CPAMMR2560044,CSW19MR3925528} to describe the iteration scheme for the reformulated problem \eqref{P.new}.

\vspace*{2mm}
\noindent{\bf Assumption\;(A-1)}: {\it
Take $ (V_0^{\pm}, \psi_0)=0$.
Let $(V_k^{\pm},\psi_k)$ be given and vanish in the past for $k=0,\ldots,{n}$.
Set $\varPsi_k^{\pm}:=\chi(\pm x_1)\psi_k$ for $k=0,\ldots,{n}$.
}

\vspace*{2mm}
We consider
\begin{align}\label{NM0}
 V_{{n}+1}^{\pm}=V_{{n}}^{\pm}+\delta V_{{n}}^{\pm},
 \quad \psi_{{n}+1}=\psi_{{n}}+\delta \psi_{{n}},
 \quad  \delta\varPsi_{{n}}^{\pm}:=\chi({\pm} x_1)\delta \psi_{{n}}.
\end{align}
The above differences $\delta V_{{n}}^{\pm}$ and $\delta \psi_{{n}}$ will solve the effective linear problem
\begin{align} \label{effective.NM}
 \left\{\begin{aligned}
  &\mathbb{L}_{e\pm}'(U^{a\pm}+V_{{n}+1/2}^{\pm},\varPhi^{a\pm}+\varPsi_{{n}+1/2}^{\pm})\delta \dot{V}_{{n}}^{\pm}=f_{{n}}^{\pm}
  \ \  &&\textrm{in }\Omega_T,\\
  & \mathbb{B}_e'(U^a+V_{{n}+1/2},\varphi^a+\psi_{{n}+1/2})(\delta \dot{V}_{{n}},\delta\psi_{{n}})=g_{{n}}
  \ \ &&\textrm{on }\Sigma_T,\\
  & (\delta \dot{V}_{{n}},\delta\psi_{{n}})=0\ \ &&\textrm{for }t<0,
 \end{aligned}\right.
\end{align}
where $(V_{{n}+1/2},\psi_{{n}+1/2})$ is a smooth modified state
such that $(U^a+V_{{n}+1/2},\varphi^a+\psi_{{n}+1/2})$ satisfies
the constraints \eqref{bas1a}--\eqref{bas1d},
$\varPsi_{{n}+1/2}^{\pm}:=\chi(\pm x_1)\psi_{n+1/2}$,
and
$\delta \dot{V}_{{n}}:=(
 \delta \dot{V}_{{n}}^+ ,\, \delta \dot{V}_{{n}}^-)^{\top}$ is the good unknown ({\it cf.}~\eqref{good}) with
\begin{align} \label{good.NM}
 \delta \dot{V}_{{n}}^{\pm}:=\delta V_{{n}}^{\pm}-\frac{\delta\varPsi_{{n}}^{\pm} }{\p_1 (\varPhi^{a\pm}+\varPsi_{{n}+1/2}^{\pm})}\p_1 (U^{a\pm}+V_{{n}+1/2}^{\pm}).
\end{align}
See Proposition \ref{pro:modified} for the construction and estimate of $(V_{{n}+1/2},\psi_{{n}+1/2})$.
The source terms $f_{{n}}:=(f_n^+,f_n^-)^{\top}$  and $g_n$ will be specified through the accumulated error terms at Step ${n}$.

\vspace*{2mm}
\noindent{\bf Assumption (A-2)}: {\it Set $(e_0,\tilde{e}_0,g_0):=0$ and $f_0:=\mathcal{S}_{\theta_0}f^{a}$ for $\theta_0\geq 1$ sufficiently large.
Let $(f_k,g_k,e_k,\tilde{e}_k)$ be given and vanish in the past for $k=1,\ldots,{n}-1$.}

\vspace*{2mm}

Under Assumptions {\bf (A-1)}--\textbf{(A-2)},
we calculate the accumulated error terms at Step $n$ by
\begin{align}  \label{E.E.tilde}
 E_{{n}}:=\sum_{k=0}^{{n}-1}e_{k},\quad \widetilde{E}_{{n}}:=\sum_{k=0}^{{n}-1}\tilde{e}_{k}.
\end{align}
Then we compute $f_{{n}}$ and $g_{{n}}$ from
\begin{align} \label{source}
 \sum_{k=0}^{{n}} f_k+\mathcal{S}_{\theta_{{n}}}E_{{n}}=\mathcal{S}_{\theta_{{n}}}f^a,
 \quad
 \sum_{k=0}^{{n}}g_k+\mathcal{S}_{\theta_{{n}}}\widetilde{E}_{{n}}=0,
\end{align}
where $\mathcal{S}_{\theta_{{n}}}$ are the smoothing operators defined in Proposition \ref{pro:smooth} with $\theta_{{n}}:=(\theta^2_0+{n})^{1/2}$.
Once $f_n$ and $g_n$ are specified, we can employ Theorem \ref{thm:lin} to take $(\delta \dot{V}_{{n}},\delta \psi_{{n}})$ as the unique solutions of the problem \eqref{effective.NM}.
Then we get $\delta V_{{n}}^{\pm}$ and $(V_{{n}+1}^{\pm},\psi_{{n}+1})$ from \eqref{good.NM} and \eqref{NM0} respectively.

Let us determine the error terms $e_n$ and $\tilde{e}_n$ via the decompositions
\begin{align}
 \nonumber&\mathcal{L}(V_{{n}+1},\varPsi_{{n}+1})-\mathcal{L}(V_{{n}},\varPsi_{{n}}) \\
\nonumber &\quad = \mathbb{L}'(U^a+V_{{n}},\varPhi^a+\varPsi_{{n}})(\delta V_{{n}},\delta\varPsi_{{n}})+e_{{n}}'\\
 \nonumber&\quad  = \mathbb{L}'(U^a+\mathcal{S}_{\theta_{{n}}}V_{{n}},\varPhi^a+\mathcal{S}_{\theta_{{n}}}\varPsi_{{n}})(\delta V_{{n}},\delta\varPsi_{{n}})+e_{{n}}'+e_{{n}}''\\
 \nonumber&\quad = \mathbb{L}'(U^a+V_{{n}+1/2},\varPhi^a+\varPsi_{{n}+1/2})(\delta V_{{n}},\delta\varPsi_{{n}})+e_{{n}}'+e_{{n}}''+e_{{n}}'''\\
 \label{decom1}&\quad = \mathbb{L}_e'(U^a+V_{{n}+1/2},\varPhi^a+\varPsi_{{n}+1/2})\delta \dot{V}_{{n}}+e_{{n}}'+e_{{n}}''+e_{{n}}'''+e_{{n}}^{(4)}
\end{align}
and
\begin{align}
 \nonumber&\mathcal{B}(V_{{n}+1},\psi_{{n}+1})-\mathcal{B}(V_{{n}},\psi_{{n}})\\
 \nonumber&\quad = \mathbb{B}'(U^a+V_{{n}},\varphi^a+\psi_{{n}})(\delta V_{{n}},\delta\psi_{{n}})+\tilde{e}_{{n}}'\\
 \nonumber&\quad = \mathbb{B}'(U^a+\mathcal{S}_{\theta_{{n}}}V_{{n}},\varphi^a+\mathcal{S}_{\theta_{{n}}}\psi_{{n}})(\delta V_{{n}},\delta\psi_{{n}})+\tilde{e}_{{n}}'+\tilde{e}_{{n}}''\\
 \label{decom2}&\quad  =\mathbb{B}_e'(U^a+V_{{n}+1/2},
 \varphi^a+\psi_{{n}+1/2})(\delta \dot{V}_{{n}} ,\delta\psi_{{n}})+\tilde{e}_{{n}}'+\tilde{e}_{{n}}''+\tilde{e}_{{n}}''',
\end{align}
where we write %
{\small
\begin{align*}
\mathbb{L}'\big(\mathring{U} ,\mathring{\varPhi} \big)(V ,\varPsi )
:=\begin{pmatrix}
\mathbb{L}_+'\big(\mathring{U}^+ ,\mathring{\varPhi}^+ \big)(V^+ ,\varPsi^+)\\[1mm]
\mathbb{L}_-'\big(\mathring{U}^-,\mathring{\varPhi}^-\big)(V^-,\varPsi^-)
\end{pmatrix},
 \
\mathbb{L}_{e}'\big(\mathring{U} ,\mathring{\varPhi} \big)\dot{V}
:=\begin{pmatrix}
 \mathbb{L}_{e+}'\big(\mathring{U}^+ ,\mathring{\varPhi}^+ \big)\dot{V}^+ \\[1mm]
 \mathbb{L}_{e-}'\big(\mathring{U}^-,\mathring{\varPhi}^-\big)\dot{V}^-
\end{pmatrix}.
\end{align*}
}%
Moreover, it follows from \eqref{Alinhac} that $e_{{n}}^{(4)}=(e_{{n}}^{(4)+},e_{{n}}^{(4)-})^{\top}$ for
\begin{align}\label{error.D}
	e_{{n}}^{(4)\pm} :=\frac{\delta \varPsi_{{n}}^{\pm}}{\p_1(\varPhi^{a\pm}+\varPsi_{{n}+1/2}^{\pm})}
	\p_1\mathbb{L}_{\pm}(U^{a\pm}+V_{{n}+1/2}^{\pm},\varPhi^{a\pm}+\varPsi_{{n}+1/2}^{\pm}).
\end{align}
Then the description of the iteration scheme is finished by setting
\begin{align} \label{e.e.tilde}
 e_{{n}}:=e_{{n}}'+e_{{n}}''+e_{{n}}'''+e_{{n}}^{(4)}\quad
\textrm{and} \quad
 \tilde{e}_{{n}}:=\tilde{e}_{{n}}'+\tilde{e}_{{n}}''+\tilde{e}_{{n}}'''.
\end{align}

\vspace*{2mm}

Next we formulate the inductive hypothesis.
Let $m\in \mathbb{N}$ with $m\geq {12}$ and $\widetilde{\alpha}:=m-{2}$.
Suppose that the initial data $(U_0^+,U_0^-,\varphi_0)$ satisfy all the conditions of Lemma \ref{lem:app}, yielding the estimates  ({\it cf.}~\eqref{app2b} and \eqref{f^a:est})
\begin{align} \label{small}
 \big\|\widetilde{U}^a\big\|_{H^{m+1}(\Omega_T)}
 +\big\|\varphi^a\big\|_{H^{m+{5/2}}(\Sigma_T)}
 \leq C(M_0),\
 \big\|f^a\big\|_{H^{m}(\Omega_T)}\leq \delta_0(T),
\end{align}
where $M_0$ is defined by \eqref{M0} and $\delta_0(T)$ tends to zero as $T\to 0$.
Suppose further that Assumptions {\bf (A-1)}--{\bf (A-2)} are satisfied.
Given an integer $ {\alpha }\in (4, \widetilde{\alpha})$ and a real number $\boldsymbol{\epsilon}>0$,
our inductive hypothesis reads
\begin{align*}
 (\mathbf{H}_{{n}-1})\ \left\{\begin{aligned}
  \textrm{(a)}\,\,  &\|(\delta V_k,\delta \varPsi_k)\|_{H^{s}(\Omega_T)}+\|(\delta\psi_k,\mathrm{D}_{x'}\delta\psi_k)\|_{H^{s}(\Sigma_T)}\leq \boldsymbol{\epsilon} \theta_k^{s-{\alpha }-1}\varDelta_k\\
  &\quad \textrm{for all integers $k\in [0,{n}-1]$  and $s\in [3,\widetilde{\alpha}]$};\\
  \textrm{(b)}\,\, &\|\mathcal{L}( V_k,  \varPsi_k)-f^a\|_{H^{s}(\Omega_T)}\leq 2 \boldsymbol{\epsilon} \theta_k^{s-{\alpha }-1}\\
  &\quad \textrm{for all integers $k\in [0,{n}-1]$  and $s\in [3,\widetilde{\alpha}-2]$};\\
  \textrm{(c)}\,\,  &\|\mathcal{B}( V_k,  \psi_k)\|_{H^{s}(\Sigma_T)}\leq  \boldsymbol{\epsilon} \theta_k^{s-{\alpha }-1}\\
  &\quad \textrm{for all integers $k\in [0,{n}-1]$  and $s\in [4,{\alpha}]$},
 \end{aligned}\right.
\end{align*}
where $\varDelta_{k}:=\theta_{k+1}-\theta_k$.
Since $\theta_0\geq 1$ and $\theta_n:= (\theta_0^2+n)^{1/2}$, we find
$\varDelta_{k} \sim \theta_k^{-1}$ for all $k\in\mathbb{N}$.

We are going to show that hypothesis ($\mathbf{H}_{{n}-1}$) implies ($\mathbf{H}_{{n}}$)
provided $T,\boldsymbol{\epsilon}>0$ are small enough and $\theta_0\geq 1$ is suitably large.
After that, we shall prove that ($\mathbf{H}_0$) holds for $T>0$ sufficiently small.

Let us first assume that hypothesis ($\mathbf{H}_{{n}-1}$) is satisfied,
which implies the following consequences as in \cite[Lemmas 6--7]{CS08MR2423311} (see also \cite[Lemma 4.6]{T09CPAMMR2560044}).

\begin{lemma} \label{lem:triangle}
 If $\theta_0$ is large enough, then
 \begin{align}
  \label{tri1}&\|( V_k, \varPsi_k)\|_{H^{s}(\Omega_T)}+\|\psi_k\|_{H^{s}(\Sigma_T)}
  \leq
  \left\{\begin{aligned}
   &\boldsymbol{\epsilon}  \theta_k^{(s-{\alpha })_+}   &&\textrm{if }s\neq {\alpha },\\
   &\boldsymbol{\epsilon}  \log \theta_k   &&\textrm{if }s= {\alpha },
  \end{aligned}\right.\\
  &\| (({I}-\mathcal{S}_{\theta_k})V_k,   ({I}-\mathcal{S}_{\theta_k})\varPsi_k)\|_{H^{s}(\Omega_T)}
  \label{tri2}  +\|({I}-\mathcal{S}_{\theta_k})\psi_k \|_{H^{s}(\Sigma_T)}
  \lesssim \boldsymbol{\epsilon}  \theta_k^{s-{\alpha }},
 \end{align}
 for all integers $k\in [0,{n}]$  and $s\in [3,\widetilde{\alpha}]$,
 and
 \begin{align}
  \label{tri3}&\|(  \mathcal{S}_{\theta_k}V_k, \mathcal{S}_{\theta_k}\varPsi_k)\|_{H^{s}(\Omega_T)}
  +\|\mathcal{S}_{\theta_k}  \psi_k\|_{H^{s}(\Sigma_T)}\lesssim
  \left\{\begin{aligned}
   &\boldsymbol{\epsilon} \theta_k^{(s-{\alpha })_+}  &&\textrm{if }s\neq {\alpha },\\
   &\boldsymbol{\epsilon}  \log \theta_k  &&\textrm{if }s= {\alpha },
  \end{aligned}\right.
 \end{align}
 for all integers $k\in [0,{n}]$  and $s\in [3,\widetilde{\alpha}+3]$.
\end{lemma}

\subsection{Estimate of the Error Terms}
In this subsection we estimate the error terms $e_{k}$ and $\tilde{e}_{k}$ defined by \eqref{e.e.tilde}.
First we rewrite the quadratic error terms $e'_{k}$ and $\tilde{e}_{k}'$ given in \eqref{decom1}--\eqref{decom2} as %
\begin{align}
\nonumber e_k' = \int_{0}^{1}
 \;& \mathbb{L}'' \big(U^a+V_{k}+\tau \delta  V_{k},
 \varPhi^a+\varPsi_{k} \\
 &+\tau \delta \varPsi_{k} \big)
 \big((\delta V_{k},\delta\varPsi_{k}),(\delta V_{k},\delta\varPsi_{k})\big)
 (1-\tau)\d\tau,
\label{e'}
 \\
\nonumber \tilde{e}_k'
 = \int_{0}^{1}
\;& \mathbb{B}'' \big(U^a+V_{k}+\tau \delta  V_{k},
 \varphi^a+\psi_{k}  \\
 & +\tau \delta \psi_{k} \big) \big((\delta V_{k},\delta \psi_{k}),(\delta V_{k},\delta \psi_{k}) \big)
 (1-\tau)\d\tau,
\label{e'.tilde}
\end{align}
through the second derivatives of the operators $\mathbb{L}$ and $\mathbb{B}$:
\begin{align*}
 \mathbb{L}''\big(\mathring{U},\mathring{\varPhi}\big)
 \big((V,\varPsi),(\widetilde{V},\widetilde{\varPsi})\big)
 &:=\left.\frac{\d}{\d \theta}
 \mathbb{L}'\big(\mathring{U}+\theta \widetilde{V},
 \mathring{\varPhi}+\theta \widetilde{\varPsi}\big)
 \big(V,\varPsi\big)\right|_{\theta=0},\\[0.5mm]
 \mathbb{B}''\big(\mathring{U},\mathring{\varphi}\big)
 \big((V,\psi),(\widetilde{V},\tilde{\psi})\big)
 &:=\left.\frac{\d}{\d \theta}
 \mathbb{B}'\big(\mathring{U}+\theta \widetilde{V},
 \mathring{\varphi}+\theta \tilde{\psi}\big) \big(V,\psi\big)\right|_{\theta=0},
\end{align*}
where $\mathbb{L}'$ and $\mathbb{B}'$ are the first-derivative operators defined by \eqref{L'bb:def}.
A lengthy but straightforward computation yields the following explicit form of $\mathbb{B}''$:
\begin{align}
 \nonumber
 &\mathbb{B}''\big(\mathring{U},\mathring{\varphi}\big)
 \big((V,\psi),(\widetilde{V},\tilde{\psi})\big)
 \\ & \quad =
 \begin{pmatrix}
  \mathfrak{s}\mathrm{D}_{x'}\cdot
  \left(
  \dfrac{\mathring{\zeta}\cdot\tilde{\zeta}}{|\mathring{N}|^3}\zeta
  -\dfrac{\tilde{\zeta}\cdot {\zeta}}{|\mathring{N}|^3}\mathring{\zeta}
  -\dfrac{\mathring{\zeta}\cdot {\zeta}}{|\mathring{N}|^3}\tilde{\zeta}
  +\dfrac{3(\mathring{\zeta}\cdot {\zeta})(\mathring{\zeta}\cdot\tilde{\zeta})}{|\mathring{N}|^5}\mathring{\zeta}
  \right)\\[1mm]
  0\\ 0 \\ 0 \\
  [H_1]\p_2\tilde{\psi}+[\widetilde{H}_1]\p_2\psi\\
  [H_1]\p_3\tilde{\psi}+[\widetilde{H}_1]\p_3\psi\\
  (\tilde{v}_2^+   \p_2+\tilde{v}_3^+  \p_3 )\psi + (v_2^+  \p_2 +v_3^+  \p_3)\tilde{\psi}
 \end{pmatrix},
 \label{B''.form}
\end{align}
where $\zeta:=\mathrm{D}_{x'}\psi$, $\mathring{\zeta}:=\mathrm{D}_{x'}\mathring{\varphi}$, and $\tilde{\zeta}:=\mathrm{D}_{x'}\tilde{\psi}$.
Omitting the detailed calculations, we apply the Moser-type calculus and embedding inequalities to derive the following result.

\begin{proposition}  \label{pro:tame2}
 Let $T>0$ and $s\in\mathbb{N}$ with $s\geq 3$.
Suppose that $(\widetilde{V},\widetilde\varPsi)\in H^{s+1}(\Omega_T)$ and $\tilde{\varphi}\in H^{s+2}(\Sigma_T)$ satisfy
 $\|(\widetilde{V},\widetilde{\varPsi} )\|_{H^{4}(\Omega_T)} +\|\tilde{\varphi}\|_{H^{3}(\Sigma_T)}\leq \widetilde{K}$
 for some constant $\widetilde{K}>0$.
%\vspace*{0.5mm}\noindent
If $(V_i,\varPsi_i)\in H^{s+1}(\Omega_T)$ for $i=1,2$, then
 \begin{multline}
\big\|\mathbb{L}''\big(\widebar{U}+\widetilde{V},\widebar{\varPhi}+\widetilde{\varPsi} \big) \big((V_1,\varPsi_1),(V_2,\varPsi_2) \big)\big\|_{H^s(\Omega_{T})} \nonumber \\[1mm]
\mbox{\ } \lesssim_{\widetilde{K}}
  \sum_{i\neq j} \big\|(V_i,\varPsi_i)\big\|_{H^{4}(\Omega_{T})}\big\|(V_j,\varPsi_j)\big\|_{H^{s+1}(\Omega_{T})}
  \\
 + \big\|(V_1,\varPsi_1)\big\|_{H^{4}(\Omega_{T})}\big\|(V_2,\varPsi_2)\big\|_{H^{4}(\Omega_{T})} \big\| ({\widetilde{V}},\widetilde{\varPsi} )\big\|_{H^{s+1}(\Omega_{T})},
 \nonumber
 \end{multline}
where $\widebar{U}:=(\widebar{U}^+,\widebar{U}^-)^{\top}$ and $\widebar{\varPhi}:=(x_1,-x_1)^{\top}.$
%
%\vspace*{0.5mm}\noindent
If  $W_i \in H^{s}(\Sigma_T) $ and $ \psi_i\in  H^{s+2}(\Sigma_T)$ for $i=1,2$, then
 \begin{multline}
 \big\|\mathbb{B}''\big(\widebar{U}+\widetilde{V},\widetilde{\varphi} \big)
  \big((W_1,\psi_1),(W_2,\psi_2) \big)\big\|_{H^{s}(\Sigma_T)} \nonumber\\[1mm]
 \lesssim_{\widetilde{K}} \sum_{i\neq j} \bigg\{
 \big\|W_i\big\|_{H^{s}(\Sigma_T)}\big\|\psi_j\big\|_{H^{3}(\Sigma_T)}
 + \big\|W_i\big\|_{H^{2}(\Sigma_T)} \big\| \psi_j\big\|_{H^{s+1}(\Sigma_T)}
\mbox{\qquad\qquad}
 \\
 + \big\|\psi_i\big\|_{H^{3}(\Sigma_T)} \big\| \psi_j\big\|_{H^{s+2}(\Sigma_T)}
  +\big\|\psi_1\big\|_{H^{3}(\Sigma_T)}\big\|\psi_2\big\|_{H^{3}(\Sigma_T)}
  \big\| \tilde{\varphi}\big\|_{H^{s+2}(\Sigma_T)}
  \bigg\}.
  \nonumber
 \end{multline}
\end{proposition}

Apply Proposition \ref{pro:tame2} to obtain the estimate for $e'_{k}$ and $\tilde{e}_{k}'$ as follows.
\begin{lemma}\label{lem:quad}
Let ${\alpha }\geq 5$.
If $\boldsymbol{\epsilon}>0$ is small enough and $\theta_0\geq 1$ is sufficiently large, then
\begin{align}\label{quad.est}
 \|e_k'\|_{H^s(\Omega_{T})}+\|\tilde{e}_k'\|_{H^{s}(\Sigma_T)}
 \lesssim \boldsymbol{\epsilon}^2 \theta_k^{\varsigma_1(s)-1}\varDelta_k,
\end{align}
 for all integers $k\in [0,{n}-1]$  and $s\in [3,\widetilde{\alpha}-2]$,
where  $$\varsigma_1(s):=\max\{s+3-2\alpha,\,(s+1-\alpha)_++6-2\alpha\}.$$
\end{lemma}
\begin{proof}
In view of \eqref{small}, hypothesis ($\mathbf{H}_{{n}-1}$), and Lemma \ref{lem:triangle}, we get
\begin{align*}
\|(\widetilde{ U}^a,V_{k},\delta  V_{k},\varPsi^a,\varPsi_{k}, \delta \varPsi_{k})\|_{H^{4}(\Omega_T)}
+\|(\varphi^a,\psi_{k}, \delta \psi_{k})\|_{H^{3}(\Sigma_T)}
\lesssim 1,
\end{align*}
which allows us to apply Proposition \ref{pro:tame2} for estimating $e'_{k}$ and $\tilde{e}_{k}'$ through the identities \eqref{e'}--\eqref{e'.tilde}.
More precisely, we use \eqref{small}, hypothesis ($\mathbf{H}_{{n}-1}$), and the trace theorem to deduce
\begin{align*}
&\|e_k'\|_{H^s(\Omega_{T})}\lesssim
\boldsymbol{\epsilon}^2\varDelta_k^2 \left(\theta_{k}^{s+3-2{\alpha }} + \theta_{k}^{6-2{\alpha }} \|(V_k,\varPsi_k)\|_{H^{s+1}(\Omega_{T})}\right),\\
&\|\tilde{e}_k'\|_{H^{s}(\Sigma_T)}\lesssim
\boldsymbol{\epsilon}^2\varDelta_k^2\left(\theta_{k}^{s+3-2{\alpha }} + \theta_{k}^{4-2{\alpha }} \|\psi_k\|_{H^{s+2}(\Sigma_{T})}\right),
\end{align*}
 for all integers $k\in [0,{n}-1]$  and $s\in [3,\widetilde{\alpha}-2]$.
Then we obtain the estimate \eqref{quad.est} by utilizing the inequalities \eqref{tri1}, $\alpha\geq 5$,  and $(s+2-\alpha)_+\leq (s+1-\alpha)_++1$.
\end{proof}

The next lemma provides the estimate of  the first substitution error terms $e_{k}''$ and $\tilde{e}_{k}''$ defined in \eqref{decom1}--\eqref{decom2}.
\begin{lemma} \label{lem:1st}
 Let ${\alpha }\geq 5$.
 If $\boldsymbol{\epsilon}>0$ is small enough and $\theta_0\geq 1$ is sufficiently large, then
 \begin{align}\label{1st.sub}
  \|e_k'' \|_{H^s(\Omega_{T})} +\|\tilde{e}_k''\|_{H^{s}(\Sigma_T)}
  \lesssim \boldsymbol{\epsilon}^2 \theta_k^{\varsigma_2(s)-1}\varDelta_k,
 \end{align}
 for all integers $k\in [0,{n}-1]$  and $s\in [3,\widetilde{\alpha}-2]$,
 where
 \begin{align}\label{varsigma2.def}
  \varsigma_2(s):=\max\{s+5-2\alpha,\,(s+1-\alpha)_++8-2\alpha  \}.
 \end{align}
\end{lemma}
\begin{proof}
We first rewrite the terms $e_k''$ and $\tilde{e}_k''$ as
\begin{align}
\nonumber {e}_k''=\int_{0}^{1}\;& \mathbb{L}''\Big(U^a+\mathcal{S}_{\theta_k}V_k+\tau({I}-\mathcal{S}_{\theta_k})V_k,\,
\varPhi^a+\mathcal{S}_{\theta_k}\varPsi_k \\
& +\tau({I}-\mathcal{S}_{\theta_k})\varPsi_k\Big)\Big(\big(\delta V_k ,\delta \varPsi_k\big),\, \big(({I}-\mathcal{S}_{\theta_k})V_k,({I}-\mathcal{S}_{\theta_k})\varPsi_k\big) \Big)\d\tau,\nonumber\\
 \nonumber \tilde{e}_k''=
 \int_{0}^{1}\;&\mathbb{B}''\Big(U^a+\mathcal{S}_{\theta_k}V_k+\tau({I}-\mathcal{S}_{\theta_k})V_k,\,
 \varphi^a+\mathcal{S}_{\theta_k}\psi_k \\
 & +\tau({I}-\mathcal{S}_{\theta_k})\psi_k\Big)
 \Big(\big(\delta V_k ,\delta \psi_k\big),\, \big(({I}-\mathcal{S}_{\theta_k})V_k,({I}-\mathcal{S}_{\theta_k})\psi_k\big) \Big)\d\tau.
 \nonumber
\end{align}
Then we utilize Proposition \ref{pro:tame2}, \eqref{small}, hypothesis ($\mathbf{H}_{{n}-1}$), Lemma \ref{lem:triangle}, and the trace theorem to derive
\begin{align*}
 &\|e_k''\|_{H^s(\Omega_{T})}\lesssim
 \boldsymbol{\epsilon}^2\varDelta_k \left(\theta_{k}^{s+4-2{\alpha }} + \theta_{k}^{7-2{\alpha }} \|(\mathcal{S}_{\theta_k}V_k,\mathcal{S}_{\theta_k}\varPsi_k)\|_{H^{s+1}(\Omega_{T})}\right),\\
 &\|\tilde{e}_k''\|_{H^{s}(\Sigma_T)}\lesssim
 \boldsymbol{\epsilon}^2\varDelta_k\left(\theta_{k}^{s+4-2{\alpha }} + \theta_{k}^{5-2{\alpha }} \|\mathcal{S}_{\theta_k}\psi_k\|_{H^{s+2}(\Sigma_{T})}\right),
\end{align*}
 for all integers $k\in [0,{n}-1]$  and $s\in [3,\widetilde{\alpha}-2]$.
The estimate \eqref{1st.sub} follows by means of the inequalities \eqref{tri3}, $\alpha\geq 5$,  and $(s+2-\alpha)_+\leq (s+1-\alpha)_++1$.
\end{proof}

For the solvability of the problem \eqref{effective.NM},
we shall require that the smooth modified state $(V_{{n}+1/2},\psi_{{n}+1/2})$ vanishes in the past
and $(U^a+V_{{n}+1/2},\varphi^a+\psi_{{n}+1/2})$ satisfies the constraints \eqref{bas1a}--\eqref{bas1d}.
Since $(V_{{n}+1/2},\psi_{{n}+1/2})$ should vanish in the past and $(U^a,\varphi^a)$ satisfies \eqref{app2},
the state $(U^a+V_{{n}+1/2},\varphi^a+\psi_{{n}+1/2})$ will satisfy \eqref{bas1a}--\eqref{bas1b} and \eqref{bas1d} provided $T>0$ is sufficiently small.
Therefore, we may focus only on the constraints \eqref{bas1c}.
\begin{proposition} \label{pro:modified}
Let ${\alpha }\geq 6$.
Then there are some functions $V_{n+1/2}$ and $\psi_{n+1/2}$ vanishing in the past,
such that if $\boldsymbol{\epsilon},T>0$ are small enough and $\theta_0\geq 1$ is suitably large,
then $(U^a+V_{n+1/2},\varphi^a+\psi_{n+1/2})$ satisfies \eqref{bas1a}--\eqref{bas1d}, and
\begin{alignat}{3}
&\psi_{n+1/2}=\mathcal{S}_{\theta_n}\psi_{{n}},\quad
 \varPsi_{n+1/2}^{\pm}:= \chi(\pm x_1)\psi_{n+1/2}, \label{MS.id1}\\
& \|\mathcal{S}_{\theta_n}\varPsi_n-\varPsi_{n+1/2}\|_{H^s(\Omega_{T})}\lesssim \boldsymbol{\epsilon} \theta_n^{s-{\alpha }}
\qquad   \textrm{for } s=3,\ldots,\widetilde{\alpha}+{{3}}, \label{MS.est1} \\
& \|\mathcal{S}_{\theta_n}V_n-V_{n+1/2}\|_{H^s(\Omega_{T})}\lesssim \boldsymbol{\epsilon} \theta_n^{s+1-{\alpha }}
\quad \textrm{for } s=3,\ldots,\widetilde{\alpha}+{{2}},  \label{MS.est2}
\end{alignat}
where $\varPsi_{n+1/2}:=(\varPsi_{n+1/2}^+,\varPsi_{n+1/2}^-)^{\top}$.
\end{proposition}
\begin{proof}
We divide the proof into four steps.

\vspace*{2mm}
\noindent {\it Step 1}.\quad
Let us define $\psi_{n+1/2}$ and $\varPsi_{n+1/2}^{\pm}$ by \eqref{MS.id1}.
When $3\leq s\leq \widetilde{\alpha}$, we obtain from the inequality \eqref{tri2} that
\begin{align*}
 &\|\mathcal{S}_{\theta_n}\varPsi_n-\varPsi_{n+1/2}\|_{H^s(\Omega_{T})}\\
&\quad \lesssim \|(\mathcal{S}_{\theta_n}-I)\varPsi_n\|_{H^s(\Omega_{T})}
+   \|\chi(\pm x_1)(I-\mathcal{S}_{\theta_n})\psi_n\|_{H^s(\Omega_{T})}\\
&\quad \lesssim \|(I-\mathcal{S}_{\theta_n})\varPsi_n\|_{H^s(\Omega_{T})} +  \|(I-\mathcal{S}_{\theta_n})\psi_n\|_{H^s(\Sigma_{T})}
\lesssim \boldsymbol{\epsilon} \theta_n^{s-{\alpha }}.
\end{align*}
Then the estimate \eqref{MS.est1} follows by using \eqref{tri3} for $\widetilde{\alpha}< s\leq \widetilde{\alpha}+3$.
%\begin{align*}
% \|\mathcal{S}_{\theta_n}\varPsi_n-\varPsi_{n+1/2}\|_{H^s(\Omega_{T})}   \lesssim \|\mathcal{S}_{\theta_n}\varPsi_n\|_{H^s(\Omega_{T})}  +\|\mathcal{S}_{\theta_n}\psi_{n}\|_{H^s(\Omega_{T})}  \lesssim \boldsymbol{\epsilon} \theta_n^{s-{\alpha }}.
%\end{align*}

\vspace*{2mm}
\noindent {\it Step 2}.\quad
For $i=2,3$, we define
\begin{align*}
v_{i,n+1/2}^{\pm}:=(\mathcal{S}_{\theta_n}v_{i,n})^{\pm} \mp \frac{1}{2}
[\mathcal{S}_{\theta_n}v_{n}]\big|_{x_1=0}\chi(x_1).
\end{align*}
Using \eqref{smooth.p2} gives
\iffalse
\begin{align}
  \| [\mathcal{S}_{\theta_n}v_{n}]\|_{H^s(\Sigma_{T})}
\lesssim
\left\{
\begin{aligned}
&\theta_n^{s+1-\alpha}\| [v_{n}]\|_{H^{\alpha}(\Sigma_{T})}  &&\textrm{ if }
\alpha\leq s \leq \widetilde{\alpha}+3,\\
&\| [v_{n}]\|_{H^{s+1}(\Sigma_{T})}  &&\textrm{ if } 3\leq s< \alpha.
\end{aligned}
\right.
\label{MS:p0}
\end{align}
\fi
\begin{align}
	\| [\mathcal{S}_{\theta_n}v_{n}]\|_{H^s(\Sigma_{T})}
	\lesssim
\theta_n^{s-3} \| [v_{n}]\|_{H^{4}(\Sigma_{T})}   \quad \textrm{ if } 3\leq s\leq \widetilde{\alpha}+3.
	\label{MS:p0}
\end{align}
It follows from hypothesis ($\mathbf{H}_{n-1}$) that  for all integers $s\in [4,{\alpha}]$,
\begin{align*}
\| [v_{n}]\|_{H^{s}(\Sigma_{T})}
&\lesssim \| [v_{n-1}]\|_{H^{s}(\Sigma_{T})}
+\| \delta v_{n-1}\|_{H^{s}(\Sigma_{T})} \\
&\lesssim \| \mathcal{B}(V_{n-1},\psi_{n-1})\|_{H^{s}(\Sigma_{T})}
+\| \delta V_{n-1}\|_{H^{s+1}(\Omega_{T})}
 \lesssim \boldsymbol{\epsilon} \theta_n^{s-{\alpha }-1}.
\end{align*}
Here we recall the definition of the boundary operator $\mathcal{B}$ from \eqref{P.new} and \eqref{MCD0b}.
Plugging the last inequality to \eqref{MS:p0} implies
\begin{align}
\label{MS:p1}
  \| [\mathcal{S}_{\theta_n}v_{n}]\|_{H^s(\Sigma_{T})}
 \lesssim \boldsymbol{\epsilon} \theta_n^{s-{\alpha }}
\quad \textrm{for } s=3,\ldots, \widetilde{\alpha}+3.
\end{align}
Hence we infer
\begin{align}
\label{MS:p2}
 \|v_{i,n+1/2}-\mathcal{S}_{\theta_n}v_{i,n}\|_{H^s(\Omega_{T})}
 \lesssim \boldsymbol{\epsilon} \theta_n^{s-{\alpha }}
\quad \textrm{for }i=2,3  \textrm{ and } s=3,\ldots, \widetilde{\alpha}+3.
\end{align}

\vspace*{2mm}
\noindent {\it Step 3}.\quad
Let us set
\begin{gather*}
 v_{1,n+1/2}^{\pm}:=
 (\mathcal{S}_{\theta_{{n}}}v_{1,n} )^{\pm}+
 \chi(x_1)\left(\hat{w}_{n}-   (\mathcal{S}_{\theta_{{n}}}v_{1,n} )^{\pm}|_{x_1=0} \right),
\end{gather*}
where $\hat{w}_{n}$ is defined by
\begin{gather*}
 \hat{w}_{n}:=\p_t \psi_{n+1/2} +
 \sum_{i=2,3}  \Big( \big(v_i^{a+}+v_{i,n+1/2}^+\big) \p_i \psi_{n+1/2} +v_{i,n+1/2}^+
 \p_i \varphi^a \Big)\Big|_{x_1=0}.
\end{gather*}
It follows from \eqref{app1a} that $(v^a+v_{n+1/2},\varphi^a+\psi_{n+1/2})$ satisfies the first and third constraints in \eqref{bas1c}.
By virtue of \eqref{app1a} and \eqref{MS.id1}, we have
\begin{align*}
&\hat{w}_{n}-   (\mathcal{S}_{\theta_{{n}}}v_{1,n} )^{+}  |_{x_1=0}\\
&= \underbrace{\mathcal{B}(\mathcal{S}_{\theta_{{n}}}V_n, \mathcal{S}_{\theta_{{n}}}\psi_n )_7 |_{x_1=0}}_{\mathcal{T}_1}+
\sum_{i=2,3}\underbrace{\p_i(\varphi^a+\psi_{n+1/2})\big(v_{i,n+1/2}^+-(\mathcal{S}_{\theta_{{n}}}v_{i,n} )^{+}\big) |_{x_1=0}}_{\mathcal{T}_{2i}}.
\end{align*}
Utilizing the Moser-type calculus inequality \eqref{Moser2}, the trace theorem, \eqref{MS:p2}, and \eqref{tri3} yields
\begin{align*}
\|\mathcal{T}_{2i}\|_{H^s(\Sigma_{{T}})}\lesssim
\boldsymbol{\epsilon} \theta_n^{s+1-{\alpha }}
\quad \textrm{for }i=2,3  \textrm{ and } s=3,\ldots, \widetilde{\alpha}+2.
\end{align*}
To estimate $\mathcal{T}_1$ in $H^s(\Sigma_{{T}})$, we decompose
\begin{align*}
\mathcal{T}_1 =
\;&\mathcal{T}_{1a}+\underbrace{\mathcal{S}_{\theta_n}\big( \mathcal{B}\left(V_n ,\psi_n\right)_7  |_{x_1=0}
 -\mathcal{B}\left(V_{n-1},\psi_{n-1}\right)_7  |_{x_1=0}\big)}_{\mathcal{T}_{1b}}\\
&\!+\underbrace{\mathcal{B}\big( \mathcal{S}_{\theta_{{n}}}V_{n}  ,\mathcal{S}_{\theta_{{n}}}\psi_n\big) _7|_{x_1=0}
 -\mathcal{S}_{\theta_n}\big(\mathcal{B}\left(V_n ,\psi_n\right)_7  |_{x_1=0}\big)}_{\mathcal{T}_{1c}},
\end{align*}
where $\mathcal{T}_{1a}:=\mathcal{S}_{\theta_n}\big(\mathcal{B}\big(V_{n-1} ,\psi_{n-1}\big)_7 |_{x_1=0}\big)$.
It follows from \eqref{smooth.p1a} and point (c) of hypothesis $(\mathbf{H}_{n-1})$ that
\begin{align*}
\|\mathcal{T}_{1a}\|_{H^s(\Sigma_{{T}})}
\lesssim \theta_n^{s-3}\|\mathcal{B}\big(V_{n-1} ,\psi_{n-1}\big)\|_{H^4(\Sigma_{{T}})}
\lesssim \boldsymbol{\epsilon} \theta_n^{s-{\alpha }}
\quad \textrm{for } s=3,\ldots, \widetilde{\alpha}+2.
\end{align*}
In view of the identity
\begin{align*}
\mathcal{T}_{1b}=
\;&\mathcal{S}_{\theta_{{n}}} \big( \p_t \delta\psi_{n-1} \big)
-\mathcal{S}_{\theta_{{n}}} \big(\delta v_{1,n-1}^+ |_{x_1=0} \big)
\\
&\! +\sum_{i=2,3}\mathcal{S}_{\theta_{{n}}} \Big(
\big(v_{i}^{a+}+v_{i,n}^+   \big) |_{x_1=0}\p_i \delta\psi_{n-1}
+\delta v_{i,n-1}^+  |_{x_1=0}\p_i (\varphi^a +\psi_{n-1}  )\Big),
\end{align*}
we use Proposition \ref{pro:smooth}, hypothesis $(\mathbf{H}_{n-1})$, the trace and embedding theorems, and the Moser-type calculus inequality \eqref{Moser2} to deduce
\begin{align*}
 \|\mathcal{T}_{1b}\|_{H^s(\Sigma_{{T}})}
 \lesssim \boldsymbol{\epsilon} \theta_n^{s-{\alpha }}
 \quad \textrm{for } s=3,\ldots, \widetilde{\alpha}+2.
\end{align*}
For estimating the term $\mathcal{T}_{1c}$, we shall utilize the further decomposition
\begin{align}
\nonumber
 \mathcal{T}_{1c}=\,&
 \Big\{ \p_t(\mathcal{S}_{\theta_n}\psi_n)-\mathcal{S}_{\theta_n}\p_t\psi_n\Big\}
 -\Big\{  (\mathcal{S}_{\theta_n} v_{1,n})^+|_{x_1=0}
 -\mathcal{S}_{\theta_n} \big(v_{1,n}^+|_{x_1=0}\big)\Big\}\\[1.5mm]
\nonumber
 &\!+\sum_{i=2,3} \Big\{ \big(v_{i}^{a+}+(\mathcal{S}_{\theta_n} v_{i,n})^+  \big)|_{x_1=0}
 \p_i \mathcal{S}_{\theta_n} \psi_n
 -\mathcal{S}_{\theta_n}  \big( ( v_{i}^{a+} +v_{i,n}^+ )|_{x_1=0}\p_i  \psi_n\big)\Big\}\\
 &\!+\sum_{i=2,3} \Big\{ (\mathcal{S}_{\theta_n} v_{i,n})^+|_{x_1=0}\p_i   \varphi^a
 -\mathcal{S}_{\theta_n}  \big(  v_{i,n}^+  |_{x_1=0}\p_i  \varphi^a\big) \Big\}.
 \label{T.1c}
\end{align}
Let us make the estimate of the commutator
\begin{align*}
\mathcal{T}_{3}:= \underbrace{\big(v_{3}^{a+}+(\mathcal{S}_{\theta_n} v_{3,n})^+  \big)  |_{x_1=0}
\p_{3} \mathcal{S}_{\theta_n} \psi_n}_{\mathcal{T}_{3a}}
\underbrace{-\mathcal{S}_{\theta_n}  \big( ( v_{3}^{a+} +v_{3,n}^+ ) |_{x_1=0}\p_{3}  \psi_n\big)}_{\mathcal{T}_{3b}}.
\end{align*}
For $\alpha+1\leq s\leq \widetilde{\alpha}+2$,
recalling from \eqref{small} that $\widetilde{U}^a\in H^{\widetilde{\alpha}+3}(\Omega_{T})$,
we utilize the Moser-type calculus inequality \eqref{Moser2}, the trace and embedding theorems, Proposition \ref{pro:smooth}, and Lemma \ref{lem:triangle} to derive
\begin{align*}
\|\mathcal{T}_{3a}\|_{H^s(\Sigma_{{T}})}
\lesssim
\,&\boldsymbol{\epsilon}\|\tilde{v}_{3}^{a}+\mathcal{S}_{\theta_n} v_{3,n}\|_{H^{s+1}(\Omega_{{T}})}
+\|\mathcal{S}_{\theta_n} \psi_n\|_{H^{s+1}(\Sigma_{{T}})}
\lesssim \boldsymbol{\epsilon} \theta_n^{s-{\alpha }+1},
\\
\|\mathcal{T}_{3b}\|_{H^s(\Sigma_{{T}})}\lesssim
\,& \theta_n^{s-{\alpha }}\| ( v_{3}^{a+} +v_{3,n}^+ ) |_{x_1=0}\p_{3}  \psi_n\|_{H^{\alpha}(\Sigma_{{T}})}\\
\lesssim
\,&\theta_n^{s-{\alpha }}\big(
\boldsymbol{\epsilon}\|  \tilde{v}^{a} +v_{n} \|_{H^{\alpha+1}(\Omega_{{T}})}
+\|  \psi_n\|_{H^{\alpha+1}(\Sigma_{{T}})}
\big)
\lesssim \boldsymbol{\epsilon} \theta_n^{s-{\alpha }+1}.
\end{align*}
For $3\leq s \leq \alpha$,
thanks to the triangle inequality
\begin{align*}
\|\mathcal{T}_{3}\|_{H^s(\Sigma_{{T}})}
\leq
\,&\big\|\big( (\mathcal{S}_{\theta_n} v_{3,n})^+ -v_{3,n}^+ \big)  |_{x_1=0}
\p_{3} \mathcal{S}_{\theta_n} \psi_n \big\|_{H^s(\Sigma_{{T}})} \\
&\!+\big\|( v_{3}^{a+} +v_{3,n}^+ ) |_{x_1=0}\p_{3}   (\mathcal{S}_{\theta_n}-I) \psi_n\big\|_{H^s(\Sigma_{{T}})} \\
&\!+\big\|(I-\mathcal{S}_{\theta_n}) \big(( v_{3}^{a+} +v_{3,n}^+ ) |_{x_1=0}\p_3\psi_n\big)\big\|_{H^s(\Sigma_{{T}})},
\end{align*}
we can employ the Moser-type calculus inequality \eqref{Moser2} to infer
\begin{align*}
 \|\mathcal{T}_{3}\|_{H^s(\Sigma_{{T}})}
 \lesssim \boldsymbol{\epsilon} \theta_n^{s-{\alpha }+1}
\quad \textrm{for } s=3,\ldots, \alpha.
\end{align*}
The other commutators in the decomposition \eqref{T.1c} can be handled by following the same approach, so we can omit the details and write down the estimate
\begin{align*}
 \|\mathcal{T}_{1c}\|_{H^s(\Sigma_{{T}})}
 \lesssim \boldsymbol{\epsilon} \theta_n^{s-{\alpha }+1}
 \quad \textrm{for } s=3,\ldots, \widetilde{\alpha}+2.
\end{align*}
Combining the above estimates of $\mathcal{T}_{1a}$, $\mathcal{T}_{1b}$, and $\mathcal{T}_{1c}$ with \eqref{MS:p1} gives
\begin{align*}
 \|v_{1,n+1/2}-\mathcal{S}_{\theta_n}v_{1,n}\|_{H^s(\Omega_{{T}})}
 &\lesssim
  \|\hat{w}_{n}-   (\mathcal{S}_{\theta_{{n}}}v_{1,n} )^{+}   \|_{H^s(\Sigma_{{T}})}
+  \| [\mathcal{S}_{\theta_n}v_{n}]\|_{H^s(\Sigma_{T})}
 \\
 &\lesssim \boldsymbol{\epsilon} \theta_n^{s-{\alpha }+1}
 \qquad \textrm{for } s=3,\ldots, \widetilde{\alpha}+2.
\end{align*}

\vspace*{2mm}
\noindent {\it Step 4}.\quad
We define
\begin{align*}
H_{n+1/2}^{\pm}:=(\mathcal{S}_{\theta_n}H_{n})^{\pm} \mp \frac{1}{2}
 [\mathcal{S}_{\theta_n}H_{n}] \big|_{x_1=0}  \chi(x_1),
\end{align*}
so that $H^a+H_{n+1/2}$ satisfies the second constraint in \eqref{bas1c}, {\it i.e.},
$[H^a+H_{n+1/2}]=[H_{n+1/2}]=0$ on $\Sigma_{{T}}$.
In light of \eqref{smooth.p2}, we obtain
\begin{align}
\label{MS:p3}
 \| [\mathcal{S}_{\theta_n}H_{n}]\|_{H^s(\Sigma_{T})}
 \lesssim
  \theta_n^{s-3} \| [H_{n}]\|_{H^{4}(\Sigma_{T})}  \quad \textrm{ if } 3\leq s \leq \widetilde{\alpha}+2.
\end{align}
As in the proof of (169) in \cite[Lemma 6]{MTT18MR3766987}, when $\alpha\geq 6$, we can prove that
if $\boldsymbol{\epsilon},T>0$ are small enough and $\theta_0\geq 1$ is sufficiently large, then
\begin{align*}
\big\| \big(1,-\p_2(\varphi^a+\psi_{n-1}), -\p_3(\varphi^a+\psi_{n-1}) \big)[H_{n-1}]\big\|_{H^{s}(\Sigma_{T})}
%\lesssim \boldsymbol{\epsilon}\theta_{n-1}^{s-\alpha}
\lesssim \boldsymbol{\epsilon}\theta_{n}^{s-\alpha}
\end{align*}
for $s=3,\ldots,\alpha-1$.
%Recalling the fifth and sixth components of $\mathcal{B}(V_{n-1},\psi_{n-1})$,
From point (c) of hypothesis ($\mathbf{H}_{n-1}$), we derive
\begin{align*}
\left\|
\begin{pmatrix}
\p_2(\varphi^a+\psi_{n-1}) &1 & 0\\
\p_3(\varphi^a+\psi_{n-1})  &0 &1
\end{pmatrix}
[H_{n-1}]
 \right\|_{H^{s}(\Sigma_{T})}
\lesssim \boldsymbol{\epsilon}\theta_{n}^{s-\alpha-1}
\quad \textrm{for }s=4,\ldots, \alpha.
\end{align*}
Combine the last two estimates with point (a) of hypothesis ($\mathbf{H}_{n-1}$) to get
\begin{align}
\label{MS:p4}
\big\| [H_{n}]\big\|_{H^{s}(\Sigma_{T})}
\lesssim  \big\| [H_{n-1}]\big\|_{H^{s}(\Sigma_{T})}
+\big\| \delta H_{n-1}\big\|_{H^{s}(\Sigma_{T})}
\lesssim \boldsymbol{\epsilon}\theta_{n}^{s-\alpha}
\end{align}
for $s=4,\ldots, \alpha-1.$
Then it follows from \eqref{MS:p3}--\eqref{MS:p4} that
\begin{align*}
 \|H_{n+1/2}-\mathcal{S}_{\theta_n}H_{n}\|_{H^s(\Omega_{{T}})}
\lesssim \| [\mathcal{S}_{\theta_n}H_{n}]\|_{H^s(\Sigma_{T})}
 \lesssim \boldsymbol{\epsilon} \theta_n^{s-{\alpha }+1}
\end{align*}
for $s=3,\ldots, \widetilde{\alpha}+2.$
Setting $p_{n+1/2}:=\mathcal{S}_{\theta_n} p_n$ (pressure) and $S_{n+1/2}:=\mathcal{S}_{\theta_n} S_n$ (entropy) completes the proof of the proposition.
\end{proof}

With Proposition \ref{pro:modified} in hand, we can obtain the following estimate of  the second substitution error terms $e_{k}'''$ and $\tilde{e}_{k}'''$ defined in \eqref{decom1}--\eqref{decom2}.

\begin{lemma} \label{lem:2nd}
 Let ${\alpha }\geq 6$.
 If $\boldsymbol{\epsilon},T>0$ are small enough and $\theta_0\geq 1$ is sufficiently large, then
 \begin{align}\label{2st.sub}
  \|  \tilde{e}_k'''\|_{H^s(\Sigma_{T})}\lesssim \boldsymbol{\epsilon}^2 \theta_k^{\varsigma_2(s)-1}\varDelta_k,
\quad  \| e_k'''\|_{H^s(\Omega_{T})}
  \lesssim \boldsymbol{\epsilon}^2 \theta_k^{\varsigma_3(s)-1}\varDelta_k,
 \end{align}
for all integers $k\in [0,{n}-1]$  and $s\in [3,\widetilde{\alpha}-2]$,
where
$\varsigma_2(s)$ is defined by \eqref{varsigma2.def} and
$$\varsigma_3(s):=\max\{s+6-2\alpha,\,(s+1-\alpha)_++9-2\alpha\}.$$
\end{lemma}
\begin{proof}
In view of \eqref{B''.form} and \eqref{MS.id1}, we can rewrite the term $\tilde{e}_k'''$ as
\begin{align*}
\tilde{e}_k'''
%=\,&\int_{0}^{1} \mathbb{B}''\Big(U^a+\tau(\mathcal{S}_{\theta_k} V_k-V_{k+1/2})+V_{k+1/2},\, \varphi^a   \\       & \qquad +\psi_{k+1/2} \Big)\Big((\delta V_k ,\delta \psi_k),\, (\mathcal{S}_{\theta_k} V_k-V_{k+1/2},0)\Big) \d\tau\\
=\,&
\begin{pmatrix}
0\\[1mm]
[\mathcal{S}_{\theta_k}H_{1,k}-H_{1,k+1/2} ]\p_2\delta \psi_k\\[1mm]
[\mathcal{S}_{\theta_k}H_{1,k}-H_{1,k+1/2} ]\p_3\delta \psi_k\\[1mm]
\sum_{i=2,3} \big((\mathcal{S}_{\theta_k}v_{i,k})^+-v_{i,k+1/2}^+\big)\p_i\delta\psi_k
\end{pmatrix}.
\end{align*}
Then we utilize the Moser-type calculus inequality \eqref{Moser2}, the embedding and trace theorems, the estimate \eqref{MS.est2}, and point (a) of hypothesis ($\mathbf{H}_{n-1}$) to discover
\begin{align*}
\|  \tilde{e}_k'''\|_{H^s(\Sigma_{T})}
\lesssim\,& \| \mathcal{S}_{\theta_k} V_k-V_{k+1/2}\|_{H^{s+1}(\Omega_{T})}
\|\delta \psi_k\|_{H^3(\Sigma_{{T}})}\\
&\! +
 \| \mathcal{S}_{\theta_k} V_k-V_{k+1/2}\|_{H^{3}(\Omega_{T})}
\|\delta \psi_k\|_{H^{s+1}(\Sigma_{{T}})}
\lesssim \boldsymbol{\epsilon}^2 \theta_k^{s+4-2\alpha}\varDelta_k
\end{align*}
for $s=3,\ldots,\tilde{\alpha}-2$.
%Similar to the proof of Lemmas \ref{lem:quad}--\ref{lem:1st},
Applying  Propositions \ref{pro:tame2} and \ref{pro:modified} to the identity
\begin{align}
 \nonumber {e}_k'''=\int_{0}^{1}\,&
 \mathbb{L}''\Big(U^a+\tau(\mathcal{S}_{\theta_k} V_k-V_{k+1/2})+V_{k+1/2},\,
 \varPhi^a+\tau(\mathcal{S}_{\theta_k}\varPsi_k-\varPsi_{k+1/2})
 \\
 & +\varPsi_{k+1/2} \Big)
 \Big((\delta V_k ,\delta \varPsi_k),\, (\mathcal{S}_{\theta_k} V_k-V_{k+1/2},\mathcal{S}_{\theta_k}\varPsi_k-\varPsi_{k+1/2})\Big) \d\tau,
 \nonumber
\end{align}
we have
\begin{align*}
\| e_k'''\|_{H^s(\Omega_{T})}
\lesssim   \boldsymbol{\epsilon}^2 \varDelta_k
\big(\theta_k^{s+5-2\alpha} + \theta_k^{8-2\alpha}\| ( \mathcal{S}_{\theta_k} V_k,\mathcal{S}_{\theta_k}\varPsi_k)\|_{H^{s+1}(\Omega_{T})} \big),
\end{align*}
which combined with \eqref{tri3} implies the second estimate in \eqref{2st.sub} for $e_k'''$.
\end{proof}

The next lemma concerns the last error term $e_{{n}}^{(4)}=(e_{{n}}^{(4)+},e_{{n}}^{(4)-})^{\top}$
with $e_{{n}}^{(4)\pm}$ defined by \eqref{error.D}. Here we omit the detailed proof, which is similar to that for \cite[Lemma 4.12]{T09CPAMMR2560044} (see also \cite[Lemma 4.10]{TW21MR4201624}).
\begin{lemma}  \label{lem:last}
 Let ${\alpha }\geq 6$ and $\widetilde{\alpha}\geq {\alpha }+2$.
 If $\boldsymbol{\epsilon},T>0$ are small enough and $\theta_0\geq 1$ is sufficiently large, then
 \begin{align}
 \nonumber  %	\label{last.e0}
  \|e_{{n}}^{(4)}\|_{H^s(\Omega_{T})}\lesssim \boldsymbol{\epsilon}^2 \theta_k^{\varsigma_4 (s)-1}\varDelta_k,
 \end{align}
for all integers $k\in [0,{n}-1]$  and $s\in [3,\widetilde{\alpha}-2]$,
 where
 \begin{align} \label{varsigma4.def}
  \varsigma_4(s):=\max\{s+6-2\alpha,\,(s-\alpha)_++10-2\alpha\}.
 \end{align}
\end{lemma}

As a direct corollary to Lemmas \ref{lem:quad}--\ref{lem:last},
we have the estimate for $e_k$ and $\tilde{e}_k$ defined by \eqref{e.e.tilde} as follows.
\begin{corollary}  \label{cor:sum1}
 Let ${\alpha }\geq 6$ and $\widetilde{\alpha}\geq {\alpha }+2$.
 If $\boldsymbol{\epsilon}, T>0$ are sufficiently small and $\theta_0\geq 1$ is suitably large, then
 \begin{align} \label{es.sum1a}
  \|e_k\|_{H^s(\Omega_{T})}
  \lesssim \boldsymbol{\epsilon}^2 \theta_k^{\varsigma_4(s)-1}\varDelta_k,
\quad  \|\tilde{e}_k\|_{H^{s}(\Sigma_T)}
  \lesssim \boldsymbol{\epsilon}^2 \theta_k^{\varsigma_2(s)-1}\varDelta_k,
 \end{align}
for all integers $k\in [0,{n}-1]$  and $s\in [3,\widetilde{\alpha}-2]$,
 where $\varsigma_2(s)$ and $\varsigma_4(s)$ are defined by \eqref{varsigma2.def} and \eqref{varsigma4.def}, respectively.
\end{corollary}

Similar to \cite[Lemma 4.12]{TW21MR4201624},
we can use \eqref{es.sum1a} to derive the following estimate for the accumulated error terms $E_n$ and $\widetilde{E}_n$ defined by \eqref{E.E.tilde}.

\begin{lemma} \label{lem:sum2}
 Let ${\alpha }\geq 7$ and $\widetilde{\alpha}={\alpha }+3$.
 If $\boldsymbol{\epsilon},T>0$ are small enough and $\theta_0\geq 1$ is sufficiently large, then
 \begin{align} \nonumber  % \label{es.sum2}
  \|E_{{n}}\|_{H^{{\alpha }+1}(\Omega_{T})}\lesssim \boldsymbol{\epsilon}^2 \theta_{{n}},\quad
  \|\widetilde{E}_{{n}}\|_{H^{{\alpha }+1}(\Sigma_T)}\lesssim \boldsymbol{\epsilon}^2.
 \end{align}
\end{lemma}
\begin{proof}
If ${\alpha }\geq 6$, then $\varsigma_4({\alpha }+1)-1 \leq 0$,
which together with \eqref{es.sum1a} leads to
\begin{align*}
 &\|E_{{n}}\|_{H^{{\alpha }+1}(\Omega_{T})}
 \lesssim \sum_{k=0}^{{n}-1}\|e_{k} \|_{H^{{\alpha }+1}(\Omega_{T})}
 \lesssim \sum_{k=0}^{{n}-1} \boldsymbol{\epsilon}^2 \varDelta_k
 \lesssim \boldsymbol{\epsilon}^2\theta_{{n}},
\end{align*}
provided ${\alpha }+1\leq \widetilde{\alpha}-2$.
Since $\varsigma_2({\alpha }+1)-1\leq -2$ for ${\alpha }\geq 7$,
we get from \eqref{es.sum1a} and ${\alpha }+1\leq \widetilde{\alpha}-2$ that
\begin{align*}
 \|\widetilde{E}_{{n}}\|_{H^{{\alpha }+1}(\Sigma_T)}
 \lesssim \sum_{k=0}^{{n}-1} \|\tilde{e}_{k} \|_{H^{{\alpha }+1}(\Sigma_T)}
 \lesssim \sum_{k=0}^{{n}-1} \boldsymbol{\epsilon}^2 \theta_{{k}}^{-3}
 \lesssim \boldsymbol{\epsilon}^2.
\end{align*}
The minimal possible $\widetilde{\alpha}$ is ${\alpha }+3$.
The proof is thus complete.
\end{proof}

\subsection{Proof of Theorem \ref{thm:main}}

Similar to \cite[Lemma 4.13]{TW21MR4201624}, we can obtain the following result for the source terms $f_{{n}}$ and $g_{{n}}$ computed from \eqref{source}.

\begin{lemma}\  \label{lem:source}
 Let ${\alpha }\geq 7$ and $\widetilde{\alpha}={\alpha }+3$.
 If $\boldsymbol{\epsilon},T>0$ are small enough and $\theta_0\geq 1$ is sufficiently large, then
 for all integers $s\in [3,\widetilde{\alpha}]$,
 \begin{align}
\nonumber  % \label{es.fl}
  &\|f_{{n}}\|_{H^{\alpha}(\Omega_{T})}
  \lesssim \varDelta_{{n}}\left(\theta_{{n}}^{s-{\alpha }-1} \|f^a\|_{H^{s}(\Omega_{T})}
  +\boldsymbol{\epsilon}^2 \theta_{{n}}^{s-{\alpha }-1} +\boldsymbol{\epsilon}^2\theta_{{n}}^{\varsigma_4(s)-1}\right),\\
\nonumber  %  \label{es.gl}
  &\|g_{{n}}\|_{H^{s+1}(\Sigma_T)}
  \lesssim  \boldsymbol{\epsilon}^2 \varDelta_{{n}}\left(\theta_{{n}}^{s-{\alpha }-1}+\theta_{{n}}^{\varsigma_2(s+1)-1}\right),
 \end{align}
 where $\varsigma_2(s)$ and $\varsigma_4(s)$ are defined by \eqref{varsigma2.def} and \eqref{varsigma4.def}, respectively.
\end{lemma}

The next lemma follows by applying the tame estimate \eqref{tame:es} to the problem \eqref{effective.NM}
and using Proposition \ref{pro:modified}.
We omit the proof for brevity, since it is similar to the proof of \cite[Lemma 4.17]{T09CPAMMR2560044}.

\begin{lemma}\  \label{lem:Hl1}
 Let ${\alpha }\geq 7$ and $\widetilde{\alpha}={\alpha }+3$.
 If $\boldsymbol{\epsilon}, T>0$ and $\frac{1}{\boldsymbol{\epsilon}}\|f^a\|_{H^{\alpha }(\Omega_{{T}})}$ are small enough, and $\theta_0\geq1$ is sufficiently large, then for all integers $s\in [3,\widetilde{\alpha}]$,
 \begin{align} \nonumber %\label{Hl.a}
  \|(\delta V_{{n}},\delta\varPsi_{{n}})\|_{H^{s}(\Omega_{{T}})}
  +\|(\delta\psi_{{n}},\mathrm{D}_{x'}\delta\psi_{{n}})\|_{H^{s}(\Sigma_T)}
  \leq \boldsymbol{\epsilon} \theta_{{n}}^{s-{\alpha }-1}\varDelta_{{n}}.
 \end{align}
\end{lemma}

Lemma \ref{lem:Hl1} provides point (a) in hypothesis $(\mathbf{H}_{{n}})$.
The other points in $(\mathbf{H}_{{n}})$ are given in the next lemma,
whose proof can be found in \cite[Lemma 4.19]{T09CPAMMR2560044}.

\begin{lemma}\ \label{lem:Hl2}
 Let ${\alpha }\geq 7$ and $\widetilde{\alpha}={\alpha }+3$.
 If $\boldsymbol{\epsilon}, T>0$ and $\frac{1}{\boldsymbol{\epsilon}}\|f^a\|_{H^{\alpha }(\Omega_{{T}})}$ are small enough, and $\theta_0\geq1$ is sufficiently large, then
 \begin{alignat}{3}\label{Hl.b}
  & \|\mathcal{L}( V_{{n}},  \varPsi_{{n}})-f^a\|_{H^{s}(\Omega_{{T}})}
  \leq 2 \boldsymbol{\epsilon} \theta_{{n}}^{s-{\alpha }-1}
  \quad &&\textrm{for } s=3,\ldots,\widetilde{\alpha}-2,\\
  \label{Hl.c}
  & \|\mathcal{B}( V_{{n}} ,  \psi_{{n}})\|_{H^{s}(\Sigma_T)}\leq  \boldsymbol{\epsilon} \theta_{{n}}^{s-{\alpha }-1}
  \quad &&\textrm{for } s=4,\ldots,{\alpha }.
 \end{alignat}
\end{lemma}

From Lemmas \ref{lem:Hl1}--\ref{lem:Hl2}, we have obtained hypothesis $(\mathbf{H}_{{n}})$ from $(\mathbf{H}_{{n}-1})$,
provided that ${\alpha }\geq 7$ and $\widetilde{\alpha}={\alpha }+3$ hold,
$\boldsymbol{\epsilon},T>0$ and $\frac{1}{\boldsymbol{\epsilon}}\|f^a\|_{H^{\alpha }(\Omega_{{T}})}$ are small enough, and $\theta_0 \geq 1$ is sufficiently large.
Fixing the constants ${\alpha } \geq 7$, $\widetilde{\alpha}=\alpha+3 $,
$\boldsymbol{\epsilon}>0$, and $\theta_0\geq1$,
we can prove hypothesis $(\mathbf{H}_{0})$ as in \cite[Lemma 4.20]{T09CPAMMR2560044}.

\begin{lemma}\ \label{lem:H0}
 If time $T>0$ is small enough, then hypothesis $(\mathbf{H}_0)$ holds.
\end{lemma}

We are ready to conclude the proof of Theorem \ref{thm:main}.

\vspace*{3mm}
\noindent  {\bf Proof of Theorem {\rm\ref{thm:main}}.}\quad
Let the initial data $(U_0^+,U_0^-,\varphi_0)$ satisfy all the assumptions of Theorem {\rm\ref{thm:main}}.
Let $\widetilde{\alpha}=m-{2}$ and ${\alpha }=\widetilde{\alpha}-3\geq 7$.
Then the initial data $(U_0^+,U_0^-,\varphi_0)$ are compatible up to order $m=\widetilde{\alpha}+{2}$.
In view of \eqref{app2b} and \eqref{f^a:est}, we obtain \eqref{small} and all the requirements of Lemmas \ref{lem:Hl1}--\ref{lem:H0}, provided $ \boldsymbol{\epsilon},T>0$ {are} sufficiently small and $\theta_0\geq 1$ is large enough.
Hence, for suitably short time $T$, hypothesis $(\mathbf{H}_{{n}})$ holds for all ${n}\in\mathbb{N}$.
In particular,
\begin{align*}
	\sum_{k=0}^{\infty}\left(\|(\delta V_k,\delta \varPsi_k)\|_{H^{s}(\Omega_{{T}})}+\|(\delta\psi_k,\mathrm{D}_{x'}\delta\psi_k)\|_{H^{s}(\Sigma_T)} \right)
	\lesssim \sum_{k=0}^{\infty}\theta_k^{s-{\alpha }-2} <\infty
\end{align*}
for all integers $s\in [3,{\alpha}-1]$,
Hence the sequence $(V_{k},\psi_k)$ converges to some limit $(V,\psi)$ in $H^{{\alpha }-1}(\Omega_T)\times H^{{\alpha }-1}({\Sigma_T})$.
Passing to the limit in \eqref{Hl.b}--\eqref{Hl.c} for $s={\alpha }-1=m-{6}$, we obtain \eqref{P.new}.
Therefore, $(U^+,U^-, \varphi)=(U^{a+}+V^+,U^{a-}+V^-, \varphi^a+\psi)$ is a solution of the original problem \eqref{MCD0} on the time interval $[0,T]$.
The uniqueness of solutions to the problem \eqref{MCD0} can be obtained through a standard argument; see, for instance, \cite[\S 13]{ST14MR3151094}.
This completes the proof.
\qed

%\newpage
%%%%%%%%%%%%%%%%%%%%%%%%%%%%%%%%%%%%
%%%%%%%%%%%%%%%%%%%%%%%%%%%%%%%%%%%%%

\begin{appendices}
\section[Jump Conditions with or without Surface Tension]{Jump Conditions with or without Surface Tension} \label{App:A}

We assume that the surface $\Sigma(t)$ is smooth with a well-defined unit normal $\bm{n}(t,x)$ and moves with the normal speed $\mathcal{V}(t,x)$ at point $x\in \Sigma(t)$ and time $t\geq 0$.
Let $\Omega^+(t)$ and $\Omega^-(t)$  denote the space domains occupied by the two conducting fluids at time $t$, respectively. Without loss of generality we assume that the unit normal $\bm{n}$ points into $\Omega^+(t)$.
Piecewise smooth weak solutions of the compressible MHD equations \eqref{MHD1}--\eqref{divH1} must satisfy
the following MHD Rankine--Hugoniot conditions on the surface of discontinuity $\Sigma(t)$ (see {\sc Landau--Lifshitz} \cite[\S 70]{LL84MR766230}):
\begin{subequations}
\label{RH1}
\begin{alignat}{3}
\label{RH1a}   & -\mathcal{V}[\rho]+\bm{n}\cdot[\rho v]=0,\\
\label{RH1b}    &-\mathcal{V}[\rho v]+\bm{n}\cdot[\rho v\otimes v-H\otimes H ]+\bm{n}[q]=0,\\
\label{RH1c}   & -\mathcal{V}[H]-\bm{n}\times [  v\times H]=0,\\
\label{RH1d}   & -\mathcal{V}\big[\rho E+\tfrac{1}{2}|H|^2\big]+\bm{n}\cdot [  v(\rho E+p)+ H\times (v\times H)]=0,\\
\label{RH1e}   & \,\bm{n}\cdot [H]=0,
\end{alignat}
\end{subequations}
where $[g]:=g^+-g^-$ denotes the jump in the quantity $g$ across $\Sigma(t)$
with
$$g^{\pm}(t,x):=\lim_{\epsilon\to 0^+}g(t,x\pm \epsilon \bm{n}(t,x))
\quad \textrm{for }x\in\Sigma(t).$$

The condition \eqref{RH1a} means that the mass transfer flux $\mathfrak{j}:=\rho (v\cdot \bm{n}-\mathcal{V})$ is continuous through $\Sigma(t)$.
We can rewrite \eqref{RH1} in terms of $\mathfrak{j}$ as
\begin{align}
\label{RH2}
\left\{
\begin{aligned}
&{[}{\mathfrak{j}}{]}=0,\quad
\mathfrak{j} [v_{\bm{n}}]+[q]=0,\quad
\mathfrak{j} [v_{\tau}]=H_{\bm{n}}[H_{\tau}],\quad [H_{\bm{n}}]=0,\\
&\mathfrak{j} \left[\tfrac{1}{\rho}H_{\tau}\right]=H_{\bm{n}}[v_{\tau}],\quad
		\mathfrak{j}\left[E+\tfrac{1}{2\rho}|H|^2\right]+[qv_{\bm{n}}-(H\cdot v)H_{\bm{n}}]=0,
\end{aligned}\right.
\end{align}
where $v_{\bm{n}}:=v\cdot \bm{n}$ (resp.~$H_{\bm{n}}:=H\cdot \bm{n}$) is the normal component of $v$ (resp.~$H$) and
$v_{\tau}$ (resp.~$H_{\tau}$) is the tangential part of $v$ (resp.~$H$).
If there is no flow across the discontinuity, that is, $\mathfrak{j}=0$ on $\Sigma(t)$,
then compressible MHD permits two distinct types of characteristic discontinuities \cite[\S 71]{LL84MR766230}:
tangential discontinuities ($H_{\bm{n}}|_{\Sigma(t)}= 0$)
and contact discontinuities ($H_{\bm{n}}|_{\Sigma(t)}\neq 0$).
For tangential discontinuities (or called current-vortex sheets), the jump conditions \eqref{RH2} become
\begin{align}
\label{BC:tangential}
H^{\pm}\cdot {\bm{n}}=0,\quad
[q]=0,\quad
\mathcal{V}=v^+\cdot {\bm{n}}=v ^-\cdot {\bm{n}}\quad
\textrm{on }\Sigma(t).
\end{align}
Moreover, from \eqref{RH2}, we obtain the following boundary conditions for MHD contact discontinuities:
\begin{align}
\label{BC:contact}
H^{\pm}\cdot {\bm{n}}\neq 0,\quad
[p]=0,\quad
[v]=[H]=0,\quad
\mathcal{V}=v^+\cdot {\bm{n}}
\quad  \textrm{on }\Sigma(t).
\end{align}

With surface tension present on the interface $\Sigma(t)$, we must take into account the corresponding surface force produced, so that the conditions \eqref{RH1b} and \eqref{RH1d} have to be modified respectively into (see {\sc Delhaye} \cite{Delhaye1974} or
{\sc Ishii--Hibiki} \cite[Chapter 2]{IH11MR3203021})
\begin{align}
\nonumber %\label{RH1b'}
&-\mathcal{V}[\rho v]+\bm{n}\cdot[\rho v\otimes v-H\otimes H ]+\bm{n}[q]=\mathfrak{s}\mathcal{H}\bm{n},\\
\nonumber %\label{RH1d'}
& -\mathcal{V}\big[\rho E+\tfrac{1}{2}|H|^2\big]+\bm{n}\cdot [  v(\rho E+p)+ H\times (v\times H)]=\mathfrak{s}\mathcal{H}\mathcal{V},
\end{align}
where
$\mathfrak{s}>0$ denotes the constant coefficient of surface tension
and $\mathcal{H}$ twice the mean curvature of $\Sigma(t)$.
Hence, for any interface with surface tension,  the boundary conditions \eqref{RH2} should be replaced by
\begin{align}
\nonumber %	\label{RH2'}
	\left\{
	\begin{aligned}
		&{[}{\mathfrak{j}}{]}=0,\quad
		\mathfrak{j} [v_{\bm{n}}]+[q]=\mathfrak{s}\mathcal{H},\quad
		\mathfrak{j} [v_{\tau}]=H_{\bm{n}}[H_{\tau}],\quad [H_{\bm{n}}]=0,\\
		&\mathfrak{j}\left[\tfrac{1}{\rho}H_{\tau}\right]=H_{\bm{n}}[v_{\tau}],\quad
		\mathfrak{j}\left[E+\tfrac{1}{2\rho}|H|^2\right]+[qv_{\bm{n}}-(H\cdot v)H_{\bm{n}}]=\mathfrak{s}\mathcal{H}\mathcal{V}.
	\end{aligned}\right.
\end{align}
Considering $\mathfrak{j}=0$ on $\Sigma(t)$, we get two different possibilities of interfaces, viz.
\begin{itemize}
\item[(a)] {\bf current-vortex sheets with surface tension}, for which the boundary conditions read
\begin{align}
	\label{BC:tangential'}
	H^{\pm}\cdot {\bm{n}}=0,\quad
	[q]=\mathfrak{s}\mathcal{H},\quad
	\mathcal{V}=v^+\cdot {\bm{n}}=v ^-\cdot {\bm{n}}\quad
	\textrm{on }\Sigma(t).
\end{align}
%For flows without magnetic fields, this kind of interfaces becomes compressible vortex sheets with surface tension studied in \cite{S16MR3524197}.
\item[(b)] {\bf MHD contact discontinuities with surface tension}, for which the boundary conditions read
\begin{align}
	\label{BC:contact'}
	H^{\pm}\cdot {\bm{n}}\neq 0,\quad
	[p]=\mathfrak{s}\mathcal{H},\quad
	[v]=[H]=0,\quad
	\mathcal{V}=v^+\cdot {\bm{n}}
	\quad  \textrm{on }\Sigma(t).
\end{align}
\end{itemize}

\end{appendices}

\bigskip

\section*{Acknowledgements}

The authors would like to thank the anonymous referees for comments and suggestions that improved the quality of the paper.

\section*{Conflict of interest}

The authors declare that they have no conflict of interest.

%\noindent{\it Acknowledgements.} \   The authors would like to thank the anonymous referees for helpful comments and suggestions to improve the quality of redaction.

\bigskip
%\noindent{\bf Compliance with Ethical Standards}

%\vspace*{2mm}
%\noindent {\bf Conflict of interest} \  The authors declare that they have no conflict of interest.

% \newpage
%%%%%----------
%%%%%---------- References

%\bibliographystyle{mpbasic}       % basic style, author-year citations
%\bibliographystyle{mpmpsci}      % mathematics and physical sciences
%\bibliographystyle{mpphys}        % APS-like style for physics
%\bibliographystyle{siamplain}      % or abbrvnat for \citet{}
%\bibliography{references}            % name your BibTeX data base

{\footnotesize
\begin{thebibliography}{99}\setlength{\itemsep}{0.35mm}
	
	
\bibitem{A89MR976971} Alinhac, S.: 
\newblock{Existence d'ondes de rar{\'e}faction pour des syst{\`e}mes quasi-lin{\'e}aires hyperboliques multidimensionnels.}
\newblock{{\it Commun. Partial Differ. Eqs.} {\bf 14}(2), 173--230 (1989).}
\newblock{\url{http://doi.org/10.1080/03605308908820595}}

\bibitem{AG07MR2304160} Alinhac, S., G{{\'e}}rard, P.:  
\newblock{\it Pseudo-differential Operators and the {N}ash--{M}oser Theorem.} 
\newblock{American Mathematical Society, Providence (2007).}
\newblock{\url{http://doi.org/10.1090/gsm/082}}

\bibitem{AM07MR2334849} Ambrose, D.M., Masmoudi, N.: 
\newblock{Well-posedness of 3D vortex sheets with surface tension.}
\newblock{{\it Commun. Math. Sci.} {\bf 5}(2), 391--430 (2007).}
\newblock{\url{https://doi.org/10.4310/CMS.2007.v5.n2.a9}}

\bibitem{C61MR0128226} Chandrasekhar, S.: 
\newblock{\it Hydrodynamic and Hydromagnetic Stability.} 
\newblock{Clarendon Press, Oxford (1961)}

\bibitem{CP82MR0678605}  Chazarain, J., Piriou, A.: 
\newblock{\it Introduction to the Theory of Linear Partial Differential Equations.} 
\newblock{North-Holland Publishing Co., Amsterdam (1982).}
\newblock{\url{https://www.sciencedirect.com/bookseries/studies-in-mathematics-and-its-applications/vol/14/suppl/C}}

\bibitem{CW08MR2372810}  Chen, G.-Q., Wang, Y.-G.: 
\newblock{Existence and stability of compressible current-vortex sheets in three-dimensional magnetohydrodynamics.}
\newblock{{\it Arch. Ration. Mech. Anal.} {\bf 187}(3), 369--408 (2008).}
\newblock{\url{https://doi.org/10.1007/s00205-007-0070-8}}

\bibitem{CW12MR3289359}  Chen, G.-Q., Wang, Y.-G.: 
\newblock{Characteristic discontinuities and free boundary problems for hyperbolic conservation laws.}
\newblock{In: Holden, H., Karlsen, K.H. (eds.)
{\it Nonlinear Partial Differential Equations}, pp.~53--81. Springer, Heidelberg (2012).}
\newblock{\url{http://dx.doi.org/10.1007/978-3-642-25361-4_4}}

\bibitem{CSW19MR3925528}  Chen, G.-Q., Secchi, P., Wang, T.: 
\newblock{Nonlinear stability of relativistic vortex sheets in three-dimensional Minkowski spacetime.}
\newblock{{\it Arch. Ration. Mech. Anal.} {\bf 232}(2), 591--695 (2019).}
\newblock{\url{https://doi.org/10.1007/s00205-018-1330-5}}

\bibitem{CSW20MR4110436}  Chen, G.-Q., Secchi, P., Wang, T.: 
\newblock{Stability of multidimensional thermoelastic contact discontinuities.}
\newblock{{\it Arch. Ration. Mech. Anal.} {\bf 237}(3), 1271--1323 (2020).}
\newblock{\url{https://doi.org/10.1007/s00205-020-01531-5}}

\bibitem{C10MR2827870}
Chen, S.: 
\newblock{Study of multidimensional systems of conservation laws: problems, difficulties and progress.}
\newblock{In: Bhatia, R., Pal, A., Rangarajan, G., Srinivas, V., Vanninathan, M.  (eds.) 
{\it Proceedings of the International Congress of Mathematicians, Vol.~III}, pp.~1884--1900. Hindustan Book Agency, New Delhi (2010).}
\newblock{\url{https://doi.org/10.1142/9789814324359_0126}}

\bibitem{CCS08MR2456184} Cheng, C.-H., Coutand, D., Shkoller, S.: 
\newblock{On the motion of vortex sheets with surface tension in three-dimensional Euler equations with vorticity.}
\newblock{{\it Commun. Pure Appl. Math.} {\bf 61}(12), 1715--1752 (2008).}
\newblock{\url{http://dx.doi.org/10.1002/cpa.20240}}

\bibitem{CS08MR2423311} Coulombel, J.-F., Secchi, P.: 
\newblock{Nonlinear compressible vortex sheets in two space dimensions.}
\newblock{{\it Ann. Sci. \'{E}c. Norm. Sup\'{e}r. (4)}}
\newblock{{\bf 41}(1), 85--139 (2008).}
\newblock{\url{https://doi.org/10.24033/asens.2064}}

\bibitem{Delhaye1974} Delhaye, J.M.: 
\newblock{Jump conditions and entropy sources in two-phase systems. Local instant formulation.}
\newblock{{\it Int. J. Multiphase Flow} {\bf 1}(3), 395--409 (1974).}
\newblock{\url{https://doi.org/10.1016/0301-9322(74)90012-3}}

\bibitem{FM63MR0154509} Fejer, J.A., Miles, J.W.: 
\newblock{On the stability of a plane vortex sheet with respect to three-dimensional disturbances.} 
\newblock{{\it J. Fluid Mech.}  {\bf 15}, 335--336 (1963).}
\newblock{\url{https://doi.org/10.1017/S002211206300029X}}

\bibitem{H76MR0602181} H\"{o}rmander, L.: 
\newblock{The boundary problems of physical geodesy.}
\newblock{{\it Arch. Ration. Mech. Anal.} {\bf 62}(1), 1--52 (1976).}
\newblock{\url{https://doi.org/10.1007/BF00251855}}

\bibitem{IH11MR3203021} Ishii, M., Hibiki, T.: 
\newblock{{\it Thermo-Fluid Dynamics of Two-Phase Flow,}} 
\newblock{2nd edition. Springer, New York (2011).}
\newblock{\url{https://doi.org/10.1007/978-1-4419-7985-8}}

\bibitem{L11zbMATH05278418} Lautrup, B.:
\newblock{{\it Physics of Continuous Matter. Exotic and Everyday Phenomena in the Macroscopic World,}} 
\newblock{2nd edition. CRC Press, Boca Raton (2011).}
\newblock{\url{https://doi.org/10.1201/9781439894200}}


\bibitem{LL84MR766230}  Landau, L.D., Lifshitz, E.M.: 
\newblock{{\it Electrodynamics of Continuous Media.}}
\newblock{2nd edition, Pergamon Press, Oxford (1984).}
\newblock{\url{https://www.sciencedirect.com/book/9780080302751/electrodynamics-of-continuous-media}}

\bibitem{LM72MR0350178}  Lions, J.-L., Magenes, E.: 
\newblock{{\it Non-homogeneous Boundary Value Problems and Applications, Vol. II.}}
\newblock{Springer, New York (1972).}
\newblock{\url{https://doi.org/10.1007/978-3-642-65217-2}}

\bibitem{M01MR1842775}  M{{\'e}}tivier, G.: 
\newblock{Stability of multidimensional shocks.}
\newblock{In: Freist{\"u}hler, H., Szepessy, A. (eds.) 
{\it Advances in the Theory of Shock Waves}, pp.~25--103. Birkh{\"a}user, Boston (2001).}
\newblock{\url{https://doi.org/10.1007/978-1-4612-0193-9_2}}

\bibitem{M00MR1781515} Mishkov, R.L.: 
\newblock{Generalization of the formula of {F}aa di {B}runo for a composite function with a vector argument.} 
\newblock{{\it Int. J. Math. Math. Sci.}  {\bf 24}, 481--491 (2000).}
\newblock{\url{http://dx.doi.org/10.1155/S0161171200002970}}


\bibitem{MTT15MR3306348} Morando, A., Trakhinin, Y., Trebeschi, P.: 
\newblock{Well-posedness of the linearized problem for {MHD} contact discontinuities.}
\newblock{{\it J. Differ. Eqs.} {\bf 258}(7), 2531--2571 (2015).}
\newblock{\url{http://dx.doi.org/10.1016/j.jde.2014.12.018}}


\bibitem{MTT18MR3766987}  Morando, A., Trakhinin, Y., Trebeschi, P.: 
\newblock{Local existence of MHD contact discontinuities.}
\newblock{{\it Arch. Ration. Mech. Anal.} {\bf 228}(2), 691--742 (2018).}
\newblock{\url{https://doi.org/10.1007/s00205-017-1203-3}}


\bibitem{R85MR0797053} Rauch, J.: 
\newblock{Symmetric positive systems with boundary characteristic of constant multiplicity.} 
\newblock{{\it Trans. Amer. Math. Soc.} {\bf 291}(1), 167--187 (1985).}
\newblock{\url{https://doi.org/10.1090/S0002-9947-1985-0797053-4}}


\bibitem{SDGX07MR2356385} Samulyak, R., Du, J., Glimm, J., Xu, Z.: 
\newblock{A numerical algorithm for MHD of free surface flows at low magnetic Reynolds numbers.}
\newblock{{\it J. Comput. Phys.} {\bf 226}, 1532--1549 (2007).}
\newblock{\url{https://doi.org/10.1016/j.jcp.2007.06.005}}

\bibitem{S96MR1405665} Secchi, P.: 
\newblock{Well-posedness of characteristic symmetric hyperbolic systems.}
\newblock{{\it Arch. Ration. Mech. Anal.}  {\bf 134}, 155--197 (1996).}
\newblock{\url{http://dx.doi.org/10.1007/BF00379552}}

\bibitem{S16MR3524197}  Secchi, P.: 
\newblock{On the Nash-Moser iteration technique.}
\newblock{In: Amann, H., Giga, Y., Kozono, H., Okamoto, H., Yamazaki, M. (eds.) 
{\it Recent Developments of Mathematical Fluid Mechanics}, pp.\;443--457, Birkh{\"a}user, Basel (2016).}
\newblock{\url{https://doi.org/10.1007/978-3-0348-0939-9_23}}

\bibitem{ST14MR3151094}  Secchi, P., Trakhinin, Y.: 
\newblock{Well-posedness of the plasma-vacuum interface problem.}
\newblock{{\it Nonlinearity} {\bf 27}(1), 105--169 (2014).}
\newblock{\url{https://doi.org/10.1088/0951-7715/27/1/105}}

\bibitem{SZ08MR2400608} Shatah, J., Zeng, C.: 
\newblock{A priori estimates for fluid interface problems.}
\newblock{{\it Commun. Pure Appl. Math.} {\bf 61} (6), 848--876 (2008).}
\newblock{\url{https://doi.org/10.1002/cpa.20241}}

\bibitem{SZ11MR2763036} Shatah, J., Zeng, C.: 
\newblock{Local well-posedness for fluid interface problems.}
\newblock{{\it Arch. Ration. Mech. Anal.} {\bf 199}(2), 653--705 (2011).}
\newblock{\url{http://dx.doi.org/10.1007/s00205-010-0335-5}}


\bibitem{Stevens16MR3544315} Stevens, B.: 
\newblock{Short-time structural stability of compressible vortex sheets with surface tension.}
\newblock{{\it Arch. Ration. Mech. Anal.} {\bf  222}(2), 603--730 (2016).}
\newblock{\url{http://dx.doi.org/10.1007/s00205-016-1009-8}}

\bibitem{Syr54MR0065343} Syrovatski\u{\i}, S.I.: 
\newblock{Instability of tangential discontinuities in a compressible medium. (in Russian)}
\newblock{{\it Akad. Nauk SSSR. \v{Z}urnal Eksper. Teoret. Fiz.} {\bf 27}, 121--123 (1954)}


\bibitem{T05MR2187618}  Trakhinin, Y.: 
\newblock{Existence of compressible current-vortex sheets: variable coefficients linear analysis.}
\newblock{{\it Arch. Ration. Mech. Anal.} {\bf 177}(3), 331--366 (2005).}
\newblock{\url{https://doi.org/10.1007/s00205-005-0364-7}}

\bibitem{T09ARMAMR2481071}   Trakhinin, Y.: 
\newblock{The existence of current-vortex sheets in ideal compressible magnetohydrodynamics.} 
\newblock{{\it Arch. Ration. Mech. Anal.} {\bf 191}(2), 245--310 (2009).}
\newblock{\url{https://doi.org/10.1007/s00205-008-0124-6}}
	
\bibitem{T09CPAMMR2560044}   Trakhinin, Y.: 
\newblock{Local existence for the free boundary problem for nonrelativistic and relativistic compressible Euler equations with a vacuum boundary condition.} 
\newblock{{\it Commun. Pure Appl. Math.} {\bf 62}(11), 1551--1594 (2009).}
\newblock{\url{https://doi.org/10.1002/cpa.20282}}


\bibitem{TW21MR4201624} Trakhinin, Y., Wang, T.:
\newblock{Well-posedness of free boundary problem in non-relativistic and relativistic ideal compressible magnetohydrodynamics.}
\newblock{{\it Arch. Ration. Mech. Anal.} {\bf 239}(2), 1131--1176 (2021).}
\newblock{\url{https://doi.org/10.1007/s00205-020-01592-6}}

\bibitem{TW21b} Trakhinin, Y., Wang, T.:  
\newblock{Well-posedness for the free-boundary ideal compressible magnetohydrodynamic equations with surface tension.}
\newblock{{\it Math. Ann.} (2021).}
\newblock{\url{https://doi.org/10.1007/s00208-021-02180-z}}

\end{thebibliography}  }


\end{document}